\documentclass[10pt]{amsart}

	\usepackage{geometry}                % See geometry.pdf to learn the layout options. There are lots.
	\usepackage{amsmath}
	\usepackage{amssymb}
	\usepackage{amsthm}	% need this for theorem scheme below

	\usepackage[all]{xy}

	\usepackage{color}

	\usepackage{hyperref}

	\newcommand{\Number}{section} 	
	\theoremstyle{plain}
		\newtheorem{thm}{Theorem}[\Number]
		\newtheorem{lem}[thm]{Lemma}
		
		\newtheorem{prp}[thm]{Proposition}
		
		\newtheorem{cor}[thm]{Corollary}
		
		\newtheorem{claim}[thm]{Claim}

		\newtheorem{dfn}[thm]{Definition}
		\newtheorem{defn}[thm]{Definition}
		\newtheorem{ex}[thm]{Example}
	\providecommand{\Frac}[2]{\displaystyle\frac{#1}{#2}}
	\providecommand{\norm}[1]{\lVert#1\rVert}

	\providecommand{\removeme}[1][default]{}

	\providecommand{\C}{\mathbb{C}}

	\providecommand{\N}{\mathbb{N}}

	\providecommand{\Q}{\mathbb{Q}}
	\providecommand{\R}{\mathbb{R}}
	
	\providecommand{\T}{\mathbb{T}}

	\providecommand{\Z}{\mathbb{Z}}

	\providecommand{\CalC}{\mathcal{C}}

	\providecommand{\CalF}{\mathcal{F}}
	\providecommand{\CalG}{\mathcal{G}}

	\providecommand{\CalP}{\mathcal{P}}
	
	\providecommand{\CalR}{\mathcal{R}}

	\providecommand{\claim}{\textbf{Claim: \,}}

	\providecommand{\coker}{\mbox{coker}}

	\providecommand{\onemm}{\vspace{1mm}}
	\providecommand{\twomm}{\vspace{2mm}}
	
	\providecommand{\onecm}{\vspace{1cm}}

	\newcommand{\af}{\alpha}
	\newcommand{\bt}{\beta}

	\newcommand{\ep}{\varepsilon}

	\newcommand{\ph}{\varphi}

	\newcommand{\id}{{\mathrm{id}}}

	\newcommand{\diam}{{\mathrm{diam}}}
	\newcommand{\dist}{{\mathrm{dist}}}

	\newcommand{\diag}{{\mathrm{diag}}}

	\newcommand{\ca}{C*-algebra}

\begin{document}

%------------------------------ Title ------------------------------------%

\title[Crossed Product C*-algebras of Minimal Dynamical Systems]{Crossed Product C*-algebras of Minimal Dynamical Systems on the Product of the Cantor Set and the Torus}

\author{%
	Wei Sun \\
	School of Mathematical Sciences \\
	University of Nottingham
}

%\date{Feb 10, 2010}

%-------------------------------- Abstract --------------------------------%

\begin{abstract}

	This paper studies the relationship between minimal dynamical systems on the product of the Cantor set ($X$) and torus ($\T^2$) and their corresponding crossed product $C^*$-algebras.
	
	For the case when the cocycles are rotations, we studied the structure of the crossed product $C^*$-algebra $A$ by looking at a large subalgebra $A_x$. It is proved that, as long as the cocycles are rotations, the tracial rank of the crossed product $C^*$-algebra is always no more than one, which then indicates that it falls into the category of classifiable $C^*$-algebras.
	
	If a certain rigidity condition is satisfied, it is shown that the crossed product $C^*$-algebra has tracial rank zero. Under this assumption, it is proved that for two such dynamical systems, if $A$ and $B$ are the corresponding crossed product $C^*$-algebras, and we have an isomorphism between $K_i(A)$ and $K_i(B)$ which maps $K_i(C(X \times \T^2))$ to $K_i(C(X \times \T^2))$, then these two dynamical systems are approximately $K$-conjugate. The proof also indicates that $C^*$-strongly flip conjugacy implies approximate $K$-conjugacy in this case.
	
\end{abstract}

%-------- Abstract should precede \maketitle in AMS document class ---------%

\maketitle

% ------------------------ include the table of contents ------------------------

%\singlespacing
	
    %\titlepage

    %\addtocontents{toc}{\bfseries \underline{Chapter} \hfill \underline{Page}\par} % add underlined "Chapter" and "page" before actually contents show up.

    %\renewcommand{\@tocrmarg}{3em}
    %\renewcommand{\@tocrmarg}{4em}

    %\providecommand{\tocrmarg}{9em}

%Shall we add TOC?
    %\tableofcontents

%	\listoffigures 	% include this if you have figures
%	\listoftables 	% include this if you have tables

\setcounter{section}{-1}

% ----------------------- formatting for body text ------------------------------

	%\doublespacing
	%\setlength{\parindent}{.5in}

	%\pagenumbering{arabic}	% body pages: restart numbering in arabic
	%\setcounter{page}{1}		% re-set page numbering

\section{Introduction and Notation} \label{Sec Introduction and Notation}

	In this section, an introduction of the background is given, and the notations used in this paper are also introduced.

\vspace{2mm}
	
	Let $X$ be a compact infinite metric space, and let $\af \in \text{Homeo}(X)$ be a minimal homeomorphism of $X$. We can construct the crossed product \ca \ from the minimal dynamical system $(X, \af)$, denoted by $C^*(\Z, X, \af)$.

	One interesting question is how properties of the dynamical system $(X, \af)$ determine properties of the crossed product \ca, and how properties of the crossed product \ca s shed some light on properties of the dynamical system $(X, \af)$.

	For minimal Cantor dynamical systems, Giodano, Putnam and Skau found (in \cite{GPS}) that for two minimal Cantor dynamical systems, the corresponding crossed product \ca s are isomorphic if and only if the minimal Cantor dynamical systems are strongly orbit equivalent.

	Lin and Matui studied this problem when the base space is the product of the Cantor set and the circle (see \cite{LM1}, \cite{LM2}), and they discovered that in the rigid cases (see Definition 3.1 of \cite{LM1}), for two crossed product \ca s to be isomorphic, the dynamical systems must be approximately K-conjugate (a ``strengthened" version of weak approximate conjugacy, in the sense that it is compatible with the K-data).

	We studied minimal dynamical systems on the product of the Cantor set and the torus. For the case that the cocycles take values in the rotation group, similar results are found for the relationship between \ca \ isomorphisms and approximate K-conjugacy between two dynamical systems. It is also shown that the tracial rank of the crossed product \ca \ is no more than one.

	For the case that the cocycles are Furstenberg transformations, a necessary condition for weak approximate conjugacy between two minimal dynamical systems (via conjugacy maps whose cocycles are Furstenberg transformations) is given.
	
	In section \ref{The subalgebra Ax}, structure of the subaglebra $A_x$ is studied. In section \ref{The crossed product C-algebra A}, we studied the crossed product \ca \ and concluded that its tracial rank is always no more than one. In section \ref{Examples with cocycles being rotations}, a concrete example of minimal dynamical system of the type $(X \times \T \times \T, \af \times \mbox{R}_{\xi} \times \mbox{R}_{\eta})$ whose crossed product \ca \ has tracial rank one is given. In section \ref{Approximate conjugacies}, we give an if and only if condition for when two such rigid (as defined in Definition \ref{definition of rigidity}) minimal dynamical systems are approximately $K$-conjugate.

\vspace{1cm}

	Some notations used in this paper are listed below.

	Let $(X, \af)$ be a minimal dynamical system, by $\af$-invariant probability measure $\mu$, we mean such a probability measure $\mu$ on $X$ satisfying $\mu(D) = \mu(\af(D))$ for every $\mu$-measurable subset $D$. Following the Markov-Kakutani fixed point Theorem, it is shown that the set of $\af$-invariant probability measures is not empty (see Lemma 1.9.18 and Theorem 1.9.19 of \cite{Lin2} for details).

	Let $\mu$ be a measure on $X$. For $f \in C(X)$, we use $\mu(f)$ to denote $\int_X f(x) \ \mathrm{d} \mu$.

	For a minimal dynamical system $(X, \af)$ we use $C^*(\Z, X, \af)$ to denote $C(X) \times_{\af} \Z$, the crossed product \ca \ of the dynamical system $(X, \af)$.

	In a topological space $X$, we say a subset $D$ is clopen, if $D$ is both closed and open.
	
	In section \ref{The subalgebra Ax} to \ref{Approximate conjugacies}, unless otherwise specified, $X$ denotes the Cantor set, $\T$ denotes the circle, and $\T^2$ denotes the two-dimensional torus.

%\begin{sloppypar}
	For a compact Hausdorff space $Y$, $\mathrm{Homeo}(Y)$ is used to denote the set of all the homeomorphisms of $Y $.
%\end{sloppypar}

	As the Cantor set $X$ is totally disconnected, we can write a homeomorphism of $X \times \T^2$ as $\af \times \varphi$ (the skew product form), with $\af \in \text{Homeo}(X)$ and $\varphi \colon X \rightarrow \text{Homeo}(\T^2)$ being continuous, and
	$$ \af \times \varphi \colon X \times \T^2 \rightarrow X \times \T^2 \ \text{defined by} \  ( x, t_1, t_2 ) \mapsto ( \af(x),  \varphi(x)( t_1, t_2 ) ) . $$

	For the case that the cocycles take values in rotation groups, we can further express $\af \times \varphi$ as $(X \times \T \times \T, \af \times \mathrm{R}_{\xi} \times \mathrm{R}_{\eta})$, with $\xi, \eta \colon X \rightarrow \T$ continuous, and
	$$ \af \times \mathrm{R}_{\xi} \times \mathrm{R}_{\eta} \colon X \times \T^2 \rightarrow X \times \T^2 \ \text{defined by} \ ( x, t_1, t_2 ) \mapsto ( \varphi(x), t_1 + \xi(x), t_2 + \eta(x) ) . $$

	We use $A$ to denote the corresponding crossed product \ca. For $x \in X$, the subalgebra $A_x$ is defined as below.

\begin{dfn} \label{dfn of Ax}

	For a minimal dynamical system $( X \times \T \times \T, \af \times \mathrm{R}_{\xi} \times \mathrm{R}_{\eta} )$, $A_x$ is defined to be the subalgebra of the crossed product \ca \ generated by $C(X \times \T \times \T)$ and $u \cdot C_0( (X \backslash \{ x \}) \times \T \times \T)$, with $u$ being the implementing unitary in $A$ satisfying $u^* f u = f \circ (\af \times \mathrm{R}_{\xi} \times \mathrm{R}_{\eta})^{-1}$.

\end{dfn}	

\noindent \textbf{Remark:} The idea to define such a sub-algebra in the crossed product can be traced to Putnam's work (see \cite{Putnam}). From the definition, if $D$ is a clopen subset of the Cantor set $X$, and $1_{D \times \T^2}$ is the characteristic function of $D \times \T^2$, then $u 1_{D \times \T^2} u^* = 1_{D \times \T^2} \circ (\af \times \mathrm{R}_{\xi} \times \mathrm{R}_{\eta}) = 1_{\af^{-1}(D) \times \T^2}$.

\vspace{1mm}

	Let $\{ \CalP_n \colon n \in \N \}$ be as in the Bratteli-Vershik model of the minimal Cantor dynamical system $(X, \af)$ (see \cite[Theorem 4.2]{HPS}), and let $Y_n$ be the roof of $\CalP_n$ (denoted as $R(\CalP_n)$). Then $\{ Y_n \}$ will be a decreasing sequence of clopen sets such that $\bigcap_{n = 1}^{\infty} Y_n = \{ x \}$. Use $A_n$ to denote the subalgebra generated by $C(X \times \T \times \T)$ and $u \cdot C_0( (X \backslash Y_n) \times \T \times \T)$.

	In a \ca \ $A$, for $a, b \in A$, $a \approx_{\ep} b$ just means $\norm{a - b} \leq \ep$. By $a \approx_{\ep_1} b \approx_{\ep_2} c$, we mean $\norm{a - b} \leq \ep_1$ and $\norm{b - c} \leq \ep_2$. It is clear that $a \approx_{\ep_1} b \approx_{\ep_2} c$ implies $a \approx_{\ep_1 + \ep_2} c$.

	In a \ca \ $A$, $[a, b]$ (the commutator) is defined to be $a b - b a$.

	For a \ca \ $A$ we use $T(A)$ to denote the convex set of all the tracial states on $A$, and $\mathrm{Aff}(T(A))$ to denote all the affine linear functions from $T(A)$ to $\R$.

	In a \ca \ $A$, for $a \in A_{+}$, we use $\mathrm{Her}(a)$ to denote the smallest hereditary subaglebra that contains $a$.

	For a \ca \ $A$, we use $\mathrm{TR}(A)$ to denote the tracial rank of $A$ (see \cite[Definition 3.6.2]{Lin6}). We use $\mathrm{RR}(A)$ to denote the real rank of $A$ (\cite[Definition 3.1.6]{Lin6}) and $\mathrm{tsr}(A)$ to denote the stable rank of $A$ (\cite[Definition 3.1.1]{Lin6}).

\begin{dfn} \label{projection M-vN equivalent to a projection in hereditary subalgebra of a positive element}

	Let $A$ be a \ca. Let $p$ be a projection of $A$ and let $a \in A_+$. We say that $p \preceq a$ if $p$ is Murray-von Neumann equivalent to a projection $q \in \mathrm{Her}(a)$.

\end{dfn}

	Let $A$ be a \ca. We use $U(A)$ to denote the group of all the unitary elements in $A$. We use $CU(A)$ to denote the norm closure of the group generated by the commutators of $U(A)$. In other words, $CU(A)$ is the norm closure of the group generated by elements in $\{ u v u^* v^* \colon u, v \in U(A) \}$. One can check that $CU(A)$ is a normal subgroup of $U(A)$ and $U(A) / CU(A)$ is an abelian group.

\vspace{2mm}

\begin{dfn} \label{defn of sharp morphism}
	Let $\ph: A \longrightarrow B$ be a \ca \ homomorphism. We define
	$$ \ph^{\sharp} : U(A) / CU(A) \longrightarrow U(B) / CU(B) $$
	to be the map induced by $\ph$ which maps $[u] \in U(A) / CU(A)$ to $[\ph(u)] \in U(B) / CU(B)$.
\end{dfn}

\vspace{2cm}

\section{The subalgebra \texorpdfstring{$A_x$}{Ax}} \label{The subalgebra Ax}

	In this section, we study properties of a ``large" subalgebra of $A$, namely $A_x$. The idea of the construction of $A_x$ was first given by Putnam, but the construction here is a bit different from that in the sense that we are removing one fiber $\{ x \} \times \T \times \T$ instead of one point. In other words, we define $A_x$ to be the subalgebra generated by $C(X \times \T \times \T)$ and $u \cdot C_0( (X \backslash \{ x \}) \times \T \times \T)$, with $u$ being the implementing unitary in $A$ (as defined in Section \ref{Sec Introduction and Notation}).

\vspace{2mm}

	The following lemma gives the basic structure of $A_x$, which is used to study the structure of $A$.

\begin{lem} \label{AT2 structure of A_x}

	If $(X \times \T \times \T, \af \times \mathrm{R}_{\xi} \times \mathrm{R}_{\eta})$ is minimal, then for any $x \in X$ there are $k_1, k_2, \ldots \in \N$ and $d_{s, n} \in \N$ for $n \in \N$ such that $A_x \cong \displaystyle \varinjlim_{n} \bigoplus_{s = 1}^{k_n} M_{d_{s, n}}(C(\T^2))$.

\end{lem}

\begin{proof}
	As $\af \times \mathrm{R}_{\xi} \times \mathrm{R}_{\eta}$ is minimal, it follows that $(X, \af)$ is also minimal. For $x \in X$, let $\CalP = \{ X(n, v, k) \colon v \in V_n, k = 1, 2, \ldots, h_n(v) \}$ be as in the Bratteli-Vershik model (\cite[Theorem 4.2]{HPS}) for $(X, \af)$. Let $R(\CalP_n)$ be the roof set of $\CalP_n$, defined by $R(\CalP_n) = \bigcup_{v \in V_n} X(n, v, h_n(v))$. We can assume that the roof sets satisfy
	$$ \bigcap_{n \in \N} R(\CalP_n) = \{ x \}. $$
	
	Let $A_n$ be the subalgebra of the crossed product \ca\ $A$ such that $A_n$ is generated by $C(X \times \T \times \T)$ and $u \cdot C_0( ( X \backslash R(\CalP_n) ) \times \T \times \T)$, with $u$ being the implementing unitary element satisfying $u f u^* = f \circ (\af \times \mathrm{R}_{\xi} \times \mathrm{R}_{\eta})$ for all $f \in C(X \times \T \times \T)$. Then it is clear that $A_1 \subset A_2 \subset \cdots$. As we can approximate $f \in C_0( ( X \backslash \{ x \} ) \times \T \times \T)$ with
	$$ f_n \in C_0( ( X \backslash R(\CalP_n) ) \times \T \times \T) = C( ( X \backslash R(\CalP_n) ) \times \T \times \T) , $$
	we have $\varinjlim (A_n, \phi_n) = A_x$ with $\phi_n \colon A_n \rightarrow A_{n + 1}$ being the canonical embedding.

	For $C(X \backslash R(\CalP_n) \times \T \times \T)$, it is clear that we have
		$$ C( ( X \backslash R(\CalP_n) ) \times \T \times \T) \cong \bigoplus_{v \in V_n} \bigoplus_{1 \leq k \leq h_n(v) - 1} C \left( X(n, v, k) \times \T^2 \right) . $$

	We will show that $A_n \cong \bigoplus_{v \in V_n} M_{h_n(v)}(C(X(n, v, 1)) \otimes C(\T^2))$.

	Let $e_{i, j}^v = 1_{X(n, v, i)} \cdot u^{i - j}$. Then $e_{i, j}^v \cdot e_{i', j'}^{v'} = 0$ if $v \neq v'$.
	Note that
	\begin{align*}
		e_{i, j}^v \cdot e_{k, s}^v & = 1_{X(n, v, i)} \cdot u^{i - j} \cdot 1_{X(n, v, k)} \cdot u^{k - s} \\
			& = 1_{X(n, v, i)} \cdot 1_{X(n, v, k + i - j)} \cdot u^{i - j + k - s} \\
			& = \delta_{k, j} \cdot e_{i, s}^v .
	\end{align*}
	In other words, $\{ e_{i, j}^v \}_{i, j = 1}^{h(v)}$ is a system of matrix units.

	As $A_n$ is generated by
	$$ \displaystyle\{ e ^v _{i, j} \otimes C \left( X(n, v, 0) \otimes C(T^2) \right) \colon v \in V_n, 1 \leq i, j \leq h(v) \} , $$
	it follows that
	$$ A_n \cong \bigoplus_{v \in V_n} M_{h_n(v)} \left( C(X(n, v, 1)) \otimes C(\T^2) \right) . $$

	Let $B_n = \bigoplus_{v \in V_n} M_{h_n(v)}(\C \otimes C(\T^2))$. Then it is clear that $B_n$ can be regarded as a subalgebra of $A_n$.

	As for the canonical embedding $\phi_{n, n + 1} \colon A_n \rightarrow A_{n + 1}$, consider
	$$ a \in A_n \cong \bigoplus_{v \in V_n} M_{h_n(v)}(C(X(n, v, 1)) \otimes C(\T^2)) $$
	such that $a = (f \otimes g) \cdot u^{i - j} \in e_{i, j}^v \otimes C( X(n, v, 1) \otimes C(\T^2) )$, with $f \in C(X(n, v, i)) \cong C(X(n, v, 1))$ and $g \in C(\T^2)$.
	
	Note that the Kakutani-Rokhlin partition of $A_{n + 1}$ is finer than that of $A_n$. We can write
	$$ f = \sum_{X(n + 1, v_s, k ) \subset X(n, v, i)} f_{s, k} \ \ \text{with} \ f_{s, k} \in C( X( n + 1, v_s, k ) ) . $$
	It follows that
	$$ \phi_{n, n + 1} (f \otimes g) = \sum_{X(n + 1, v_s, k ) \subset X(n, v, i)} f_{s, k} \otimes g . $$
	Then we have
	\begin{align*}
		\phi_{n, n + 1}(a) & = \left( \sum_{X(n + 1, v_s, k ) \subset X(n, v, i)} f_{s, k} \otimes g \right) \cdot u^{i - j} \\
			& = \sum_{X(n + 1, v_s, k ) \subset X(n, v, i)} \left( f_{s, k} \otimes g \right) \cdot u^{i - j} ,
	\end{align*}
	with $\sum_{X(n + 1, v_s, k ) \subset X(n, v, i)} \left( f_{s, k} \otimes g \right) \cdot u^{i - j}$ being an element in $A_{n + 1}$. It is then clear that \\ $\phi_{n, n + 1}(B_n) \subset B_{n + 1}$ if we regard $B_n$ as a subalgebra of $A_n$ and $B_{n + 1}$ as a subalgebra of $A_{n + 1}$.

	Just abuse notation and use $\phi_{n, n + 1}$ to denote the canonical embedding from $B_n$ to $B_{n + 1}$. Then we have the following commutative diagram:
	$$
	\xymatrix{
 	\cdots \ar@{->}[r] & B_n \ar@{->}[r]^{\displaystyle \phi_{n, n + 1}} \ar@{->}[d]^{\displaystyle j_n}
    	  & B_{n + 1} \ar@{->}[d]^{\displaystyle j_{n + 1}} \ar@{->}[r]^{\displaystyle \phi_{n + 1, n + 2}} & B_{n + 2} \ar@{->}[d]^{\displaystyle j_{n + 2}} \ar@{->}[r] & \cdots \\
	 \cdots \ar@{->}[r] & A_n \ar@{->}[r]_{\displaystyle \phi_{n, n + 1}} & A_{n + 1} \ar@{->}[r]_{\displaystyle \phi_{n + 1, n + 2}} & A_{n + 2} \ar@{->}[r] & \cdots
	 }
	$$

	For every $a \in A_x = \displaystyle \varinjlim_n (A_n, \phi_{n, n + 1})$ and every $\ep > 0$, there exists $a_n \in A_n$ such that $\norm{a - a_n} < \ep / 2$ if we identity $a_n$ with $\phi_{n, \infty}(a_n) \in A_x$. Without loss of generality, we can assume that
	$$ a_n = \sum_{k = 1}^L \sum_{v \in V_n} \sum_{i , j = 1}^{h_n(v)} \left( f_{k, v, i, j} \otimes g_{k, v, i, j} \right) \cdot e^{v}_{i, j} , $$
	with $f_{k, v, i, j} \in C(X(n, v, 0))$ and $g_{k, v, i, j} \in C(\T^2)$.

	Let $M = \max_{k, v, i, j} \{ \norm{g_{k, v, i, j}} \}$. For all $k, v, i, j$ as above, we can find $\delta > 0$ such that for $x, y \in X$, if $\dist(x, y) < \delta$, then
	$$\norm{f_{k, v, i, j}(x) - f_{k, v, i, j}(y)} < \frac{\ep}{2 \cdot M \cdot L \cdot |V_n| \cdot h_n(v)^2} . $$

	According to the Bratteli-Vershik model, $\bigcap_{n \in \N} R(\CalP_n) = \{ x \}$. We may further require that for all $n \in \N$, every block $X(n, v, k)$ in $\CalP_n$ satisfies $\diam( X(n, v, k) ) < 1/n$. Then we can choose $N \in \N$ such that $\diam(R(\CalP_N)) < \delta$. Without loss of generality, we can assume that $N \geq n$.

	In $\CalP_N$, for every $X(N, v, k)$, choose $w_{N, v, k} \in X(N, v, k)$. For $k = 1, \ldots, L$, $v \in V_n$,  $i, j = 1, \ldots, h_n(v)$, define
	$$ \widetilde{f_{k, v, i, j}} = \sum_{X(N, v', k') \subset X(n, v, k)} f_{k, v, i, j}(w_{N, v', k'}) \cdot 1_{X(N, v', k')} . $$
	According to our choice of $N$, it is clear that $\norm{f_{k, v, i, j} - \widetilde{f_{k, v, i, j}}} < \frac{\ep}{2 \cdot M \cdot L \cdot |V_n| \cdot h_n(v)^2}$.

	For the $a_n$ given above, define $\widetilde{a_n} \in A_n$ by
	$$\widetilde{a_n} = \sum_{k = 1}^L \sum_{v \in V_n} \sum_{i , j = 1}^{h_n(v)} \left( \widetilde{f_{k, v, i, j}} \otimes g_{k, v, i, j} \right) \cdot e^{v}_{i, j} . $$
	As
	$$ \norm{f_{k, v, i, j} - \widetilde{f_{k, v, i, j}}} < \frac{\ep}{2 \cdot M \cdot L \cdot |V_n| \cdot h_n(v)^2} , $$
	it follows that $\norm{a_n - \widetilde{a_n}} <  \ep / 2$.
	
	As $\widetilde{f_{k, v, i, j}}$ is constant on $X(N, v', k')$, it follows that $\phi_{n, N} (\widetilde{a_n}) \in B_N$. It is clear that
	%$$
	%\begin{array}{lll}
	\begin{align*}
		\norm{\phi_{n, N} (\widetilde{a_n}) - a} & \leq \norm{\phi_{n, N} (\widetilde{a_n}) - a_n} + \norm{a - a_n} \\
			& = \norm{\widetilde{a_n} - a_n} + \norm{a - a_n} \\
			& \leq \ep / 2 + \ep / 2 \\
			& = \ep .
	\end{align*}
	%\end{array}
	%$$
	Note that $a \in A_x$ and $\ep > 0$ are arbitrary. It follows that $\bigcup_{n \in \N} \phi_{n, \infty} (B_n)$ is dense in $A_x$. In other words, we have $\displaystyle \varinjlim_n (B_n, \phi_{n, n + 1}) \cong A_x$. As $B_n = \bigoplus_{v \in V_n} M_{h_n(v)}(\C \otimes C(\T^2))$, we conclude that $A_x \cong \displaystyle \varinjlim_n \bigoplus_{s = 1}^{k_n} M_{d_{s, n}}(C(\T^2))$.
\end{proof}

%\onecm

\begin{lem} \label{simplicity of A_x}
 	Let $A_x$ be defined as above. If $\af \times \mathrm{R}_{\xi} \times \mathrm{R}_{\eta}$ is minimal, then $A_x$ is simple.
\end{lem}

\begin{proof}
	This proof is essentially the same as that of Proposition 3.3 (5) in \cite{LM1}. It works like this:

	Note that $X \times \T \times \T$ is compact and $\af \times \mathrm{R}_{\xi} \times \mathrm{R}_{\eta}$ is minimal. It is clear that the positive orbit (under $\af \times \mathrm{R}_{\xi} \times \mathrm{R}_{\eta}$) of $(x, t_1, t_2)$ is dense in $X \times \T \times \T$.

	The \ca \ $A$ corresponds to the groupoid \ca \ associated with the equivalence relation
	$$ \CalR = \{ ( (x, t_1, t_2), (\af \times \mathrm{R}_{\xi} \times \mathrm{R}_{\eta})^k (x, t_1, t_2) ) \colon (x, t_1, t_2) \in X \times \T \times \T \} , $$
	and the C*-subalgebra $A_x$ corresponds to the groupoid \ca \ associated with the equivalence relation
	$$ \CalR_x = \CalR \setminus \{ (\af \times \mathrm{R}_{\xi} \times \mathrm{R}_{\eta})^k (x, t_1, t_2) ) , (\af \times \mathrm{R}_{\xi} \times \mathrm{R}_{\eta})^l (x, t_1, t_2) ) \colon  \\$$
	$$ \hspace{2cm} (t_1, t_2) \in \T \times \T,  k \geq 0, l \leq 0 \ \text{or} \ k \leq 0 , \ l \geq 0  \}  . $$

	As the positive orbit of any $(x, t_1, t_2)$ is dense in $X \times \T \times \T$, it follows that each equivalence class of $\CalR_x$ is dense in $X \times \T \times \T$. \removeme[check the reference myself]{According to \cite[Proposition 4.6]{Renault}}, this is equivalent to the simplicity of $A_x$.
\end{proof}

%\vspace{8mm}

	Now we study the $K$-theory of $A_x$ using its direct limit structure.

\vspace{2mm}

\begin{lem} \label{K_0 of C(T^2)}
	The group $K_0(C(\T^2))$ is order isomorphic to $\Z^2$ with the unit element identified with $(1, 0)$ and the positive cone $D$ being $\{ (m, n) \colon m > 0 \} \cup \{ (0, 0) \}$, and the group $K_1( C( \T^2 ) )$ is isomorphic to $\Z^2$.
\end{lem}

\begin{proof}
%	By the K\"unneth Theorem, it follows that
%	$$ K_0(C(\T^2)) \cong K_0(C(\T)) \otimes K_0(C(\T)) \bigoplus K_1(C(\T)) \otimes K_1(C(\T)) \cong \Z^2 , $$
%	and
%	$$ K_1(C(\T^2)) \cong K_0(C(\T)) \otimes K_1(C(\T)) \bigoplus K_1(C(\T)) \otimes K_0(C(\T)) \cong \Z^2 . $$

	It follows from the K\"unneth Theorem that $K_0( C(\T^2) ) \cong \Z^2$ and $K_1( C(\T^2) ) \cong \Z^2$.

%	Let $e$ be the Bott element, then it is known that $1_m \oplus e$ is a projection of rank $m$ and $[1_m \oplus e] \neq [1_m]$ in $K_0(A)$.

%	For $C(\T^2)$, it is known that the order on $K_0( C ( \T^2 ) )$ is determined by the first copy of $\Z$, which corresponds to the rank of projections. It follows that $K_0( C ( \T^2 ) )_+$ can be identified with $D$.

%	QQQQ
	From algebraic topology, we know that the complex vector bundles on $\T^2$ is generated by the $1_m$ and the Bott element, with $1_m$ being the $m$-dimensional trivial bundle, and the rank is determined by the rank of the trivial bundle, this will give the positive cone of $K_0(C(\T^2))$ as $\{ (m, n) \colon (m, n) \in \Z^2, m > 0 \} \cup \{ (0, 0) \}$.
% QQQQ: Add the details about the order.
\end{proof}

\vspace{2mm}

\begin{lem} \label{K_0 of C(X T^2)}

	There is an isomorphism $\iota \colon K_0(C(X \times \T^2)) \longrightarrow C(X, \Z^2)$ which sends $[1]$ to the constant function with value $(1, 0)$. Furthermore, $\iota$ maps $K_0( C( X \times \T^2 ) )_+$ onto $C(X, D)$, with $D$ as defined in Lemma \ref{K_0 of C(T^2)}.
	
	Moreover, for a clopen set $U$ of $X$ and a projection $\eta \in M_k(C(\T^2))$ such that $[\eta] \in K_0(C(\T^2))$ corresponds to $(a, b)$ as in Lemma \ref{K_0 of C(T^2)}, $\iota([\diag(\underbrace{1_U, \ldots, 1_U}_{k}) \cdot \eta]) = (1_U \cdot a, 1_U \cdot b)$.
	
\begin{proof}
	For $D$ as in Lemma \ref{K_0 of C(T^2)}, define
	$$\varphi \colon C(X, D) \rightarrow (K_0(C(X \times \T^2)))_+$$
	by
	$$ \varphi(f) = \sum_{(m, n) \in D} \left[ (\underbrace{1_{f^{-1}((m, n))}, \ldots, 1_{f^{-1}((m, n))}}_{d_{m, n}}) \cdot \eta_{m, n} \right] , $$
	where $\eta_{m, n}$ is a projection in $M_{d_{m, n}}(C(\T^2))$ which is identified with $(m, n)$ as in Lemma \ref{K_0 of C(T^2)}.

	If we can show that $\varphi$ is one-to-one, preserves addition, and maps the constant function with value $(1, 0)$ to $[1_{C(X \times \T^2)}]$, then we can extend $\varphi$ to a group isomorphism from $C(X, \Z^2)$ to $K_0(C(X \times \T^2))$.
	
	It is easy to check that $\varphi((1, 0)) = [1_{C(X \times \T^2)}]$. From the definition, it follows that $\varphi$ preserves addition. We just need to show that $\varphi$ is one-to-one.
	
	\onemm
	
	Injectivity of $\varphi$:
	
	\onemm
	
	If $\varphi(f) = 0$ for some $f \in C(X, D)$, then
	$$ \sum_{(m, n) \in D} \left[ (\underbrace{1_{f^{-1}((m, n))}, \ldots, 1_{f^{-1}((m, n))}}_{d_{m, n}}) \cdot \eta_{m, n} \right] = 0 $$
	in $(K_0(C(X \times \T^2)))_+$. As
	$$ K_0(C(X \times \T^2)) \cong \bigoplus_{(m, n) \in D} K_0(C(f^{-1}((m, n)) \times \T^2)) , $$
	we get that 	
	$$ [(\underbrace{1_{f^{-1}((m, n))}, \ldots, 1_{f^{-1}((m, n))}}_{d_{m, n}}) \cdot \eta_{m, n}] = 0 \ \text{in} \ K_0(C(f^{-1}((m, n)) \times \T^2)) \ \text{for all} \ (m, n) \in D . $$
	That is, there exists $k \in \N$ such that
	$$(\underbrace{1_{f^{-1}((m, n))}, \ldots, 1_{f^{-1}((m, n))}}_{d_{m, n}}) \cdot \eta_{m, n} \bigoplus \diag(\underbrace{1_{C(f^{-1}((m, n)) \times \T^2)}, \ldots, 1_{C(f^{-1}((m, n)) \times \T^2)}}_k)$$
	is Murray-von Neumann equivalent to $\diag(\underbrace{1_{C(f^{-1}((m, n)) \times \T^2)}, \ldots, 1_{C(f^{-1}((m, n)) \times \T^2)}}_k)$.

	Let $s \in M_{d_{m, n} + k}(f^{-1}((m, n)) \times \T^2)$ be the partial isometry corresponding to the Murray-von Neumann equivalence above. Choose $x \in f^{-1}((m, n))$. Then $s(x)$ can be regarded as an element in $M_{d_{m, n} + k}(\T^2)$ that gives a Murray-von Neumann equivalence between 
	$$ \eta_{m, n} \oplus \diag(\underbrace{1_{C(\T^2))}, \ldots, 1_{C(\T^2)}}_k) \ \text{and} \ \diag(\underbrace{1_{C(\T^2))}, \ldots, 1_{C(\T^2)}}_k) . $$
	It then follows that $\eta_{m, n} = 0$, which proves injectivity.
	
	\onemm
	
	Surjectivity of $\varphi$:
	
	For every projection $p \in M_{\infty}(C(X \times \T^2))$, we can find a partition $X = \bigsqcup_{i = 1}^M X_i$ such that $\norm{p(x) - p(y)} < 1$ for all $x, y \in X_i$. Choose $x_i \in X_i$ for $i = 1, \ldots, M$, and identify $M_{\infty}(C(X \times \T^2))$ with $C(X, M_{\infty}(C(\T^2)))$. Define $p' \in C(X, M_{\infty}(C(\T^2)))$ by $p' \left|_{X_i} \right. = p(x_i)$. It is clear that we can regard $p' \left|_{X_i} \right.$ as an element in $M_{\infty}(C(\T^2))$.

	Use $(a_i, b_i)$ to denote the corresponding element in $K_0(C(\T^2))$ as identified in Lemma \ref{K_0 of C(T^2)} and let $f = \sum_{i = 1}^M 1_{X_i} \cdot (a_i, b_i)$. Then we can check that $\varphi(f) = [p']$ in $(K_0(C(X \times \T^2)))_+$. As $[p] = [p']$, we have proved surjectivity of $\varphi$.
	
	\twomm
	
	As $\varphi$ is unital, one-to-one and preserves addition, we can extend it to an ordered group isomorphism $\widetilde{\varphi} \colon C(X, \Z^2) \longrightarrow K_0(C(X \times \T^2))$. Let $\iota = \widetilde{\varphi}^{-1}$, and we have finished the proof.
\end{proof}

\end{lem}

\begin{lem} \label{structure of A_n}
	There is an isomorphism
	$$ \gamma_n \colon A_n \longrightarrow \bigoplus_{v \in V_n} M_{h_n(v)} \left( C(X(n, v, 1)) \otimes C(\T^2) \right) , $$
	such that for every clopen set $U$ in $X$,
	$$ \gamma_n(1_{U \times \T^2}) = \bigoplus_{v \in V_n} \diag \left( 1_{X(n, v, 1) \cap U}, \ldots, 1_{X(n, v, h(v)) \cap U} \right) . $$

\begin{proof}
  The proof is essentially the same as that of \cite[Lemma 3.1]{Putnam}. It can also be obtained as a K-theory version of part of the proof of Lemma \ref{AT2 structure of A_x}.
\end{proof}

\end{lem}

\begin{lem} \label{technical lemma}

	There is a group isomorphism
	$$ \phi \colon \bigoplus_{v \in V_n} C \left( X(n, v, 1), \Z^2 \right) \longrightarrow  C \left( X, \Z^2 \right) / \{ f - f \circ \af^{-1}  \colon f \left| _{Y_n} \right. = 0  \} $$
	such that
	$$ \phi \left( (f_1, \ldots, f_{|V_n|}) \right) = \sum_{v \in V_n} [1_{X(n, v, 1)} \cdot f_v] $$
	for $(f_1, \ldots, f_{|V_n|}) \in \displaystyle\bigoplus_{v \in V_n} C\left( X(n, v, 1), \Z^2 \right)$.
	
	Furthermore, if we define $D$ to be
	$$ \{ (m, n) \in \Z^2 \colon  m >0 \} \cup \{ (0, 0) \} , $$
	and if we define the positive cone of $\displaystyle\bigoplus_{v \in V_n} C(X(n, v, 1), \Z^2)$ to be $\displaystyle\bigoplus_{v \in V_n} C(X(n, v, 1), D)$ and the positive cone of $C(X, \Z^2) / \{ f - f \circ \af^{-1} \colon f \left| _{Y_n} \right. = 0  \}$ to be $C(X, D) / \{ f - f \circ \af^{-1} \colon f \left| _{Y_n} \right. = 0  \}$, then both $\phi$ and $\phi^{-1}$ are order preserving.

\begin{proof}	
	For $(f_1, \ldots, f_{|V_n|}) \in \bigoplus_{v \in V_n} C \left( X(n, v, 1), \Z^2 \right)$, define
	$$ \phi \left( f_1, \ldots, f_{|V_n|} \right) = \sum_{v \in V_n} [1_{X(n, v, 1)} \cdot f_v] . $$
	
	Injectivity of $\phi$:
	
	Suppose
	$$ (f_1, \ldots, f_{|V_n|}) \in \bigoplus_{v \in V_n} C(X(n, v, 1), \Z^2) $$
	and that $\phi((f_1, \ldots, f_{|V_n|})) = 0$. That is, there exists $H \in C(X, \Z^2)$ with $H \left| _{Y_n} \right. = 0$ such that
	$$ \sum_{v = 1}^{|V_n|} f_v = H - H \circ \af^{-1} . $$
	It follows that
	$$ \left( \sum_{k = 1}^{h(v)} 1_{X(n, v, k)} \right) \cdot \left( \sum_{v = 1}^{|V_n|} f_v \right) = \left( \sum_{k = 1}^{h(v)} 1_{X(n, v, k)} \right) \cdot \left( H - H \circ \af^{-1} \right) . $$
	As $H \left|_{Y_n} \right. = 0$,
	$$ \left( \sum_{k = 1}^{h(v)} 1_{X(n, v, k)} \right) \cdot (H \circ \af^{-1}) = \left(\sum_{k = 1}^{h(v)} 1_{X(n, v, k)} \cdot H \right) \circ \af^{-1} . $$
	It then follows that
	$$ \left( \sum_{k = 1}^{h(v)} 1_{X(n, v, k)} \right) \cdot \left(H - H \circ \af^{-1} \right) =  \left( \sum_{k = 1}^{h(v)} 1_{X(n, v, k}) \cdot H \right) - \left( \sum_{k = 1}^{h(v)} 1_{X(n, v, k} \cdot H \right) \circ \af^{-1} . $$
	Use $H_v$ to denote $\displaystyle \left( \sum_{k = 1}^{h(v)} 1_{X(n, v, k)} \right) \cdot H$. It is clear that $H_v$ is supported on $\displaystyle\bigsqcup_{k = 1}^{h(v)} X(n, v, k)$.
	
	Now we have $f_v = H_v - H_v \circ \af^{-1}$. As $f_v$ is supported on $X(n, v, 1)$, we get
	$$ H_v - H_v \circ \af^{-1} = 0 $$
	on $X(n, v, k)$ for $2 \leq k \leq h(v)$, which implies that for all $x \in X(n, v, 1)$,
	$$ H(x) = H_v(\af(x)) = \cdots = H_v \left( \af^{h(v) - 1}(x) \right) . $$
	As $\af^{h(v) - 1}(x) \in Y_n$, it follows that $H_v(\af^{h(v) - 1}(x)) = 0$. Now we can conclude that $H_v = 0$. It is then clear that $f_v = 0$.
	
	Applying the process to all $v = 1, \ldots, h(v)$, we get $H = 0$. It follows that $f_i = 0$ for $i = 1, \ldots, |V_n|$, which proves the injectivity of $\phi$.

\vspace{4mm}
	
	Surjectivity of $\phi$:
	
	For every $g \in C(X, \Z^2)$, we need to find
	$$ (f_1, \ldots, f_{|V_n|}) \in \bigoplus_{v \in V_n} C \left( X(n, v, 1), \Z^2 \right) $$
	such that
	$$ \phi \left( (f_1, \ldots, f_{|V_n|}) \right) - g = h - h \circ \af^{-1} $$
	for some $h \in C(X, \Z^2)$ satisfying $h \left| _{Y_n} \right. = 0$.
	
	Write $g$ as
	$$ g = 1 \cdot g = \sum_{v \in V_n} \sum_{k = 1}^{h(v)} 1_{X(n, v, k)} \cdot g . $$
	For every $k$ with $2 \leq k \leq h(v)$, consider $(1_{X(n, v, k)} \cdot g) \circ \af$. It is easy to check that \\ $(1_{X(n, v, k)} \cdot g) \circ \af \left|_{Y_n} \right. = 0$ and
	$$ 1_{X(n, v, k)} \cdot g + \left( (1_{X(n, v, k)} \cdot g) \circ \af - (1_{X(n, v, k)} \cdot g) \circ \af  \circ \af^{-1} \right) $$
	is supported on $X(n, v, k - 1)$.
	
	By repeating this process, we get $s \in C(X, \Z^2)$ such that $1_{X(n, v, k)} \cdot g + (s - s \circ \af)$ is supported on $1_{X(n, v, 1)}$.
	
	Apply the process for all $1_{X(n, v, k)} \cdot g$ with $v \in V_n$ and $1 < k \leq h(v)$. We can find $H \in C(X, \Z^2)$ such that $g + (H - H \circ \af^{-1})$ is supported on $\af(R(\CalP_n)) = \bigoplus_{v \in V_n} X(n, v, 1)$. According to the definition, if we set $f_v = 1_{X(n, v, 1)} \cdot (g + (H - H \circ \af^{-1}))$, then $\phi$ will map $(f_1, \ldots, f_{|V_n|})$ to $g$.

	Positivity of $\phi$:
	
	As
	$$ \phi \left( (f_1, \ldots, f_{|V_n|}) \right) = \sum_{v \in V_n} 1_{X(n, v, 1)} \cdot f_v , $$
	for
	$$\left( f_1, \ldots, f_{|V_n|} \right) \in \bigoplus_{v \in V_n} C \left( X(n, v, 1), \Z^2 \right) , $$
	if the range of each $f_i$ is in the positive cone $D$, it is clear that $\sum_{v \in V_n} 1_{X(n, v, 1)} \cdot f_v \in C(X, D)$. Thus $\phi$ is order preserving.
	
	Positivity of $\phi^{-1}$:
	
	For $f \in C(X, D)$, we will show that if there is
	$$ (f_1, \ldots, f_{| V_n |}) \in \bigoplus_{v \in V_n} C \left( X(n, v, 1), D \right) $$
	such that
	$$ \phi( f_1, \ldots, f_{| V_n |} ) = [f] , $$
	then $f_v \in C( X(n, v, 1 )$ for all $1 \leq v \leq | V_n |$.
	
	In fact, such an element $(f_1, \ldots, f_{| V_n |})$ can be constructed from $f$ as in the proof of surjectivity of $\phi$. The fact that $f \in C(X, D)$ then implies that for all $v$ with $1 \leq v \leq | V_n |$, the image $f_k$ is in $D$, which finishes the proof.
\end{proof}

\end{lem}

\begin{lem} \label{K_0 of A_n}
	There is an order isomorphism
	$$ \rho_n \colon K_0(A_n) \longrightarrow C(X, \Z^2) / \{ f - f \circ \af^{-1} \colon f \in C(X, \Z^2), f \left|_{Y_n} \right. = 0 \} $$
	with the unit element and positive cone of
	$$ C(X, \Z^2) / \{ f - f \circ \af^{-1} \colon f \in C(X, \Z^2), f \left|_{Y_n} \right. = 0 \} $$
	being $[(1_X, 0)]$ and
	%$$
	%\begin{array}{l}
	\begin{align*}
	 \{ [g] \in C(X, \Z^2) /  \{ f - f \circ \af^{-1} \colon f \in C(X, \Z^2), f \left|_{Y_n} \right. = 0 \} \colon \\
	 \hspace{1cm} \forall  x \in X, g(x) = (0, 0) \ \text{or} \ g(x) = (a, b) \ \text{with} \ a > 0 \} .
	 \end{align*}
	 %\end{array}
	%$$
	For a clopen subset $U$ of $X$ and $\eta \in M_k(C(\T^2))$ such that $[\eta] \in K_0(C(\T^2))$ corresponds to $(a, b)$ as in Lemma \ref{K_0 of C(T^2)}, $\rho_n([\diag(\underbrace{1_U, \ldots, 1_U}_{k}) \cdot \eta])$ is exactly $[(1_U \cdot a, 1_U \cdot b)]$ with $1_U$ denoting the continuous function from $X$ to $\Z$ that is $1$ on $U$ and $0$ otherwise.

\begin{proof}
	Consider the isomorphism
	$$ \gamma_n \colon A_n \longrightarrow \bigoplus_{v \in V_n} M_{h_n(v)} \left( C(X(n, v, 1)) \otimes C(\T^2) \right) $$
	as in Lemma \ref{structure of A_n}. It is clear that
	$$ {( \gamma_n )}_{*0} \colon K_0(A_n) \longrightarrow K_0 \left( \bigoplus_{v \in V_n} M_{h_n(v)} \left( C(X(n, v, 1)) \otimes C \left( \T^2 \right) \right) \right) $$
	is an order isomorphism.
	
	We know that
	$$ K_0 \left( \bigoplus_{v \in V_n} M_{h_n(v)}(C(X(n, v, 1)) \otimes C(\T^2)) \right) \cong \bigoplus_{v \in V_n} K_0 \left( M_{h_n(v)}(C(X(n, v, 1)) \otimes C(\T^2)) \right) , $$
	and use
	$$ h_n \colon K_0 \left( \bigoplus_{v \in V_n} M_{h_n(v)}(C(X(n, v, 1)) \otimes C(\T^2)) \right) \longrightarrow \bigoplus_{v \in V_n} K_0 \left( M_{h_n(v)}(C(X(n, v, 1)) \otimes C(\T^2)) \right) $$
	to denote the order isomorphism.
	
	There are
	% QQQQ : Add more details?
	natural order isomorphisms
	$$ l_{n, v} \colon K_0(M_{h_n(v)}(C(X(n, v, 1)) \otimes C(\T^2))) \longrightarrow K_0( C(X(n, v, 1)) \otimes C(\T^2) ) . $$
	By Lemma \ref{K_0 of C(X T^2)}, we can find order isomorphisms
	$$ s_{n ,v} \colon K_0( C(X(n, v, 1)) \otimes C(\T^2) ) \longrightarrow C(X(n, v, 1), \Z^2) $$
	such that each $s_{n, v}$ maps $[1_{C(X(n, v, 1)) \otimes C(\T^2)}]$ to the constant function with value $(1, 0)$.
	
	Combining $l_{n, v}$ and $s_{n, v}$ for all $v$, we get an order isomorphism
		$$ \varphi \colon \bigoplus_{v \in V_n} K_0(M_{h_n(v)}(C(X(n, v, 1)) \otimes C(\T^2))) \longrightarrow \bigoplus_{v \in V_n} C(X(n, v, 1), \Z^2) $$
	with the positive cone of $\bigoplus_{v \in V_n} C(X(n, v, 1), \Z^2)$ being $\bigoplus_{v \in V_n} C(X(n, v, 1), D)$ ($D$ as defined in Lemma \ref{K_0 of C(T^2)}). Note that $\varphi$ is not unital.
	
	According to Lemma \ref{technical lemma}, there is an order isomorphism
	$$ \psi \colon \bigoplus_{v \in V_n} C(X(n, v, 1), \Z^2) \rightarrow  C(X, \Z^2) / \{ f - f \circ \af^{-1} \colon f \left| _{Y_n} \right. = 0  \} . $$

	Let
	$$\rho_n = \psi \circ \varphi \circ h_n \circ (\gamma_n)_{*0} . $$
	Then $\rho_n$ is a group isomorphism from $K_0(A_n)$ to
	$$ C(X, \Z^2) / \{ f - f \circ \af^{-1} \colon f \in C(X, \Z^2), f \left|_{Y_n} \right. = 0 \} $$
	because $\psi$, $\varphi$, $h_n$ and $(\gamma_n)_{*0}$ are all group isomorphisms.

\twomm

	According to Lemma \ref{structure of A_n},
	$$ \gamma_n(1_{A_n}) =  \bigoplus_{v \in V_n} \diag(1_{X(n, v, 1)}, \ldots, 1_{X(n, v, h(v)) }) . $$ Thus
	$$ (\gamma_n)_{*0}([1_{A_n}]) =  \sum_{v \in V_n} \sum_{1 \leq k \leq h(v)} \left[ 1_{X(n, v, k)} \right] . $$
	It is then clear that
	$$ h_n( (\gamma_n)_{*0}([1_{A_n}])) = \left( \sum_{1 \leq k \leq h(1)} \left[ 1_{X(n, v, k)} \right], \ldots, \sum_{1 \leq k \leq h(|V_n|)} \left[ 1_{X(n, v, h(k))} \right] \right) . $$
	
	Note that $[1_{X(n, v, k)}] = [1_{X(n, v, 1)}]$ in $K_0(M_n(X(n, v, 1)))$. It follows that
	%$$
	%\begin{array}{lll}
	\begin{align*}
		\displaystyle \varphi \left( h_n({( \gamma_n )}_{*0}([1_{A_n}])) \right) & = \varphi \left( \displaystyle\sum_{1 \leq k \leq h(1)} [1_{X(n, v, k)}], \ldots, \sum_{1 \leq k \leq h(|V_n|)} [1_{X(n, v, h(k))}]  \right) \vspace{1mm} \\
			& = \displaystyle\sum_{v \in V_n} h(v) \cdot \left[ 1_{X(n, v, 1)} \right] .
	\end{align*}
	%\end{array}
	%$$
	
	According to the definition of $\phi$ as stated in Lemma \ref{technical lemma}, we get
	$$ \psi \left( \varphi(h_n({( \gamma_n )}_{*0}([1_{A_n}]))) \right) = \psi \left(\sum_{v \in V_n} h(v) \cdot [1_{X(n, v, 1)}] \right) = \sum_{v \in V_n} \left[ f_v \right]$$
	with $f_v \in C(X, \Z^2)$ satisfying $f_v \left|_{X(n, v, 1)} \right. = h(v)$ and $f_v \left|_{X \backslash X(n, v, 1)} \right. = 0$.
	
	Let
	$$ H = \sum_{v \in V_n} \sum_{1 \leq k \leq h(v) - 1} 1_{X(n, v, k)} \cdot (h(v) - k) . $$
	Then it is clear that $H \left|_{Y_n} \right. = 0$ and
	$$ H_v \circ \af^{-1} = \sum_{v \in V_n} \sum_{2 \leq k \leq h(v)} 1_{X(n, v, k)} \cdot (h(v) - k + 1) . $$
	It is easy to check that
	$$ H - H \circ \af^{-1} = \sum_{v \in V_n} \left[ \left( \sum_{2 \leq k \leq h(v)} 1_{X(n, v, k)} \cdot (-1) \right) + 1_{X(n, v, 1)} \cdot \left( h(v) - 1 \right) \right] . $$
	
	In $C(X, \Z^2)$, it is easy to check that $(\sum_{v \in V_n} f_v) - 1_{X} = H - H \circ \af^{-1}$. In other words, we have
	$$ \psi(\varphi(h_n({( \gamma_n )}_{*0}([1_{A_n}])))) = \sum_{v \in V_n} [f_v] = [1_X] \ , $$
	which implies that $\rho_n$ is unital.	
	
\twomm
	
	To show that $\rho_n$ is order preserving, we just need to show that $\psi, \varphi, h_n$ and $(\gamma_n)_{*0}$ are all order preserving.
	
	It is clear that $h_n$ and ${( \gamma_n )}_{*0}$ are order preserving. According to Lemma \ref{technical lemma}, $\psi$ is also order preserving. We just need to show that $\varphi$ is order preserving.
	
	Note that $\varphi = \bigoplus_{v \in V_n} (s_{n, v} \circ l_{n, v})$. We just need to show that each $s_{n, v} \circ l_{n, v}$ is order preserving. In fact, $l_{n, v}$ is order preserving and $s_{n, v}$ is an order isomorphism. It follows that $s_{n, v} \circ l_{n, v}$ is order preserving. Thus $\varphi$ is order preserving.

\twomm

	Now we will show that $\rho_n$ is order isomorphism. In fact, we just need to show that for every $(a, b) \in \{ (m, n) \colon m > 0, n \in \Z \} \cup \{ 0, 0 \}$ and every clopen subset $U$ of $X$, if we regard $(1_U \cdot a, 1_U \cdot b)$ as a function in $C(X, \Z^2)$ defined by
	$$ (1_U \cdot a, 1_U \cdot b) (x) = \left\{  \begin{array}{ll} (a, b) & \text{if} \ x \in U  \\ (0, 0) & \text{if} \ x \notin U \end{array} \right. $$
	and we define
	$$ \pi \colon C(X, \Z^2) \longrightarrow C(X, \Z^2) / \{ f - f \circ \af^{-1} \colon f \in C(X, \Z^2), f \left|_{Y_n} \right. = 0 \} $$
	to be the natural quotient map, then $\pi ( (1_U \cdot a, 1_U \cdot b) )$ is in the image of $\rho_n ( K_0(A_n)_+ )$.	
	
\twomm	
	
	For a clopen subset $U$ of $X$ and $\eta \in M_k(C(\T^2))$ such that $[\eta] \in K_0(\T^2)$ corresponds to the $(a, b)$ above (see Lemma \ref{K_0 of C(T^2)}), we have
	$$ \rho_n([\diag(\underbrace{1_U, \ldots, 1_U}_{k}) \cdot \eta]) = (\phi \circ \varphi \circ h_n \circ (\gamma_n)_{*0})([\diag(\underbrace{1_U, \ldots, 1_U}_{k}) \cdot \eta]) . $$
	According to Lemma \ref{structure of A_n},
	%$$
	%\begin{array}{lll}
	\begin{align*}
	& (h_n \circ {( \gamma_n )}_{*0})([\diag(\underbrace{1_U, \ldots, 1_U}_{k}) \cdot \eta]) \\
	& = ( h_n \circ {( \gamma_n )}_{*0}) \left( \displaystyle\sum_{v \in V_n, 1 \leq k \leq h(v)} [\diag(\underbrace{1_{X(n, v, k) \cap U}, \ldots, 1_{X(n, v, k) \cap U}}_{k}) \cdot \eta] \right) \\
	& = \displaystyle \left( \sum_{1 \leq k \leq h(v)} \left[ 1_{X(n, v, k) \cap U} \cdot \eta \right] \right)_{v \in V_n} .
	\end{align*}
	%\end{array}
	%$$
	Then
	\begin{align*}
	 & (\varphi \circ h_n \circ {( \gamma_n )}_{*0})([\diag(\underbrace{1_U, \ldots, 1_U}_{k}) \cdot \eta]) \\
	 & = \left( \displaystyle \sum_{1 \leq k \leq h(v)} \left( 1_{\af^{-(k - 1)}(X(n, v, k) \cap U)} \cdot a, 1_{\af^{-(k - 1)}(X(n, v, k) \cap U)} \cdot b \right) \right) _{v \in V_n}
	\end{align*}
	which is an element of $\bigoplus_{v \in V_n} C(X(n, v, 1), \Z^2)$.
	
	According to the definition of $\phi$ as in Lemma \ref{technical lemma}, it follows that
	%$$
	%\begin{array}{lll}
	\begin{align*}
	(\psi \circ \varphi \circ h_n \circ (\gamma_n)_{*0})([\diag(\underbrace{1_U, \ldots, 1_U}_{k}) \cdot \eta]) & = (\psi) ((\varphi \circ h_n \circ (\gamma_n)_{*0})([\diag(\underbrace{1_U, \ldots, 1_U}_{k}) \cdot \eta]))\\
	& = \displaystyle\sum_{v \in V_n} 1_{X(n, v, 1)} \cdot f_v
	\end{align*}
	%\end{array}
	%$$
	with
	$$ \displaystyle f_v = \left(\sum_{1 \leq k \leq h(v)} 1_{\af^{-(k - 1)}(X(n, v, k) \cap U)} \cdot a, \sum_{1 \leq k \leq h(v)} 1_{\af^{-(k - 1)}(X(n, v, k) \cap U)} \cdot b \right) . $$
	
	Note that for all $k$ with $1 \leq k \leq h(v) - 1$, we have $1_{X(n, v, k)} \left|_{Y_n} \right. = 0$. Also, we can check that
	$$ 1_{X(n, v, k) \cap U} - 1_{X(n, v, k) \cap U} \circ \af^{-1} = 1_{X(n, v, k) \cap U} - 1_{\af(X(n, v, k) \cap U)} . $$
	It follows that
	$$ [1_{X(n, v, k) \cap U}] = [1_{\af(X(n, v, k) \cap U)}] \ \text{in} \ C(X, \Z) / \{ f - f \circ \af^{-1}  \colon f \in C(X, \Z), f \left|_{Y_n} \right. = 0 \} $$
	for $k = 1, \ldots, h(v)$. We then get that in $C(X, \Z) / \{ f - f \circ \af^{-1} \colon f \in C(X, \Z), f \left|_{Y_n} \right. = 0 \}$,
	$$ \left[ \sum_{1 \leq k \leq h(v)} 1_{\af^{-(k - 1)}(X(n, v, k) \cap U)} \right] = \left[ \sum_{1 \leq k \leq h(v)} 1_{X(n, v, k) \cap U} \right] . $$
	It then follows that
	%$$
	%\begin{array}{lll}
	\begin{align*}
	\displaystyle \left[ \sum_{v \in V_n} f_v \right] & = \left[ \displaystyle\sum_{ \substack{v \in V_n \\ 1 \leq k \leq h(v)} } \left( 1_{X(n, v, k) \cap U} \cdot a, 1_{X(n, v, k) \cap U} \cdot b \right) \right] \\
	& = \left( [1_U] \cdot a, [1_U] \cdot b \right)
	\end{align*}
	%\end{array}
	%$$
	in $C(X, \Z^2) / \{ f - f \circ \af^{-1} \colon f \in C(X, \Z^2), f \left|_{Y_n} \right. = 0 \}$.
	
	We have proved that $\rho_n ([\diag(\underbrace{1_U, \ldots, 1_U}_{k}) \cdot \eta]) = \pi ( (1_U \cdot a, 1_U \cdot b) )$. It then follows that $\rho_n$ is an order isomorphism, which finishes the proof.
\end{proof}

\end{lem}

\onecm

\begin{cor} \label{representation of K_0(A_n)}
	Let $p$ be a projection in $M_{\infty}(A_n)$. Then there exists $p' \in M_{\infty}(C(X \times \T^2)) \subset M_{\infty} (A_n)$ such that $[p] = [p']$ in $K_0(A_n)$.
	
\begin{proof}	
	According to Lemma \ref{K_0 of A_n}, we have an isomorphism
	$$ \rho_n \colon K_0(A_n) \rightarrow C(X, \Z^2) / \{ f - f \circ \af^{-1} \colon f \in C(X, \Z^2), f \left|_{Y_n} \right. = 0 \} . $$
	Let $\rho_n([p]) = [g]$ for some $g \in C(X, \Z^2)$. Without loss of generality, we can assume that there is a partition of $X$ as $X = \bigsqcup_{i = 1}^N X_{i}$ such that this partition is finer than $\CalP_n$ and $g \left|_{X_i} \right.$ is constant for $i = 1, \ldots, N$.
	
	As $[p]$ is in $(K_0(A_n))_+$ and $\rho_n$ is an order isomorphism, it follows that $[g]$ is in the positive cone (defined in the statement of Lemma \ref{technical lemma}). For as $g$ above with $\rho_n([p]) = [g]$, we can assume that on any given $X_i$, $g \left|_{X_i} \right.$ is either $(0, 0)$ or $(a_i, b_i) \in \Z^2$ with $a_i > 0$.

	According to Lemma \ref{K_0 of C(T^2)}, there exist projections $\eta_i \in M_{d(i)} (C(\T^2))$ such that $[\eta_i]$ in $K_0(C(\T^2))$ can be identified with $(a_i, b_i)$.
	
	Let
	$$ p' = \diag \left( \diag ( \underbrace{1_{X_1}, \ldots, 1_{X_1}}_{d(1)} ) \cdot \eta_1, \ldots, \diag(\underbrace{1_{X_N}, \ldots, 1_{X_N}}_{d(N)}) \cdot \eta_N \right) . $$
	Then it is clear that $p' \in M_{\infty}(C(X \times \T^2))$.
	
	According to Lemma \ref{K_0 of A_n}, $\rho_n([p']) = [g]$, so that $\rho_n([p']) = \rho_n([p])$. As $\rho_n$ is an isomorphism (by Lemma \ref{K_0 of A_n} again), it follows that $[p] = [p']$ in $K_0(A_n)$.
\end{proof}

\end{cor}

\begin{lem} \label{commutativity of identification}

	Let $j_n \colon C(X \times \T^2) \longrightarrow A_n$ be the canonical embedding, and let $\iota$ and $\rho_n$ be as in Lemma \ref{K_0 of C(X T^2)} and Lemma \ref{K_0 of A_n}. Let $(j_n)_{*0} \colon K_0(C(X \times \T^2)) \rightarrow K_0(A_n)$ be the induced map on $K_0$ and let
	$$ \pi_n \colon C(X, \Z^2) \rightarrow C(X, \Z^2) / \{ f - f \circ \af^{-1} \colon f \in C(X, \Z^2), f \left|_{Y_n} \right. = 0 \} $$
	be the canonical quotient map. Then the follow diagram commutes:
	$$
	\xymatrix{
		K_0(C(X \times \T^2))\ar@{->}[rrrr]^{\iota} \ar[d]_{(j_n)_{*0}} &&&& C(X, \Z^2) \ar[d]^{\pi_n} \\
		K_0(A_n) \ar[rrrr]^{\rho_n} &&&& C(X, \Z^2) / \{ f - f \circ \af^{-1} \colon f \in C(X, \Z^2), f \left|_{Y_n} \right. = 0 \} \\
	}_.
	$$

\end{lem}

\begin{proof}
	As $K_0(C(X \times \Z^2))$ is generated by its positive cone $(K_0(C(X \times \Z^2)))_+$, we just need to show that $\pi_n \circ \iota = \rho_n \circ (j_n)_{*0}$ on $(K_0(C(X \times \Z^2)))_+$.

	For every projection $p \in M_{\infty}(C(X \times \T^2))$, according to the proof of surjectivity of $\varphi$ in Lemma \ref{K_0 of C(X T^2)}, there exist a partition $X = \bigsqcup_{i = 1}^M X_i$ and projections $\eta_i \in M_{d_i}(C(\T^2))$ for $i = 1, \ldots, M$ such that
	$$ [p] = \sum_{i = 1}^M [(\underbrace{1_{X_i}, \ldots, 1_{X_i}}_{d_i}) \cdot \eta_i] . $$

	According to Lemma \ref{K_0 of C(T^2)}, $\eta_i$ can be identified with $(a_i, b_i) \in D$. By Lemma \ref{K_0 of C(X T^2)}, we get $\iota([p]) = \sum_{i = 1}^M (1_{X_i} \cdot a_i, 1_{X_i} \cdot b_i)$.

	By Lemma \ref{K_0 of A_n},
	\begin{align*}
		\rho_n((j_n)_{*0}([p])) & = \rho_n((j_n)_{*0}( \sum_{i = 1}^M [(\underbrace{1_{X_i}, \ldots, 1_{X_i}}_{d_i}) \cdot \eta_i] )) \\
			& = \sum_{i = 1}^M [(1_{X_i} \cdot a_i, 1_{X_i} \cdot a_i)] .
	\end{align*}
	It is then clear that $(\pi_n \circ \iota) ([p])= (\rho_n \circ (j_n)_{*0}) ([p])$. Since $p$ is arbitrary, we have finished the proof.
\end{proof}

\begin{cor} \label{handy technical result}

	Let $p, q$ be projections in $M_{\infty}(C(X \times \T^2)) \subset M_{\infty}(A_n)$ such that $\iota([p]) - \iota([q]) = h - h \circ \af^{-1}$ for some $h \in C(X, \Z^2)$ satisfying $h \left|_{Y_n} \right. = 0$, with $\iota$ as in Lemma \ref{K_0 of C(X T^2)}. Then $(j_n)_{*0}([p]) = (j_n)_{*0}([q])$ in $K_0(A_n)$ with $j_n$ as in Lemma \ref{commutativity of identification}.

\end{cor}

\begin{proof}
	This follows directly from Lemma \ref{commutativity of identification}.
\end{proof}

\begin{lem} \label{K_i of A_x}

	For $A_x$ as defined in the beginning of this section,
	$$ K_i(A_x) \cong C(X, \Z^2) / \{  f - f \circ \af^{-1} \colon f \in C(X, \Z^2) \} , $$
	and
	$$ K_0(A_x)_+ \cong C(X, D) / \{  f - f \circ \af^{-1} \colon f \in C(X, \Z^2) \} , $$
	with $D$ defined to be $\{ (a, b) \in \Z^2 \colon a > 0, b \in \Z \} \cup \{ (0, 0) \}$.

\begin{proof}
	From Lemma \ref{K_0 of A_n}, we know that
	$$ K_i(A_n) \cong C(X, \Z^2) / \{ f - f \circ \af^{-1} \colon f \in C(X, \Z^2) \ \text{and} \ f \left|_{Y_n} \right. = 0 \} . $$

	As $A_x = \displaystyle \varinjlim A_n$, we get $K_i(A_x) = \varinjlim K_i(A_n)$. Note that the map
	$$ (j_{n, n + 1})_{*i} \colon K_i(A_n) \rightarrow K_i(A_{n + 1}) $$
	satisfies $(j_{n, n + 1})_{*i} ([f]) = [f]$ for all $f \in C(X, \Z^2)$. We can conclude that
    	$$  K_i(A_x) \cong C(X, \Z^2) / \{  f - f \circ \af^{-1} \colon f \in C(X, \Z^2) \ \text{and} \ f \left|_{Y_n} \right. = 0  \ \text{for some} \ n \in \N \} . $$
	As $\bigcap_{n = 1}^{\infty} Y_n = \{ x \}$, it follows that
	$$ \{ f \in C(X, \Z^2) \colon  f \left|_{Y_n} \right. = 0  \ \text{for some} \ n \in \N  \} = \{ f \in C(X, \Z^2) \colon  f(x) = 0 \} . $$
	Then we have
	$$ K_i(A_x) \cong C(X, \Z^2) / \{  f - f \circ \af^{-1} \colon f \in C(X, \Z^2) \ \text{and} \ f(x) = 0 \} . $$

	For every $g \in C(X, \Z^2)$, define $g_0 = g - g(x)$. It is clear that
	$$ g_0 \in \{ f \in C(X, \Z^2) \ \text{and} \ f(x) = 0 \}. $$
	
	As $g_0 - g_0 \circ \af^{-1} = g - g \circ \af^{-1}$, we have
	$$ K_i(A_x) \cong C(X, \Z^2) / \{  f - f \circ \af^{-1} \colon f \in C(X, \Z^2) \} . $$

	Let $j_{n, \infty} \colon A_n \rightarrow A_x$ be the embedding of $A_n$ into $A_x$. Then
	$$ K_0(A_x)_+ = \bigcup {( j_{n, \infty} )}_{*0} (K_0(A_n)_+) . $$
	According to Lemma \ref{K_0 of A_n},
	$$ K_0(A_n)_+ \cong C(X, D) / \{ f - f \circ \af^{-1} \colon f \in C(X, \Z^2) \ \text{and} \ f \left|_{Y_n} \right. = 0 \} . $$
	Similarly, using the fact that
	$$ \{ f \in C(X, \Z^2) \colon  f \left|_{Y_n} \right. = 0  \ \text{for some} \ n \in \N  \} = \{ f \in C(X, \Z^2) \colon  f(x) = 0 \} , $$ we can conclude that $K_0(A_x)_+ \cong C(X, D) / \{  f - f \circ \af^{-1} \colon f \in C(X, \Z^2) \ \text{and} \  f(x) = 0 \}$.

	As 
	$$ \{  f - f \circ \af^{-1} \colon f \in C(X, \Z^2) \ \text{and} \  f(x) = 0 \} = \{  f - f \circ \af^{-1} \colon f \in C(X, \Z^2) \} , $$
	we get $K_0(A_x)_+ \cong C(X, D) / \{  f - f \circ \af^{-1} \colon f \in C(X, \Z^2) \}$. 
\end{proof}

\end{lem}

\onecm

\begin{cor} \label{tracial rank of A_x is torsion free}

	For $A_x$ as in Definition \ref{dfn of Ax}, $K_i(A_x)$ is torsion free for $i = 0, 1$.

\end{cor}

\begin{proof}
	According to Lemma \ref{K_i of A_x}, we just need to show that $C(X, \Z^2) / \{ f - f \circ \af^{-1} \colon f \in C(X, \Z^2) \}$ is torsion free. A purely algebraic proof is given here.

	Suppose we have $g \in C(X, \Z^2)$ and $n \in \Z \setminus \{ 0 \}$ such that
	$$ [n g] = 0 \ \text{in} \ C(X, \Z^2) / \{ f - f \circ \af^{-1} \colon f \in C(X, \Z^2) \} . $$
	If we can show that $[g] = 0$, then we are done. In other words, we need to find $f \in C(X, \Z^2)$ such that $g = f - f \circ \af^{-1}$.

	As $[n g] = 0$, there exists $F \in C(X, \Z^2)$ such that $n g = F - F \circ \af^{-1}$. If $F(x) \in n \Z^2$ for all $x$, just divide both sides by $n$. Then we get $g = \left( \frac{F}{n} \right) - \left( \frac{F}{n} \right) \circ \af^{-1}$ with $\frac{F}{n} \in C(X, \Z^2)$.

	Fix $x_0 \in X$, and define $\widetilde{F} = F - F(x_0)$. It is clear that $\widetilde{F}(x_0) = 0$. As $F - F \circ \af^{-1} = n g$, we can easily check that $\widetilde{F} - \widetilde{F} \circ \af^{-1} = ng$. It then follows that
	$$ \widetilde{F}(\af (x_0) ) = \widetilde{F} (x_0) + n g(\af(x_0)) = 0 + n g(\af(x_0)) \in n \Z^2 , $$
	$$ \widetilde{F}(\af^2 (x_0) ) = \widetilde{F} (\af(x_0)) + n g(\af^2(x_0)) \in n \Z^2 , $$
	$$ \cdots $$
	So for every $x \in \text{Orbit}_{\Z} (x_0)$, we get $\widetilde{F}(x) \in n \Z^2$. Note that $\widetilde{F}$ is continuous on $X$ and $\text{Orbit}_{\Z} (x_0)$ is dense in $X$. It follows directly that $\widetilde{F}(x) \in n \Z^2$ for all $x \in X$, thus finishing the proof.
\end{proof}

\begin{cor} \label{tracial rank of A_x}

	For $A_x$ as in Definition \ref{dfn of Ax}, $\mathrm{TR}(A_x) \leq 1$.

\begin{proof}
	From Lemma \ref{AT2 structure of A_x}, we know that $A_x$ is a AH algebra with no dimension growth. By Lemma \ref{simplicity of A_x}, $A_x$ is simple. According to Lemma \ref{K_i of A_x}, $K_i(A_x)$ is torsion free.

\removeme[Find the reference.]{As $A_x$ is a simple AH algebra with no dimension growth, it follows that $\mathrm{TR}(A_x) \leq 1$.}
\end{proof}

\end{cor}

\vspace{2cm}

\section{The crossed product \ca \ A} \label{The crossed product C-algebra A}

	This section contains the main theorem (Theorem \ref{tracial rank of A without rigidity}), which states that the tracial rank of the crossed product $C^*(\Z, X \times \T \times \T, \af \times \mbox{R}_{\xi} \times \mbox{R}_{\eta})$ has tracial rank no more than one.
	
\vspace{2mm}

\begin{sloppypar}
	We start by showing that for the natural embedding $j \colon A_x \rightarrow A$, the induced homomorphisms $(j_*)_i \colon K_i(A_x) \rightarrow K_i(A)$ are injective for $i = 0, 1$.
\end{sloppypar}

\begin{lem} \label{injectivity of K_0}

 	Let $A$ be $C^*(\Z, X \times \T \times \T, \af \times \mathrm{R}_{\xi} \times \mathrm{R}_{\eta} )$ and let $A_x$ be as in Definition \ref{dfn of Ax}. Let $j \colon A_x \rightarrow A$ be the canonical embedding. Then $j_{*0}$ is an injective order homomorphism from $K_0(A_x)$ to $K_0(A)$.

\begin{proof}
	It is clear that $j_{*0}$ will induce an order homomorphism from $K_0(A_x)$ to $K_0(A)$ and $j_{*0}$ maps $[1_{A_x}]$ to $[1_A]$.
	
%	As $K_0(A_x)_+$ is generated by projections $\{ 1_U \cdot 1_{C(\T^2)}, \diag(1_U, 1_U) \cdot \eta \colon U \ \text{is a clopen set of} \ X \}$, with $\eta \in M_2(C(\T^2))$ being the projection in $C(\T^2)$.

	To show that $j_{*0}$ is injective, we need to show that whenever $p, q \in M_{\infty}(A_x)$ are projections such that $j_{*0}([p]) = j_{*0}([q])$ in $K_0(A)$, we have $[p] = [q]$ in $K_0(A_x)$.
	For projections $p, q \in M_{\infty}(A_x)$, we can find $n \in \N$ and projections $e, f \in M_{\infty}(A_n)$ such that $[e] = [p]$ and $[f] = [q]$ in $K_0(A_x)$.
%	Consider $\rho_n([p])$ with $\rho_n$ as in Lemma \ref{K_0 of A_n}, which is a postive element in $C(X, \Z^2) / \{ f - f \circ \af^{-1} \colon f \in C(X, \Z^2), f \left|_{Y_n} \right. = 0 \}$ because $\rho_n$ is order preserving by Lemma \ref{K_0 of A_n}. Then we can find $h \in C(X, D)$ with $D$ as in the statement of Lemma \ref{K_0 of C(T^2)} such that $\pi_n(h) = \rho_n$ with $\pi_n$ as in Lemma \ref{commutativity of identification}.	
%	According to Lemma \ref{K_0 of C(X T^2)}, there exists $p' \in M_{\infty}(C(X \times \T^2))$ such that $\iota([p']) = h$. By Lemma \ref{commutativity of identification}, it follows that $(j_n)_{*0}([p'_n]) = [p_n]$ in $K_0(A_n)$.
	According to Corollary \ref{representation of K_0(A_n)}, we can find $e', f' \in M_{\infty}(C(X \times \T^2))$ such that $[e'] = [e]$ and $[f'] = [f]$ in $K_0(A_n)$.
	We need to show that if $j_{*0}([p]) = j_{*0}([q])$ in $K_0(A)$, then $[p] = [q]$ in $K_0(A_x)$.
	In fact, if $j_{*0}([p] - [q]) = 0$, we have $j_{*0}([p]) = j_{*0}([q])$, which implies that $j_{*0}([e']) = j_{*0}([f'])$ in $K_0(A)$.
	
	The Pimsner-Voiculescu six-term exact sequence in our situation reads as follows:

	$$
	\xymatrix{
		K_0(C(X \times \T^2)) \ar@{->}[rr]^{\id_{*0} - \af_{*0}} && K_0(C(X \times \T^2)) \ar@{->}[rr]^{j_{*0}} && K_0(A) \ar@{->}[d] \\
		K_1(A) \ar@{->}[u] && K_1(C(X \times \T^2)) \ar@{->}[ll]_{j_{*1}} && K_1(C(X \times \T^2)) \ar@{->}[ll]_{\id_{*1} - \af_{*1}}
	}
	$$
	
	As $j_{*0}([p'_n]) = j_{*0}([q'_n])$, by the exact sequence above, $[p'_n] - [q'_n]$ is in the image of $(\id_{*0} - \af_{*0})$. That is, there exists $x$ in $K_0(C(X \times \T^2))$ such that $[p'_n] - [q'_n] = x - \af_{*0}(x)$. Apply $\iota$ as defined in Lemma \ref{K_0 of C(X T^2)} on both sides. We get
	$$ \iota([p'_n]) - \iota([q'_n]) = \iota(x) - \iota(\af_{*0}(x)) \ \text{in} \ C(X, \Z^2) . $$
	Note that $\iota(\af_{*0}(x)) = \iota(x) \circ \af$. We get $\iota([p'_n]) - \iota([q'_n]) = (-\iota(x) \circ \af) - (-\iota(x) \circ \af) \circ \af^{-1}$. We can choose $N \in \N$ such that for all $k > N$, $(-\iota(x) \circ \af)$ restricted to $Y_k$ will be a constant function, say $c \in \Z^2$. It is clear that
	$$ \iota([p'_n]) - \iota([q'_n]) = (-\iota(x) \circ \af - c) - (-\iota(x) \circ \af - c) \circ \af^{-1} . $$
	
	Choose $m \in \N$ such that $m > \max(n, N)$. Then $(-\iota(x) \circ \af - c) \left|_{Y_m} \right. = 0$. According to Corollary \ref{handy technical result}, we have $(j_m)_{*0}([p'_n]) = (j_m)_{*0}([q'_n])$ with $j_m$ as in Lemma \ref{commutativity of identification}.
	
	We have shown that $[p'_n] = [q'_n]$ in $K_0(A_m)$. Note that $[p'_n] = [p_n]$ and $[q'_n] = [q_n]$ in $K_0(A_n)$ and $m > n$. It follows that $[p'_n] = [p_n]$ and $[q'_n] = [q_n]$ in $K_0(A_m)$. We then have that $[p_n] = [q_n]$ in $K_0(A_m)$, so that $[p_n] = [q_n]$ in $K_0(A_x)$.
	
	Note that $[p_n] = [p]$ and $[q_n] = [q]$ in $K_0(A_x)$. It then follows that $[p] = [q]$ in $K_0(A_x)$, which finishes the proof.
	%As $p'_n, q'_n$ are projections in $M_{\infty}(C(X \times \T^2))$ such that $(j_n)_{*0}([p'_n]) = (j_n)_{*0}([q'_n])$ in $K_0(A_n)$, for the $m > n$, we have $(j_m)_{*0}([p'_n]) = (j_m)_{*0}([q'_n])$ in $K_0(A_m)$, with $j_m$ being the canonical embedding from $C(X \times \T^2)$ to $A_m$.
	%Now we have two projections $p'_n, q'_n$ in $M_{\infty}(C(X \times \T^2))$ satisfying $\iota([p'_n]) - \iota([q'_n]) = (-\iota(x) \circ \af - c) - (-\iota(x) \circ \af - c) \circ \af^{-1}$
\end{proof}	

\end{lem}

\begin{lem} \label{injectivity of K_1}

 	Let $A$ be $C^*(\Z, X \times \T \times \T, \af \times \mathrm{R}_{\xi} \times \mathrm{R}_{\eta} )$ and let $A_x$ be as in Definition \ref{dfn of Ax}. Let $j \colon A_x \rightarrow A$ be the canonical embedding. Then $j_{*1}$ is an injective homomorphism from $K_1(A_x)$ to $K_1(A)$.

\begin{proof}
	The proof is similar to the proof of Lemma \ref{injectivity of K_0}.

	For any two unitaries $x, y \in A_x$ such that $j_{*1}([x]) = j_{*1}([y])$ in $K_1(A)$, we need to show that $[x] = [y]$. For $x$, $y$ as above, we can find $n \in \N$ and $x', y' \in M_{\infty}(A_n)$ such that $[x] = [x']$ and $[y] = [y']$ in $K_1(A_x)$.

	%QQQQ Fill in more details?
	From Lemma \ref{structure of A_n}, we get the structure of $A_n$, which then implies the fact that
	$$ K_1(A_n) \cong C(X, \Z^2) / \{ f - f \circ \af^{-1} \colon f \in C(X, \Z^2) \ \text{and} \ f \left|_{Y_n} \right. = 0 \} . $$
%	\removeme[remove this fact b/c it is not needed, although it is true]Note that $A_x = \lim_{n \rightarrow \infty} A_n$. We can show that $K_1(A_x) \cong C(X, \Z^2) / \{ f - f \circ \af^{-1} \colon f \in C(X, \Z^2) \}$}.
	Similar to the analysis of the Pimsner-Voiculescu six-term exact sequence as in the proof of Lemma \ref{injectivity of K_0}, we get $[x'] = [y']$ in $K_1(A_m)$ for $m$ large enough. It then follows that $[x'] = [y']$ in $K_1(A_x)$, which implies that $[x] = [y]$ in $K_1(A_x)$.
\end{proof}

\end{lem}

	The following result is a known fact, and it is used later to show approximate unitary equivalence.

\begin{prp} \label{a handy known fact}

	Let $A$ be an infinite dimensional simple unital AF algebra and let $CU(A)$ be as in Section \ref{Sec Introduction and Notation}. Then $U(A) = CU(A)$.

\begin{proof}
	For every unitary $u \in A$ and every $\ep > 0$, we will show that $\dist(u, CU(A)) < \ep$.

	As $A$ is unital and infinite dimensional, we can assume that $A \cong \displaystyle\varinjlim A_n$ with each $A_n$ being a finite dimensional \ca \ and  each map $j_{n, n + 1} \colon A_n \hookrightarrow A_{n + 1}$ being unital.  Write
	$$ A_n \cong \bigoplus_{k = 1}^{s_n} M_{d_{n; k}}(\C) $$
	with $d_{n; 1} \leq d_{n; 2} \leq \cdots \leq d_{n; s_n}$.
	%QQQQ [ I think it is not needed. ] We can also assume that $d_{n; s_n} \rightarrow \infty$.

	Let $d_n' = \min \{d_{n; s_1}, \ldots, d_{n; s_n}\}$. As $A$ is simple, we have $\lim_{n \rightarrow \infty} d_n' = \infty$.

	For $u$ and $\ep$ as given above, we can choose $n$ large enough such that $d'_{n} > \frac{2 \pi}{\ep}$ and there exists $v \in U(A_n)$ satisfying $\norm{u - v} < \ep / 2$. Let $\pi_{n; k}$ be the canonical projection from $A_n$ to $M_{d_{n; k}}(\C)$. {It is known that for any $w \in U(A)$, we have $w \in CU(A_n)$ if and only if $\det(\pi_{n; k}(w)) = 1$ for $k = 1, \ldots, s_n$}. Without loss of generality, we can assume that
	$$ \pi_{n; k}(u_n) = \diag(\lambda_{k, 1}, \ldots, \lambda_{k, d_{n; k}}) , \ \text{with} \ | \lambda_{k, d_{n; i}} | = 1 . $$
	Choose $L_k$ such that $- \pi \leq L_k < \pi$ and $\det(\pi_{n; k}(u_n)) = e^{i L_k}$. For $k = 1, \ldots, s_n$, define
	$$ v_{k}' = \diag (\lambda_{k, 1} \cdot e^{- i L / d_{n; k}}, \ldots, \lambda_{k, d_{n; k}} \cdot e^{- i L / d_{n; k}}) . $$
	
	Let $v'= \diag(v'_{1}, \ldots, v'_{s_n})$. It is then clear that $\norm{u_n - u_n'} \leq \pi / d'_{n} $. It is easy to check that $\det(\pi_{n; s_k}(v')) = 1$ for all $k = 1, \ldots, s_n$, which then implies that $v' \in CU(A_n) \subset CU(A)$.

	Note that $d'_{n} > \frac{2 \pi}{\ep}$. We have
	%$$
	%\begin{array}{lll}
	\begin{align*}
	\dist(u, CU(A)) & \leq \norm{u - v'} \\
		& \leq \norm{u - v} + \norm{v - v'} \\
		& \leq \ep /2 + \ep /2 \\
		& = \ep .
	\end{align*}
	%\end{array}
	%$$
	As $\ep$ can be chosen to be arbitrarily small, it follows that $u \in CU(A)$.
\end{proof}

\end{prp}

	We will need the fact that a cut-down of the crossed product \ca \ by a projection in $C(X)$ is similar to the original crossed product \ca, and can be regarded as a crossed product \ca \ of the induced action. Some definitions and facts will be given here.

	Let $(X \times \T \times \T, \af \times \mathrm{R}_{\xi} \times \mathrm{R}_{\eta})$ be a minimal topological dynamical system as defined in Section \ref{Sec Introduction and Notation}. Let $D$ be a clopen subset of $X$, and let $x \in D$. For simplicity, we use $\varphi$ to denote $\af \times \mathrm{R}_{\xi} \times \mathrm{R}_{\eta}$.

	Define $\widetilde{\varphi} \colon D \times \T \times \T \rightarrow D \times \T \times \T$ by $\widetilde{\varphi}((y, t_1, t_2)) = \varphi^{f(x)} ((y, t_1, t_2))$, where $f(x)$ is the first return time function defined by
	$$ f(x) = \min \{ n \in \N \colon n > 0, \varphi^n(x) \in U \} . $$
	
	As $\varphi$ is minimal on $X \times \T \times \T$, for every $x \in X$, the orbit of $x$ under $\varphi$ is dense in $X$. It then follows that the intersection of this orbit with $D$ is dense in $D$, which implies that $\widetilde{\varphi}$ is also minimal on $D \times \T \times \T$. As the composition of rotations on the circle is still a rotation on the circle, we can find maps $\widetilde{\xi} , \widetilde{\eta} \colon D \rightarrow \T$ such that $\widetilde{\varphi} = \widetilde{\af} \times \mathrm{R}_{\widetilde{\xi}} \times \mathrm{R}_{\widetilde{\eta}}$ with $\widetilde{\af}(x) = \af^{f(x)}(x)$ for $f$ as defined above.
	
	It follows that $\widetilde{\xi}$ and $\widetilde{\eta}$ are both continuous functions. In fact, as $D$ is clopen, we have that $f$ is continuous, which then implies that $\widetilde{\xi}$ and $\widetilde{\eta}$ are continuous.

%\begin{sloppypar}
	As $(D \times \T \times \T, \widetilde{\varphi})$ is a minimal dynamical system, the corresponding crossed product \ca \ $C^*(\Z, D \times \T \times \T, \widetilde{\varphi})$ is simple. Use $\widetilde{u}$ to denote the implementing unitary in $C^*(\Z, D \times \T \times \T, \widetilde{\varphi})$.
%\end{sloppypar}

	Define $\widetilde{A_x}$ to be the subalgebra of $C^*(\Z, D \times \T \times \T, \widetilde{\varphi})$ generated by $C(D \times \T \times \T)$ and $\widetilde{u} \cdot C_0( ( D \backslash \{ x \} ) \times \T \times \T)$.

	The lemma below shows that the cut down of the original crossed product \ca \ is isomorphic to the crossed product \ca \ of the induced homeomorphism.

\begin{lem} \label{identification of cut-down}

	Let $\varphi$ and $\widetilde{\varphi}$ be defined as above. There is a \ca \ isomorphism from $C^*(\Z, D \times \T \times \T, \widetilde{\varphi})$ to $1_{D \times \T \times \T} \cdot A \cdot 1_{D \times \T \times \T}$.

\begin{proof}
	Let $f \colon D \rightarrow \N$ be the first return time function. As $D$ is clopen, $f$ is continuous. As $X$ is compact and $D$ is closed in $X$, $D$ is also compact. Continuity of $f$ then implies that $f(D)$ is a compact set, that is, a finite subset of $\N$. Write $f(D) = \{ k_1, \ldots, k_N \}$ with $N, k_1, \ldots, k_N \in \N$ and set $D_i = f^{-1}(k_i)$.

	In $1_{D \times \T \times \T} \cdot A \cdot 1_{D \times \T \times \T}$,
	% QQQQ: Why w \in 1_{D \times \T \times \T} \cdot A \cdot 1_{D \times \T \times \T} ? I think it is true, but we need more details.
	let $w = \sum_{i = 1}^N 1_{D_i \times \T \times \T} \cdot u^{k_i}$. Then we have
	\begin{align*}
	w w^*   & = \left( \displaystyle\sum_{i = 1}^N 1_{D_i \times \T \times \T} \cdot u^{k_i} \right) \cdot \left( \sum_{i = 1}^N 1_{D_i \times \T \times \T} \cdot u^{k_i} \right)^* \vspace{2mm} \\
                & = \displaystyle \left( \sum_{i = 1}^N 1_{D_i \times \T \times \T} \cdot u^{k_i} \right) \cdot \left( \sum_{j = 1}^N  u^{- k_j} 1_{D_j \times \T \times \T} \right) \vspace{2mm} \\
                & = \displaystyle \sum_{i, j = 1}^N 1_{D_i \times \T \times \T} \cdot u^{k_i} \cdot u^{- k_j}1_{D_j \times \T \times \T} \vspace{2mm} \\
                & = \displaystyle \sum_{i, j = 1}^N 1_{D_i \times \T \times \T} \cdot u^{k_i - k_j} \cdot 1_{D_j \times \T \times \T} \vspace{2mm} \\
                & = \displaystyle \sum_{i, j = 1}^N 1_{D_i \times \T \times \T} \cdot (1_{D_j \times \T \times \T} \circ (\af \times \mathrm{R}_{\xi} \times \mathrm{R}_{\eta})^{k_i - k_j}) \cdot u^{k_i - k_j}.
	\end{align*}
	We need the following claim to get that $w w^* = 1_D$.

\claim For $D_i, k_i$ as above,
	$$ 1_{D_i \times \T \times \T} \cdot ( 1_{D_j \times \T \times \T} \circ (\af \times \mathrm{R}_{\xi} \times \mathrm{R}_{\eta})^{k_i - k_j} ) = \left\{ \begin{array}{lr} 1_{D_i \times \T \times \T} & i = j \\ 0 & i \neq j \end{array} \right. . $$

Proof of claim:

	If $k_j > k_i$, then $\af^{k_j - k_i}(D_j) \subset X \setminus D$. Thus $D_i \cap \af^{k_j - k_i}(D_j) = \varnothing$.

	If $k_j < k_i$, we claim that $D_i \cap \af^{k_j - k_i}(D_j) = \varnothing$. If not, choose $s \in D_i \cap \af^{k_j - k_i}(D_j)$. We can assume $s = \af^{k_j - k_i}(y)$ for some $y \in D_j$. It is then clear that $\af^{k_i - k_j}(s) = y \in D_j \subset D$, contradicting the fact that the first return time of $s$ (in $D_i$) is $k_i$.

	If $k_j = k_i$, it is clear that $1_{D_i} \cdot (1_{D_j} \circ \af^{k_i - k_j}) = 1_{D_i}$.

	This proves the claim.

	Using the claim, we get
	\begin{align*}
	    w w^*   & = \displaystyle\sum_{i, j = 1}^N 1_{D_i \times \T \times \T} \cdot \left( 1_{D_j \times \T \times \T} \circ (\af \times \mathrm{R}_{\xi} \times \mathrm{R}_{\eta})^{k_i - k_j} \right) \cdot u^{k_i - k_j} \vspace{2mm} \\
	                & = \displaystyle\sum_{i = 1}^N 1_{D_i \times \T \times \T} \\
	                & = 1_{D \times \T \times \T} .
	\end{align*}
	Now we calculate $w^* w$. It is clear that
	\begin{align*}
	    w^* w   & = \displaystyle \left( \sum_{i = 1}^N 1_{D_i \times \T \times \T} \cdot u^{k_i} \right)^* \cdot \left( \sum_{i = 1}^N 1_{D_i \times \T \times \T} \cdot u^{k_i} \right) \vspace{2mm} \\
	                & = \displaystyle \left( \sum_{j = 1}^N  u^{- k_j} \cdot 1_{D_j \times \T \times \T} \right) \cdot \left( \sum_{i = 1}^N 1_{D_i \times \T \times \T} \cdot u^{k_i} \right) \vspace{2mm} \\
	                & = \displaystyle \sum_{i, j = 1}^N u^{- k_j} \cdot 1_{D_j \times \T \times \T} \cdot 1_{D_i \times \T \times \T} \cdot u^{k_i} \vspace{2mm} \\
	                & = \displaystyle \sum_{i = 1}^N u^{-k_i} \cdot 1_{D_i \times \T \times \T} \cdot u^{k_i} \vspace{2mm} \\
	                & = \displaystyle \sum_{i = 1}^N 1_{D_i \times \T \times \T} \circ (\af \times \mathrm{R}_{\xi} \times \mathrm{R}_{\eta})^{-k_i} \vspace{2mm} \\
	                & = \displaystyle \sum_{i = 1}^N 1_{(\af \times \mathrm{R}_{\xi} \times \mathrm{R}_{\eta})^{k_i} (D_i \times \T \times \T)} \vspace{2mm} \\
	                & = \displaystyle \sum_{i = 1}^N 1_{\widetilde{\ph}(D_i \times \T \times \T)} \vspace{2mm} \\
	                & = 1_{D \times \T \times \T} .
	\end{align*}
	So far, we have shown that $w$ is a unitary in $1_{D \times \T \times \T} \cdot A \cdot 1_{D \times \T \times \T}$.

	Define a map
	$$ \gamma \colon C^*(\Z, D \times \T \times \T, \widetilde{\varphi}) \longrightarrow 1_{D \times \T \times \T} \cdot A \cdot 1_{D \times \T \times \T} $$
	by
	$$ \gamma(f) = f \ \text{for all} \ f \in C(D \times \T \times \T) \ \text{and} \ \gamma(\widetilde{u}) = w . $$
	
	\removeme[more details shall be filled in]{We will check that $\gamma$ is well-defined and gives the desired isomorphism between $C^*(\Z, D \times \T \times \T, \widetilde{\varphi})$ and $1_{D \times \T \times \T} \cdot A \cdot 1_{D \times \T \times \T}$.} In fact, for all $f \in C( D \times \T \times \T )$, we have
	\begin{align*}
	\gamma( \widetilde{u}^* f \widetilde{u} ) & = \gamma ( f \circ \widetilde{\ph}^{-1} ) \\
		& = f \circ \widetilde{\ph}^{-1} .
	\end{align*}
	We also have
	\begin{align*}
	\gamma( \widetilde{u}^* f \widetilde{u} ) & = \gamma(\widetilde{u}^*) \cdot \gamma(f) \cdot \gamma(\widetilde{u}) \\
		& = w^* \cdot f \cdot w \\
		& = \displaystyle \left( \sum_{j = 1}^N 1_{D_j} \cdot u^{k_j} \right)^* \cdot \left( f \cdot \sum_{i = 1}^N 1_{D_i} \right) \cdot \left( \sum_{l = 1}^N 1_{D_l} \cdot u^{k_l} \right) \vspace{1mm} \\
		& = \displaystyle \sum_{i, j, k = 1}^N u^{- k_j} \cdot 1_{D_j} \cdot f \cdot 1_{D_i} \cdot 1_{D_l} \cdot u^{k_l} \vspace{1mm} \\
		& = \displaystyle \sum_{i = 1}^N u^{- k_i} \cdot (f \cdot 1_{D_i}) \cdot u^{k_i} \\
		& = f \circ \widetilde{\ph}^{-1} ,
	\end{align*}
	which then implies that $\gamma$ is really a homomorphism.

	To show that $\gamma$ is surjective, we will show that for every $g \in C(X \times \T \times \T)$ and $n \in \N$, $1_{D \times \T \times \T} \cdot (g u^n) \cdot 1_{D \times \T \times \T}$ is in the image of $\gamma$. Note that
	\begin{align*}
	1_{D \times \T \times \T} \cdot (g u^n) \cdot 1_{D \times \T \times \T} & = ( 1_{D \times \T \times \T} \cdot g) \cdot (u^n \cdot 1_{D \times \T \times \T} ) \\
		& = (1_{D \times \T \times \T} \cdot g \cdot 1_{\af^{-n}(D) \times \T \times \T}) \cdot u^n .
	\end{align*}
	Without loss of generality, we assume that $$ D \cap \af^{-n} ( D ) \neq \varnothing . $$
	Note that there is $s$ with $1 \leq s \leq N$ such that $D \cap \af^{-n} ( D ) = D_s$, $n = k_s$ and $D_s$ is exactly $f^{-1} ( n )$. It follows that
	\begin{align*}
	1_{D \times \T \times \T} \cdot (g u^n) \cdot 1_{D \times \T \times \T} & = (g \cdot 1_{D_n \times \T \times \T} ) \cdot u^n \\
		& = (g \cdot 1_{D_n \times \T \times \T} ) \cdot ( 1_{D_n \times \T \times \T} \cdot u^n) .
	\end{align*}
	It is clear that we can identify $g \cdot 1_{D_n \times \T \times \T}$ with a function in $C(D \times \T \times \T)$. Note that $w = \sum_{i = 1}^N 1_{D_i \times \T \times \T} \cdot u^{k_i}$. We have
	\begin{align*}
	\gamma \left( (g \cdot 1_{D_s \times \T \times \T}) \cdot (\widetilde{u}) \right) & = \gamma \left( (g \cdot 1_{D_n \times \T \times \T}) \right) \cdot \gamma( \widetilde{u} ) \\
		& = (g \cdot 1_{D_s \times \T \times \T}) \cdot \left( \sum_{i = 1}^N 1_{D_i \times \T \times \T} \cdot u^{k_i} \right) \\
		& = g \cdot 1_{D_s \times \T \times \T} \cdot u^{k_s} \\
		& = g \cdot 1_{D_s \times \T \times \T} \cdot u^{n} \\
		& = 1_{D \times \T \times \T} \cdot (g u^n) \cdot 1_{D \times \T \times \T} .
	\end{align*}
	Then we have proved that $\gamma$ is surjective. As $C^*(\Z, D \times \T \times \T, \widetilde{\varphi})$ is a simple \ca, it follows that $\gamma$ is a \ca \ isomorphism. \qedhere
	
	%QQQQ: The following is obsolete.
		
%	It is also clear that we can regard $1_{D_i \times \T \times \T} \cdot g \cdot 1_{\af^{-n}(D) \times \T \times \T}$ as a function in $C(D \times \T \times \T)$. Then we have
%	$$ \gamma(1_{D_i \times \T \times \T} \cdot g) \cdot (1_{\af^{-n}(D) \times \T \times \T} \cdot ) $$	
	
	%QQQQ: Finish it! The commented approach is not right, though.
	
%	We will check that $\gamma$ maps $\widetilde{A_x}$ onto $1_{D \times \T \times \T} \cdot A_x \cdot 1_{D \times \T \times \T}$. In fact, we just need to show that
%	$$ \ph( C( D \times \T \times \T ) ) = \ 1_{D \times \T \times \T} \cdot C( D \times \T \times \T ) \cdot 1_{D \times \T \times \T} $$
%	and
%	$$ \ph( \widetilde{u} \cdot C_0( ( D \backslash \{ x \} ) \times \T \times \T) ) = 1_{D \times \T \times \T} \cdot ( u \cdot C_0( ( D \backslash \{ x \} ) \times \T \times \T) ) \cdot 1_{D \times \T \times \T} . $$
\end{proof}

\end{lem}

\vspace{2mm}

	The idea of topological full group of the Cantor set is needed in the next lemma, and a definition is given below.

\vspace{2mm}

\begin{dfn}

	Let $X$ be the Cantor set and let $\af$ be a minimal homeomorphism of $X$. We say that $\bt \in \text{Homeo}(X)$ is in the full group of $\af$ if $\bt$ preserves the orbit of $\af$. That is, for any $x \in X$, $\bt( \{ \af^n(x) \} _ {n \in \Z} ) = \{ \af^n(x) \} _ {n \in \Z}$. In this case, there exists a unique function $n \colon X \rightarrow \Z$ such that $\bt(x) = \af^{n(x)}(x)$ for all $x \in X$.
	
	We say that $\bt \in \text{Homeo}(X)$ is in the topological full group of $\af$ if the function $n$ above is continuous.

	We use $[ \af ]$ to denote the full group of $\af$, and use $[[ \af ]]$ to denote the topological full group of $\af$.

\end{dfn}

\begin{lem} \label{extension of element in topological full group}

	Let $X$ be the Cantor set and let $\af$ be a minimal homeomorphism of $X$. Let $Y$ and $U$ be two clopen subsets of $X$ such that $U \subset Y$. If there exists $\beta \in [[\af]]$ such that $\bt(U) \subset Y$ and $U \cap \bt(U) = \varnothing$, then there exists $\gamma \in [[\af]]$ such that $\gamma(Y) = Y$, $\gamma \left|_{U} \right. = \bt \left|_{U} \right.$ and $\gamma \left|_{X \backslash Y} \right. = \id \left|_{X \backslash Y} \right.$.

\begin{proof}
	As $\bt \in [[\af]]$, there exists a continuous function $n_1 \colon X \rightarrow \Z$ such that $\bt(x) = \af^{n_1(x)}(x)$ for all $x \in X$. Let $U_j = U \cap n_1^{-1}(j)$ for $j \in \Z$. As the sets $n_1^{-1}(j)$ are  mutually disjoint for $j \in \Z$, so are the sets $U_j$ . Now we have $\bt(U) = \bigsqcup_{j = - \infty}^{\infty} \af^{j} (D_j)$.

	Define $\gamma \in \text{Homeo}(X)$ by $\gamma(x) = \af^{n_2(x)}(x)$, with
	$$
		n_2(x) = \left \{ \begin{array}{ll} n_1(x) &  x \in U \\ - j & x \in \af^{j}(U_j) \\ 0 & x \notin U \ \text{and} \ x \notin \bt(U) \end{array} \right. _.
	$$
%	Note that $U$ and $\af^j(U_j)$ (with $j \in \Z$) might have non-empty intersections. We will show that $n_2$ is still well-defined. In fact, if $x \in U \cap \af^j(U_j)$, then $\af^j(x) = x$
	As $U \cap \bt(U) = \varnothing$, we get $U \cap \af^j(U_j) = \varnothing$ for all $j \in \Z$. Thus $n_2$ is a well-defined function. Then we can check that $\gamma \left|_{U} \right. = \bt \left|_{U} \right.$ as $n_1 \left|_{U} \right. = n_2 \left|_{U} \right.$. It is also obvious that $\gamma(\bt(U)) = U$ and $\gamma \left|_{Y \backslash (U \sqcup \bt(U))} \right. = \id_{Y \backslash (U \sqcup \bt(U))}$. It follows that $\gamma(Y) = Y$.  As $n_2(x) = 0$ when $x \notin Y$, we get $\gamma \left|_{X \backslash Y} \right. = \id \left|_{X \backslash Y} \right.$.
\end{proof}

\end{lem}

\vspace{1cm}

\begin{lem} \label{partial isometry in the desired form}

	Let $X$ be the Cantor set. Let $\af$ be a minimal homeomorphism of $X$, and let $x \in X$. Let $A$ be the crossed product \ca \ of the dynamical system $(X, \af)$. Use $A_x$ to denote the subalgebra generated by $C(X)$ and $u \cdot C_0(X \backslash \{ x \})$. Let $D$ be a clopen subset of $X$ and let $n \in \N$ be such that $x \notin \bigcup_{k = 0}^{n - 1} \af^k(D)$. In $A_x$, the element $s = u \cdot 1_{\af^{n - 1}}(D) \cdots u \cdot 1_{\af(D)} \cdot u \cdot 1_D $ is a partial isometry such that $s^* s = 1_D$ and $s s^* = 1_{\af^n(D)}$.

\begin{proof}
	We just need to check $s s^* = 1_{\af^n(D)}$, $s^* s = 1_D$, and $s \in A_x$.

	In fact,
	%$$
	%\begin{array}{lll}
	\begin{align*}
	s s^*	& = (u \cdot 1_{\af^{n - 1}}(D) \cdots u \cdot 1_{\af(D)} \cdot u \cdot 1_D) \cdot (u \cdot 1_{\af^{n - 1}}(D) \cdots u \cdot 1_{\af(D)} \cdot u \cdot 1_D)^* \\
		& = u \cdot 1_{\af^{n - 1}}(D) \cdots u \cdot 1_{\af(D)} \cdot u \cdot 1_D \cdot 1_D \cdot u^* \cdot 1_{\af(D)} \cdot u^* \cdots 1_{\af^{n - 1}(D)} \cdot u^* \\
		& = 1_{\af^n(D)} ,
	\end{align*}
	%\end{array}
	%$$
	and
	%$$
	%\begin{array}{lll}
	\begin{align*}
	s^* s	& = (u \cdot 1_{\af^{n - 1}}(D) \cdots u \cdot 1_{\af(D)} \cdot u \cdot 1_D)^* \cdot (u \cdot 1_{\af^{n - 1}}(D) \cdots u \cdot 1_{\af(D)} \cdot u \cdot 1_D) \\
		& = 1_D \cdot u^* \cdot 1_{\af(D)} \cdot u^* \cdots 1_{\af^{n - 1}(D)} \cdot u^* \cdot u \cdot 1_{\af^{n - 1}}(D) \cdots u \cdot 1_{\af(D)} \cdot u \cdot 1_D \\
		& = 1_D .
	\end{align*}
	%\end{array}
	%$$

	As $x \notin \bigcup_{k = 0}^{n - 1} \af^k(D)$, it follows that $u \cdot 1_{\af^{k}(D)} \in A_x$ for $k = 0, \ldots, n -1$. Thus $s, s^* \in A_x$.
\end{proof}

\end{lem}

\twomm

\noindent \textbf{Remark:} It is easy to check that $s = u^{n} \cdot 1_D$ and $s^* = (u^n \cdot 1_D)^* = 1_D \cdot u^{-n}$.

\onecm

\begin{lem} \label{topological full group element to desired form}

	Let $X$ be the Cantor set and let $\af$ be a minimal homeomorphism of $X$. Let $u$ be the implementing unitary of the crossed product \ca \ $C^*(\Z, X, \af)$. For $\gamma \in [[ \af ]]$, there exist mutually disjoint clopen sets $X_1$, $\ldots$, $X_N$ and $n_1$, $\ldots$, $n_N \in \N$ such that $X = \bigsqcup_{i = 1}^N X_i$ and $\gamma(x) = \af^{n_i}(x)$ for $x \in X_i$. Furthermore, $w = \displaystyle\sum_{i \in \N} 1_{X_i} \cdot u^{n_i}$ is a unitary element in $C^*(\Z, X, \af)$ satisfying $w^* f w = f \circ \gamma^{-1}$ for all $f \in C(X)$.

\begin{proof}
	As $\gamma \in [[ \af ]]$, there exists a continuous function $n \colon X \rightarrow \Z$ such that $\gamma(x) = \af^{n(x)}(x)$ for all $x \in X$. As $X$ is compact and $n$ is continuous, the range $n(X)$ must be finite.

	Define
	$$ w = \sum_{k \in n(X)} 1_{Y_k} \cdot u^{k} $$
	where $Y_k = n^{-1}( k )$. As $n(X)$ is finite, we have finitely many sets $Y_k$. As $\gamma$ is a homeomorphism, it follows that $\af^{k}(Y_k) \cap \af^{j}(Y_j) = \varnothing$ if $k \neq j$.

	We will check that $w w^* = 1$ and $w^* w = 1$.
	
	Note that
	%$$
	%\begin{array}{lll}
	\begin{align*}
		w w^*	& = \displaystyle \left( \sum 1_{Y_k} \cdot u^{k}) (\sum 1_{Y_j} \cdot u^{j} \right)^* \\
			& = \displaystyle \sum_{k, j \in \Z} 1_{Y_k} \cdot u^{k} \cdot u^{-j} \cdot 1_{Y_j} \\
			& = \displaystyle \sum_{k, j \in \Z} 1_{Y_k} \cdot \left( 1_{Y_j} \circ \af^{k - j} \right) \cdot u^{k - j} \\
			& = \displaystyle \sum_{k, j \in \Z} 1_{Y_k} \cdot 1_{\af^{j - k}(Y_j)} \cdot u^{k - j} .
	\end{align*}
	%\end{array}
	%$$

	As $\af^{k}(Y_k) \cap \af^{j}(Y_j) = \varnothing$ if $k \neq j$, it follows that $\af^{j - k}(Y_j) \cap Y_k = \varnothing$ if $k \neq j$. Then we get
	\begin{align*}
	w w^* & = \displaystyle \sum_{k, j \in \Z} 1_{Y_k} \cdot 1_{\af^{j - k}(Y_j)} \cdot u^{k - j} \\
		& = \displaystyle \sum_{k} 1_{Y_k} \\
		& = 1 \ .
	\end{align*}
%	As for $w^* w$, we have
%	\begin{align*}
%		w^* w & = \displaystyle \left( \sum_{j \in \Z} 1_{Y_j} \cdot u^{j} \right)^* \cdot \left( \sum_{k \in \Z} 1_{Y_k} \cdot u^{k} \right) \\
%			& = \displaystyle \sum_{j, k \in \Z} u^{-j} \cdot 1_{Y_j} \cdot 1_{Y_k} \cdot u^k \\
%			& = \displaystyle \sum_{j \in \Z} u^{-j} 1_{Y_j} u^j \\
%			& = \displaystyle \sum_{j \in \Z} 1_{Y_j} \circ \af^{-j} \\
%			& = \displaystyle \sum_{j \in \Z} 1_{\af^{j}(1_{Y_j})} \\
%			& = \displaystyle \sum_{j \in \Z} 1_{\gamma(Y_j)} .
%	\end{align*}
%	As $\gamma$ is a homeomorphism of $X$, it follows that $w^* w = 1$.
	As $C^*(\Z, X, \af)$ has stable rank one, it is finite. It then follows that $w^* w = 1$. So far, we have shown that $w$ is a unitary element in $C^*(\Z, X, \af)$.

	To show that $w^* f w = f \circ \gamma^{-1}$, we just need to show that for each $i$ and for every clopen set $D \subset Y_i$, we have $w^* 1_{D} w = 1_{D} \circ \gamma^{-1}$. As $C(X)$ is generated by
	$$ \{ 1_{D} \colon D \ \text{is} \ \text{a clopen set of} \ Y_i \ \text{for some} \ i \in \Z \} , $$
	that will imply $w^* f w = f \circ \gamma^{-1}$ for all $f \in C(X)$.

	For a clopen set $D \subset Y_i$, it is clear that
%	$$
%	\begin{array}{lll}
	\begin{align*}
		w^* 1_{D} w & = \displaystyle \left( \sum_{j \in \Z} 1_{Y_j} \cdot u^{j} \right)^* \cdot 1_{D} \cdot \left( \sum_{k \in \Z} 1_{Y_k} \cdot u^{k} \right) \vspace{2mm} \\
				& = \displaystyle \sum_{j, k \in \Z} u^{-j} \cdot 1_{Y_j} \cdot 1_{D} \cdot 1_{Y_k} \cdot u^k \\
				& = u^{-i} \cdot 1_{D} \cdot u^i \\
				& = 1_{D} \circ \af^{-i} \\
				& = 1_{D} \circ \gamma^{-1} \ ,
	\end{align*}
%	\end{array}
%	$$
	which finishes the proof.
\end{proof}

\end{lem}

\vspace{5mm}

%QQQ: Really needed?
	{Some facts about Cantor dynamical systems that will be needed are given below.}

\vspace{5mm}

\begin{lem} \label{unitary equivalence in Ax, weaker version}

	Let $(X, \af)$ be a minimal Cantor dynamical system and let $x \in X$. Let $U$ and $V$ be two clopen subsets of $X$. Let $A$ be the crossed product \ca \ of $(X, \af)$ and let $A_x$ be the subalgebra generated by $C(X)$ and $u \cdot C_0(X \backslash \{ x \})$, with $u$ being the implementing unitary element in $A$ satisfying $u f u^* = f \circ \af^{-1}$ for all $f \in C(X)$. If there exists an integer $n \geq 1$ such that $\af^n(U) = V$ and $x \notin \bigcup_{k = 0}^{n - 1} \af^k(U)$, then there exists $w \in A_x$ such that $w \cdot 1_U \cdot w^* = 1_V$.

\end{lem}

\begin{proof}
	As $x \notin \bigcup_{k = 0}^{n - 1} \af^k(U)$, we can find a Kakutani-Rokhlin partition $\CalP$ of $X$ with respect to $\af$ such that the roof set $R(\CalP)$ is a clopen set containing $x$ and $R(\CalP) \cap ( \bigcup_{k = 0}^{n - 1} \af^k(U) )= \varnothing$.

	Write
	$$ \displaystyle \CalP = \bigsqcup_{\substack{1 \leq s \leq N \\ 1 \leq k \leq h(s)}} X(s, k) $$
	with $\af(X(s, k)) = X(s, k + 1)$ for all $k = 1, \ldots, h(s) - 1$ and $\af(R(\CalP)) \subset \displaystyle \bigsqcup_{\substack{1 \leq s \leq N}} X(s, 1)$.
	
	Use $A_{\CalP}$ to denote the subalgebra generated by $C(X)$ and $u \cdot C_0(X \backslash R(\CalP))$. Then
	$$ A_{\CalP} \cong \displaystyle \bigoplus _{s = 1}^N M_{h(s)} ( C(X(s, 1)) ) . $$
	In other words, there exists a \ca \ isomorphism
	$$ \varphi \colon A_{\CalP} \longrightarrow \bigoplus _{s = 1}^N M_{h(s)} ( C(X(s, 1)) ) $$
	satisfying
	$$ \varphi(1_{X(s, k)}) = \diag(0, \ldots, 0, 1, 0, \ldots) \in M_{h(s)} (C(X, 1)) $$
	with the $k$-th diagonal element being $1_{X(s, k)}$.

	It is clear that $1_U = \sum_{s, k} 1_{U \cap X(s, k)}$ and $1_V = \sum_{s, k} 1_{V \cap X(s, k)}$. Define $U_s$ to be \\ $\bigsqcup _{k} \left( U \cap X(s, k) \right)$ and $V_s$ to be $\bigsqcup _{k} \left( V \cap X(s, k) \right)$. It is clear that $1_U = \sum_{s} 1_{U_s}$ and $1_V = \sum_{s} 1_{V_s}$. Recall the isomorphism $\ph$ above. By abuse of notation, we can regard $1_{U_s}$ and $1_{V_s}$ as two diagonal matrices in $M_{h(s)} (C(X_{s, 1}))$.

	If we can find unitary elements $w_s \in M_{h(s)} (C(X_{s, 1}))$ such that $w_s \cdot 1_{U_s} \cdot w_s^* = 1_{V_s}$, by setting $w = w_1 + \cdots + w_s$, it is then clear that $w$ is unitary element in $\bigoplus _{s = 1}^N M_{h(s)} ( C(X(s, 1)) )$ such that $w \cdot 1_U \cdot w^* = 1_V$, which is equivalent to the existence of a unitary in $A_{\CalP}$ conjugating $1_U$ to $1_V$. As $x \in R(\CalP)$, we can regard $A_{\CalP}$ as a subalgebra of $A_x$. Then the unitary $w$ in $A_{\CalP}$ is also a unitary in $A_x$.

	Let $w_s$ be a unitary element in $M_{h(s)} (C(X_{s, 1}))$ satisfying
	$$ w_s \cdot E_{i, i} \cdot w_s^* = E_{i + 1, i + 1} $$
	for $i = 1, \ldots, h(s) - 1$ and
	$$ w_s E_{h(s), h(s)} w_s^* = E_{1, 1} , $$
	with $( E_{i, j} )$ being the standard system of matrix units. It follows that $w_s \cdot 1_{U_s} \cdot w_s^* = 1_{V_s}$, which finishes the proof.
\end{proof}

\onecm

\begin{lem}

	Let $(X, \af)$ be a minimal Cantor dynamical system and let $U, V$ be two clopen subsets of $X$ satisfying $\af^n(U) = V$ for some $n \in \N$. Then there exists a partition of $U$, say $U = \bigsqcup_{i = 1}^m U_i$ with each $U_i$ clopen such that for all $k = 1, \ldots, n$ and $i, j = 1, \ldots, m$ with $i \neq j$, we have $\af^{k} (U_i) \cap \af^{k} (U_j) = \varnothing$.

\end{lem}

\begin{proof}
	We just need to find a partition of $U$ into $U = \bigsqcup_{i = 1}^m U_i$ such that for every given $i$ with  $1 \leq i \leq m$, the clopen sets $\af^{1} (U_i), \ldots, \af^{n} (U_i)$ are mutually disjoint.

	For every $y \in U$, as $\af$ is a minimal homeomorphism, we can find a clopen set $D_y \subset U$ such that $\af^{1}(D_y), \ldots, \af^{n}(D_y)$ are mutually disjoint. As $U$ is compact, there exists a finite subset of $U$, say $\{ y_1, \ldots, y_N \}$, such that $\bigcup_{s = 1}^N D_{y_s} = U$.

	As the intersection of two clopen sets is still clopen, without loss of generality, we may assume that the sets $D_{y_1}, \ldots, D_{y_N}$ are mutually disjoint. That is, $U = \bigsqcup_{i = 1}^m D_{y_i}$. It is then clear that for any given $s$ with $1 \leq s \leq N$, $\af^{k} (D_{y_s})$ are mutually disjoint for $k = 1, \ldots, n$ , which finishes the proof.
\end{proof}

\onecm

The lemma below is the strengthened version of Lemma \ref{unitary equivalence in Ax, weaker version} in the sense that we no longer require $x \notin \bigcup_{k = 0}^{n - 1} \af^k(U)$.

\vspace{2mm}

\begin{lem} \label{unitary conjugation as normalizer}

	Let $X$ be the Cantor set and let $x \in X$. Let $\af$ be a minimal homeomorphism of $X$ and let $A_x$ be defined as in Lemma \ref{unitary equivalence in Ax, weaker version}. For every $n \in \N$ and clopen subset $U \subset X$, there exists a unitary element $w \in A_x$ such that
	$$ w = \sum_{j \in \Z} 1_{D_j} u^j \  \text{and} \ w \cdot 1_U \cdot w^* = 1_{\af^n(U)} , $$
	where $D_j$ for $j \in \Z$ are mutually disjoint clopen subsets of $X$ satisfying $X = \displaystyle \bigsqcup_{j \in \Z} D_j$, and all but finitely many $D_j$ are empty.

\begin{proof}
	Let $d$ be the metric on $X$. As $(X, \af)$ is a minimal dynamical system, $x, \af(x), \ldots, \af^n(x)$ are distinct from each other.
	
	Let
	$$ R = \Frac{1}{2} \min_{0 \leq i, j \leq n, i \neq j} d(\af^i(x), \af^j(x)) . $$
	It is clear that $R > 0$.

	For $k$ with $0 \leq k \leq n$, if $x \in \af^k(U)$, as $\af^k(U)$ is clopen, there exists $r_k > 0$ such that the open set $\{ y \in X \colon d(x, y) < r_k \}$ is a subset of $\af^k(U)$. If $x \notin \af^k(U)$, as $\af^k(U)$ is compact, $\inf_{y \in \af^k(U)} d(x, y) = d(x, y')$ for some $y' \in \af^k(U)$. In this case, let $r_k = \inf_{y \in \af^k(U)} d(x, y)$.

	Let
	$$ r = \min (R, r_0, r_1, \ldots, r_n) > 0 $$
	and define $E'$ to be
	$$ \{ y \in X \colon d(x, y) < r \} . $$
	Then $E'$ is an open subset of $X$. As the topology of the Cantor set $X$ is generated by clopen sets, we can find a clopen subset $E \subset E'$ such that $x \in E$.

	According to the definition of $r$, it follows that for $k = 0, 1, \ldots, n$, either $E' \subset \af^k(U)$ or $E' \cap \af^k(U) = \varnothing$. The fact that $E \subset E'$ implies that for $k = 0, 1, \ldots, n$, either $E \subset \af^k(U)$ or $E \cap \af^k(U) = \varnothing$.

	Let $\CalP$ be a Kakutani-Rokhlin tower such that the roof set is $E$. \removeme[More details?]{As $E$ is the roof set and $E, \af(E), \ldots, \af^n(E)$ are mutually disjoint, it follows that the height of each tower in $\CalP$ is greater than $n + 1$}.

	Use $X(N, v, s)$ to denote the clopen subset of the partition $\CalP$ at the $v$-th tower,  with height $s$. Then
	$$ X = \bigsqcup_{v \in V, 1 \leq k \leq h(v)} X(n, v, s) , $$
	where $h(v)$ is the height of the $v$-th tower.

	Let $U_{v, k} = U \cap X(N, v, k)$. Then
	$$ U = \bigsqcup_{v \in V, 1 \leq k \leq h(v)} U_{v, k} . $$
	For every $v, k$ such that $U_{v, k} \neq \varnothing$, if there exists $m \in \N$ such that $1 \leq m \leq n$ and $\af^m(U_{v, k}) \subset \af(E)$, then $E \cap \af^{m - 1}(U) \neq \varnothing$. According to our choice of $E$, for all $s$ with $1 \leq s \leq n$, either $E \subset \af^s(U)$ or $E \cap \af^s(U) = \varnothing$. By assumption, we have $\af^m(U_{v, k}) \subset \af(E)$ and $U_{v, k} \neq \varnothing$. Then
	$$ E \cap \af^{m - 1}(U) \hspace{1mm} \supset \hspace{1mm} E \cap \af^{m - 1}(U_{v, k}) \hspace{1mm} = \hspace{1mm} \af^{m - 1} (U_{v, k}) \hspace{1mm} \neq \hspace{1mm} \varnothing , $$
	which implies that $E \subset \af^{m - 1}(U)$.

%\begin{sloppypar}
	Let $A_E$ be the subalgebra of $A$ generated by $C(X)$ and $u \cdot C_0(X \backslash R(\CalP) )$, with $u$ being \\
	the implementing unitary of $A$. We will show that there exists a unitary element $w \in A_E$ such that
	$$ w = \sum_{j \in \Z} 1_{D_j} u^j $$
	with all the sets $D_j$ for $j \in \Z$ being mutually disjoint and $w \cdot 1_U \cdot w^* = 1_{\af^n(U)}$. As $A_E$ can be regarded as a subalgebra of $A_x$, that is enough to prove the lemma if we can find the unitary $w$ as described above.
%\end{sloppypar}

	\removeme[tentative proof below.]

	If $k + n \leq h(v)$, this is the case that $x \notin \bigsqcup_{j = 0}^{n - 1} \af^{j}(U_{v, k})$. According to Lemma \ref{partial isometry in the desired form}, there exists a partial isometry $s_{v, k} \in A_x$ such that $s_{v, k}^* s_{v, k} = 1_{U_{v, k}}$ and $s_{v, k} s_{v, k}^* = 1_{\af^n(U_{v, k})} = 1_{U_{v, k + n}}$. According to the remark after Lemma \ref{partial isometry in the desired form}, we have $s_{v, k} = u^n \cdot 1_{U_{v, k}}$.

	If there is a nonempty $U_{v, k}$ such that $k + n > h(v)$, then
	$$ \af^{h(v) - k}(U) \cap E \supset \af^{h(v) - k}(U_{v, k}) \cap E \neq \varnothing . $$
	According to the construction of $E$, it follows that $E \subset \af^{h(v) - k}(U)$, which then implies that $\af^{-(h(v) - k)}(E) \subset U$. Intersecting both sets with
	$$ \af^{-(h(v) - k)}(E) = \bigsqcup_{v' \in V} X(n, v', h(v') - (h(v) - k)), $$
	we get
%	$$
%	\begin{array}{lll}
	\begin{align*}
	\displaystyle\bigsqcup_{v' \in V} X(n, v', h(v') - (h(v) - k)) & = \af^{-(h(v) - k)}(E) \displaystyle\displaystyle\cap \, \displaystyle\bigsqcup_{v' \in V} X(n, v', h(v') - (h(v) - k)) \\
	 & \subset U \ \displaystyle \cap \, \displaystyle\bigsqcup_{v' \in V} X(n, v', h(v') - (h(v) - k)) \\
	 & \subset \displaystyle\bigsqcup_{v' \in V} X(n, v', h(v') - (h(v) - k)) \ ,
	 \end{align*}
%	\end{array}
%	$$
	which implies that
	$$ U \cap \, \bigsqcup_{v' \in V} X(n, v', h(v') - (h(v) - k)) = \bigsqcup_{v' \in V} X(n, v', h(v') - (h(v) - k)) . $$
	In other words,
	$$ U_{v', h(v') - (h(v) - k)} = X(n, v', h(v') - (h(v) - k)) \ \text{for all} \ v \in V' . $$
	Now we have
	$$ \af^{-(h(v) - k)} (E) = \bigsqcup_{v' \in V} U_{v', h(v') - (h(v) - k)} = \bigsqcup_{v' \in V} X_{v', h(v') - (h(v) - k)} . $$
	It follows that
	$$ \af^n \left( \bigsqcup_{v' \in V} U_{v', h(v') - (h(v) - k)} \right) = \af^n \left( \bigsqcup_{v' \in V} X_{v', h(v') - (h(v) - k)} \right) = \bigsqcup_{v' \in V} X_{v', n - (h(v) - k)} . $$

	By Lemma \ref{partial isometry in the desired form}, there exists a partial isometry $s_{v, k}'$ such that
	$$ s_{v, k}' s_{v, k}'^* = 1_{U(v', h(v') - (h(v) - k))} $$
	and
%	$$
%	\begin{array}{lll}
	\begin{align*}
		s_{v, k}'^* s_{v, k}' & = 1_{\af^n({U(v', h(v') - (h(v) - k))}} \\
			& = 1_{U(v', h(v') + n - (h(v) - k)) - h(v')} .
	\end{align*}
%	\end{array}
%	$$
	Furthermore, according to the remark after Lemma \ref{partial isometry in the desired form}, $s_{v, k}' \in A_E$. %QQQ, finish this: and $s_{v, k}' = $ ...

	For every non-empty $U_{v, k}$, either $k + n \leq h(v)$ or $U \supset \af^{-(h(v) - k)} (R(\CalP))$. Thus the above two cases will give a partial isometry $s \in A_E$ such that $s s^* = 1_U$ and $s^* s = 1_{\af^n(U)}$. % = 1_{} \cdot u^n$.

	\removeme[QQQQ QQQQ QQQQ : Orig: we will show that    .. Finish it.] There exists a partial isometry $\widetilde{s} \in A_E$ such that
	$$ \widetilde{s} \widetilde{s}^* = 1_{X \backslash U} \ \text{and} \ \widetilde{s}^* \widetilde{s} = 1_{X \backslash \af^n(U)} . $$
	Let $w = s + \widetilde{s}$. Then $w$ is a unitary element in $A_E$ satisfying $w \cdot 1_u \cdot w^* = 1_{\af^n(U)}$, which finishes the proof.
\end{proof}

\end{lem}

\onecm

\begin{lem} \label{technical lemma tracial rank of A}

	Let $X$ be the Cantor set and let $x \in X$.  Let $D$ be a clopen subset of $X$ satisfying $x \in D$, and use $X \times \T_1 \times \T_2$ to denote the product of the Cantor set and two dimensional torus. Let $A$ be the crossed product \ca \ $C^*(\Z, X \times \T_1 \times \T_2, \af \times \mathrm{R}_{\xi} \times \mathrm{R}_{\eta})$ and let $u$ be the implementing unitary of $A$. Let $z_1 \in C(\T_1, \C)$ be defined by $z_1(t) = t$ and let $z_2 \in C(\T_2, \C)$ be defined by $z_2(t) = t$. By abuse of notation, we identify $z_1$ with $\id_X \otimes z_1 \otimes \id_{\T_2}$ and $z_2$ with $\id_X \otimes \id_{\T_1} \otimes z_2$. Suppose that there exists $M \in \N$ such that
	$$ \norm{u^M z_i p u^{-M} - z_i q} < \ep \ \text{for} \ i = 1, 2 ,\, \text{where} \  p = 1_D \  \text{and} \  q = u^M p u^{-M} . $$
	Then there exists a partial isometry $w \in A_x$ (with $A_x$ as defined in Lemma \ref{unitary equivalence in Ax, weaker version}) such that
	$$ w^* w = p, \ w w^* = q \ \text{and} \ \ \norm{w z_i p w^* - z_iq} < \ep \ \text{for} \ i = 1, 2 . $$

\begin{proof}
	According to Lemma \ref{unitary conjugation as normalizer}, we can find a unitary element $w_1 \in A_x$ such that
	$$ w_1 = \sum_{k \in \Z} u^k 1_{n^{-1}(k)} $$
	for some $n \in C(X, \Z)$ and
	$$ w_1 p w_1^* = q . $$

	Let
	$$ j_0 \colon C(\T_1 \times \T_2) \longrightarrow C(D \times \T_1 \times \T_2) $$
	be defined by $j_0(f) = 1_D \otimes f$ for all $f \in C(\T_1 \times \T_2)$. Then it is clear that $j$ is an injective homomorphism.

	As $C(D \times \T_1 \times \T_2) \subset p A_x p$ (with $p = 1_D$), we hence get the canonical inclusion map
	$$\phi_0 \colon C(D \times \T_1 \times \T_2) \rightarrow p A_x p . $$

	Define
	$$ \phi_1 \colon C(D \times \T_1 \times \T_2) \longrightarrow p A_x p $$
	by
	$$ \phi_1(g) = w_1^* \cdot u^M \cdot g \cdot u^{-M} \cdot w_1 \ \text{for all} \  g \in C(D \times \T_1 \times \T_2) . $$
	
	As $q = u^M p u^{-M}$ and $p = 1_D$, it follows that $u^M \cdot g \cdot u^{-M} \in q C(X \times \T^2) q \subset q A_x q$.
	
	The fact that $w_1 p w_1^* = q$ implies that $w_1^* q A_x q w_1 = p A_x q$. So far, we have shown that $\phi_1$ is really a homomorphism from $C(D \times \T^2)$ to $p A_x p$. As $\norm{\phi_1(g)} = \norm{g}$, it is clear that $\phi_1$ is injective.

	Define $\varphi_0 = \phi_0 \circ j_0$ and $\varphi_1 = \phi_1 \circ j_0$. Then $\varphi_0, \varphi_1$ are two injective homomorphisms from $C(\T^2)$ to $p A_x p$.
	
	Let
	$$ j \colon p A_x p \longrightarrow p A p $$
	be the canonical embedding.
	
	By Lemmas \ref{injectivity of K_0} and \ref{injectivity of K_1}, \removeme[Do we need to prove injectivity for K1 in detail?]{
	$$ j_{*i} \colon K_i(p A_x p) \longrightarrow K_i(p A p) $$
	will induce an injective embedding of $K_i(p A_x p)$ into $K_i(p A p)$ for $i = 0, 1$}.

	Consider $(\varphi_0)_{*i}$ and $(\varphi_1)_{*i} \colon K_i(C(\T_1 \times \T_2)) \rightarrow K_i(p A_x p)$ for $i = 0, 1$. As $\varphi_1(f) = w_1^* u^M f u^{-M} w_1$, it is clear that $(\varphi_0)_{*i}(a) = (\varphi_1)_{*i}(a)$ in $K_i(p A p)$ for all $a \in K_i(\T_1 \times \T_2)$. Since we know that $j_{*i} \colon K_i(p A_x p) \rightarrow K_i(p A p)$ is injective, it follows that $(\varphi_0)_{*i}(a) = (\varphi_1)_{*i}(a)$ in $K_i(p A_x p)$ for all $a \in K_i(\T_1 \times \T_2)$.

	For a \ca \ $B$, recall from Section \ref{Sec Introduction and Notation} that $T(B)$ denotes the convex set of all tracial states on $B$.
	For all $\tau \in T(p A p)$ and $g \in C(D \times \T_1 \times \T_2)$, it is clear that
	$$ \tau(w_1^* u^M g u^{-M} w_1) = \tau(g) . $$
	\removeme[Reference? Or more details?]{As $T(p A p) = T(p A_x p)$},  it follows that for every  tracial state $\tau' \in T(p A_x p)$, we have
	$$ \tau' \left( w_1^* u^M g u^{-M} w_1 \right) = \tau'(g) . $$
	It is then clear that for all $\tau' \in T(p A_x p)$ and $f \in C(\T_1 \times \T_2)$,
	$$ \tau' \left( \varphi_0(f) \right) = \tau'( \varphi_1(f) ) . $$

	Recall from Definition \ref{defn of sharp morphism} the maps
	$$ {\varphi_0}^{\sharp}, {\varphi_1}^{\sharp} \colon U(C(\T_1 \times \T_2))/CU(C(\T_1 \times \T_2)) \rightarrow  U(p A_x p)/CU(p A_x p) . $$
	We will show that $\varphi_0(z_1 \otimes 1_{T_2}) \cdot \varphi_1(z_1 \otimes 1_{T_2})^{-1} \in CU(p A_x p)$. If that is done, then we can show that $\varphi_0(1_{T_1} \otimes z_2) \cdot \varphi_1(1_{T_1} \otimes z_2)^{-1} \in CU(p A_x p)$ in a similar way.

	In fact,
%	$$
%	\begin{array}{lll}
	\begin{align*}
	\varphi_1(z_1 \otimes 1_{\T_2}) & = w_1^* \cdot u^M \cdot \left( 1_D \otimes z_1 \otimes 1_{\T_2} \right) \cdot u^{-M} \cdot w_1 \\
		& = w_1^* \cdot \left( 1_{\af^M(D)} \otimes z_1 \cdot e^{2 \pi i s} \otimes 1_{\T_2} \right) \cdot w_1
	\end{align*}
%	\end{array}
%	$$
	for some $s \in C(X, \R)$. As $w_1 = \displaystyle \sum_{k \in \Z} u^k 1_{n^{-1}(k)}$ and $w_1 1_D w_1^* = u^M p u^{-M}$, we get
	\begin{align*}
		\varphi_1(z_1 \otimes 1_{\T_2}) & = w_1^* \cdot \left( 1_{\af^M(D)} \otimes \left( z_1 \cdot e^{2 \pi i s} \right) \otimes 1_{\T_2} \right) \cdot w_1 \vspace{1mm} \\
			& = \left( \displaystyle\sum_{k \in \Z} u^k 1_{n^{-1}(k) \times \T_1 \times \T_2} \right)^* \cdot \left( 1_{\af^M(D)} \otimes (z_1 \cdot e^{2 \pi i s}) \otimes 1_{\T_2} \right) \cdot \left( \displaystyle\sum_{k \in \Z} u^k 1_{n^{-1}(k)  \times \T_1 \times \T_2} \right) \vspace{1mm} \\
			& = \displaystyle\sum_{k, j \in \Z}  \hspace{2mm} 1_{n^{-1}(k) \times \T_1 \times \T_2} \cdot u^{-k} \cdot \left( 1_{\af^M(D)} \otimes (z_1 \cdot e^{2 \pi i s}) \otimes 1_{\T_2} ) \right) \cdot u^j \cdot 1_{n^{-1}(j)  \times \T_1 \times \T_2} \vspace{1mm} \\
			& = \displaystyle\sum_{k \in \Z} \hspace{2mm} 1_{n^{-1}(k)  \times \T_1 \times \T_2} \cdot u^{-k} \cdot\left( 1_{\af^M(D)} \otimes (z_1 \cdot e^{2 \pi i s}) \otimes 1_{\T_2} ) \right) \cdot u^k \cdot 1_{n^{-1}(k)  \times \T_1 \times \T_2} \\
			& = 1_{D} \otimes \left( z_1 \cdot e^{2 \pi i h} \right) \otimes 1_{\T_2}
	\end{align*}
	for some $h \in C(X, \R)$. Then we have
	$$ \varphi_0(z_1 \otimes 1_{\T_2}) \cdot \varphi_1(z_1 \otimes 1_{\T_2})^{-1} = 1_D \otimes e^{- 2 \pi i h} \otimes 1_{\T_2} $$
	with $h \in C(X, \R)$, and we also have
	$$ 1_D \otimes e^{- 2 \pi i h} \otimes 1_{\T_2} \in p A_x p \cap p C^*(\Z, X, \af) p . $$
	Note that $p A_x p \cap p C^*(\Z, X, \af) p \cong p C^*(\Z, X, \af)_x p$, which is an infinite dimensional simple AF algebra by \cite{HPS}. By Lemma \ref{a handy known fact}, it follows that
	$$ U(p A_x p \cap p C^*(\Z, X, \af) p) = CU(p A_x p \cap p C^*(\Z, X, \af) p) . $$
	Then we get
	$$
	\begin{array}{l}
	\varphi_0(z_1 \otimes 1_{\T_2}) \cdot \varphi_1(z_1 \otimes 1_{\T_2})^{-1} \in \\
	\hspace{1cm} U \left( p A_x p \cap p C^*(\Z, X, \af) p \right) = CU \left( p A_x p \cap p C^*(\Z, X, \af) p \right) \subset CU(p A_x p) .
	\end{array}
	$$
	So far, we have shown that $\varphi_0^{\sharp}(z_1 \otimes 1_{\T_2}) = \varphi_1^{\sharp}(z_1 \otimes 1_{\T_2})$. In the same way, it follows that $\varphi_0^{\sharp}(1_{\T_1} \otimes z_2) = \varphi_1^{\sharp}(1_{\T_1} \otimes z_2)$.

	%QQQQ [Find the reference and apply the uniqueness theorem]
	According to \cite[Theorem 10.10]{Lin1}, we conclude that $\varphi_0$ and $\varphi_1$ are approximately unitarily equivalent. Then there exists a unitary $w_2 \in p A_x p$ such that
	$$ \norm{w_1^* u^M z_i u^{-M} w_1 - w_2 z_i p w_2^*} < \ep - \norm{u^M z_i p u^{-M} - z_i q} . $$
	Let $w = w_1 w_2$. Then
	$$ \norm{u^M z_i p u^{-M} - z_i q} < \ep \ \text{for} \ i = 1, 2. $$

	We can easily check that
	$$ w^* w = w_2^* w_1^* w_1 w_2 = w_2^* w_2 = p $$
	and
	$$ w w^* = w_1 w_2 w_2^* w_1 = w_1 p w_1^* = q , $$
	which finishes the proof.
\end{proof}

\end{lem}

\onecm

\begin{lem} \label{closest technical lemma for tracial rank of A}

	We write $X \times \T \times \T$ as $X \times \T_1 \times \T_2$ to distinguish the factors. Let $A$ be the crossed product \ca \ $C^*(\Z, X \times \T_1 \times \T_2, \af \times \mathrm{R}_{\xi} \times \mathrm{R}_{\eta})$ and let $u$ be the implementing unitary of $A$. Let $x \in X$. For any $N \in \N$, any $\ep >0 $ and any finite subset $\CalG \subset C(X \times \T \times \T)$, we have a natural number $M > N$, a clopen neighborhood $U$ of $x$ and a partial isometry $w \in A_x$ (with $A_x$ defined as in Lemma \ref{unitary equivalence in Ax, weaker version}) satisfying the following:

	(1) $\af^{- N + 1}(U), \af^{- N + 2}(U), \ldots, U, \af(U), \ldots, \af^M(U)$ are mutually disjoint, and $\mu(U) < \ep / M$ for all $\af$-invariant probability measure $\mu$,

	(2) $w^* w = 1_U$ and $w w^* = 1_{\af^M(U)}$,

	(3) $u^{-i} w u^i \in A_x$ for $i = 0, 1, \ldots, M - 1$,

	(4) $\norm{w f - f w} < \ep$ for all $f \in \CalG$.

\begin{proof}
	By abuse of notation, we identify $f \in C(X)$ with $f \otimes \id_{\T_1} \otimes \id_{\T_2}$, $g \in C(\T_1)$ with $\id_X \otimes g \otimes \id_{\T_2}$ and $h \in C(\T_2)$ with $\id_X \otimes \id_{\T_2} \otimes h$.
	
	Without loss of generality, we can assume that
	$$ \CalG = \{ f_1, \ldots, f_k, z_1, z_2 \} , $$
	where
	% QQQQ Not done yet. C(X) ???
	$ f_i \in C(X) \subset C(X \times \T_1 \times \T_2) $ for $i = 1, \ldots, k$ and $z_i(t_i) = t_i$ for $t_i \in \T_i$, $i = 1, 2$.

	There exists a neighborhood $E$ of $x$ such that
	$$ |f_i(x) - f_i(y)| < \ep / 2 $$
	for all $y \in E$ and $i = 1, \ldots, k$. It then follows that for any $y_1, y_2 \in E$ and $i$ such that $1 \leq i \leq k$, we have
	$$ |f_i(y_1) - f_i(y_2)| < \ep. $$

	As $(X, \af)$ is minimal, there exists $M > N$ such that $\af^M(x) \in E$. Let
	$$ K = \max \left\{ M, \Frac{M}{\ep} + 1 \right\} . $$
	It is clear that the points $\af^{-N + 1}(x), \af^{-N + 2}(x), x, \af(x), \ldots, \af^K(x)$ are distinct. Then there exists a clopen set $U$ containing $x$ such that $U \subset E$, $\af^M(U) \subset E$
	and $\af^{-N + 1}(U)$, $\af^{-N + 2}(U)$, $U$, $\af(U)$, $\ldots$, $\af^K(U)$ are disjoint.

	As $\af^{-N + 1}(U), \af^{-N + 2}(U), U, \af(U), \ldots, \af^K(U)$ are disjoint, for every $\af$-invariant probability measure $\mu$, we have $\mu(U) < \ep / M$.

	By Lemma \ref{technical lemma tracial rank of A}, there exists a partial isometry $w \in A_x$ such that $w^* w = 1_U$ and $w w^* = 1_{\af^M(U)}$.

	As $U \subset E$ and $\af^M(U) \subset E$, it follows that $\norm{w f_i - f_i w} < \ep$ for $0 \leq i \leq k$. The fact that $\norm{u^M z_i p u^{-M} - z_i q} < \ep$ implies $\norm{w z_i - z_i w} < \ep$ for $i = 1, 2$. So far, (4) is checked.

	From our construction of $U$, we have (1). The assertion (2) follows from our construction of $w$. Note that $U, \af(U), \ldots, \af^M(U)$ are mutually disjoint. \removeme[DO check this!!!]{We can check that $u^{-i} w u^i  \in A_x$ for $i = 0, \ldots, m - 1$}, thus finishing the proof.
\end{proof}

\end{lem}

\onecm

\begin{dfn}

	Let $\CalC$ be a category of unital separable \ca s. A separable simple \ca \, $A$ is called $\CalC$-Popa if for every finite subset $\CalF \subset A$ and $\ep > 0$, there exists a nonzero projection $p \in A$ and a unital subalgebra $B$ of $p A p$ (with $1_B = p$) such that $B \in \CalC$ and

	1) $\norm{[x, p]} \leq \ep$ for all $x \in \CalF$,

	2) $p \cdot x \cdot p \in _{\ep} B$ for all $x \in \CalF$.

\end{dfn}

\onecm

\begin{lem} \label{Popa-Popa}

	Let $\CalC$ be a category of unital separable \ca s.  Let $A$ be a separable simple \ca. If for every finite set $\CalF \subset A$ and $\ep > 0$, there exists a nonzero projection $p \in A$ and a unital subalgebra $B$ of $p A p$ such that $B$ is $\CalC$-Popa and

	1) $\norm{[x, p]} \leq \ep$ for all $x \in \CalF$,

	2) $p x p \in _{\ep} B$ for all $x \in \CalF$,
	
	\noindent then $A$ is $\CalC$-Popa.

\begin{proof}
	For any $\ep > 0$ and any finite subset $\CalF \subset A$, we can find a subalgebra $B$ such that $B$ is $\CalC$-Popa and

	1) $\norm{[x, 1_B]} \leq \ep$ for all $x \in \CalF$,

	2) $1_B \cdot x \cdot 1_B \subset _{\ep} B$ for all $x \in \CalF$.

	Use $1_B \CalF 1_B$ to denote the set $\{ 1_B x 1_B \colon x \in \CalF \}$. As $1_B \cdot x \cdot 1_B \in _{\ep} B$, for every $x \in \CalF$, choose an element $y_x \in B$ satisfying $\norm{y_x - 1_B \cdot x \cdot 1_B} \leq \ep$. Use $\CalG$ to denote $\{ y_x \colon x \in \CalF \}$ with $y_x$ as described.

	As $B$ is $\CalC$-Popa, we can find $E \subset B$ such that $E \in \CalC$ and

	a) $\norm{[1_E, y_x]} \leq \ep$ for all $y_x \in \CalG$,

	b) $1_E \cdot y_x \cdot 1_E \in _{\ep} 1_E$ for all $y_x \in \CalG$.

	We then check that
	\vspace{-6mm}
	$$
	 \norm{1_E \cdot y_x - y_x \cdot 1_E} \approx_{2 \ep} \norm{1_E \cdot 1_B \cdot x \cdot 1_B - 1_B \cdot x \cdot 1_B \cdot 1_E} \approx_{2 \ep} \norm{1_E \cdot x - x \cdot 1_E}.
	$$
	It then follows that
	$$ \norm{1_E \cdot x - x \cdot 1_E} \approx_{4 \ep} \norm{1_E \cdot y_x - y_x \cdot 1_E} . $$
	As $\norm{[1_E, y_x]} \leq \ep$, we get $\norm{[x, 1_E]} \leq 5 \ep$.

	For any $x \in A$, we have
%	$$
%	\begin{array}{lll}
	\begin{align*}
	\dist(1_E \cdot x \cdot 1_E, E) & = \dist(1_E \cdot (1_B \cdot x \cdot 1_B) \cdot 1_E, E) \\
		& \approx_{\ep} \dist(1_E \cdot y_x \cdot 1_E, E)) \\
		& \approx_{\ep} 0.
	\end{align*}
%	\end{array}
%	$$
	Then it is clear that $1_E \cdot x \cdot 1_E \in_{2 \ep} E$.

	Thus for every finite subset $\CalF \subset A$ and $\ep > 0$, we can find the subalgebra $E$ of $A$ as described above such that $E \in \CalC$ and

	1) $\norm{[x, 1_E]} \leq 5 \ep$ for all $x \in \CalF$,

	2) $1_E \cdot x \cdot 1_E \in _{2 \ep} E$ for all $x \in \CalF$,
	
	\noindent which shows that $A$ is $\CalC$-Popa.
\end{proof}

\end{lem}

\onecm

	This following is a technical result that will be needed later.

\vspace{2mm}

\begin{prp} \label{tech prop 1}

	Let $A$ be a \ca. For every $a \in A_{sa}$ such that $\norm{a - a^2} \leq \delta < \frac{1}{4}$, there exists a projection $p \in C^*(a)$ such that $\norm{p - a} \leq \sqrt{\delta}$.

\begin{proof}
	Just refer \cite[Lemma 2.5.5]{Lin6}.
\end{proof}

\end{prp}

\begin{thm} \label{tracial rank of A without rigidity}

	Let $X$ be the Cantor set and let $\af \times \mathrm{R}_{\xi} \times \mathrm{R}_{\eta}$ be a minimal action on $X \times \T \times \T$. Use $A$ to denote the crossed product \ca \ of the minimal system $(X \times \T \times \T, \af \times \mathrm{R}_{\xi} \times \mathrm{R}_{\eta})$. Then $\mathrm{TR}(A) \leq 1$.

\begin{proof}
	According to \cite[Lemma 4.3]{HLX}, for simple \ca \ $A$, if for every $\ep > 0$, $c \in A_+ \setminus \{ 0 \}$ and finite subset $\CalF \subset A$, there exists a nonzero projection $p$ and a unital subalgebra $B$ of $p A p$ such that $\mathrm{TR}(B) \leq 1$ and

	1) $\norm{[x, p]} \leq \ep$ for all $x \in \CalF$,

	2) $\text{dist}(p \cdot x \cdot p, B) \leq \ep$ for all $x \in \CalF$,

	3) $1 - p \preceq c$ as in Definition \ref{projection M-vN equivalent to a projection in hereditary subalgebra of a positive element}. That is, $1 - 1_B$ is Murray-von Neumann equivalent to
	
	a projection in $\mathrm{Her}(c)$,

	\noindent then it follows that $\mathrm{TR}(A) \leq 1$.

	Let $A_x$ be as defined in Lemma \ref{dfn of Ax}. According to Lemma \ref{tracial rank of A_x}, $\mathrm{TR}(A_x) = 1$. If we can find a projection $e \in A_x$ such that $B = e A_x e$ satisfies the previous three conditions, then we are done.

	As $A$ is generated by $C(X \times \T \times \T)$ and the implementing unitary $u$, we can assume that the finite set is $\CalF \cup \{ u \}$ with $\CalF \subset C(X \times \T \times \T)$.

	Choose $N \in \N$ such that $2 \pi / N < \ep$ and let
	$$ \CalG = \bigcup_{i = 0}^{N - 1} u^i \CalF u^{-i} . $$
	According to Lemma \ref{closest technical lemma for tracial rank of A}, with respect to $\CalG$ and $\ep$ above, we can find $M > N$, a clopen neighborhood of $x$ and a partial isometry $w \in A_x$ satisfying $w^* w = 1_U$, $w w^* = 1_{\af^M(U)}$ and $\norm{[w, f]} < \ep$ for all $f \in \CalF$.

	Let $p = 1_U$ and $q = 1_{\af^M(U)}$. For $t \in [0, \pi/2]$, define
	$$ P(t) = p \cos^2 t + \sin t \cos t (w + w^*) + q \sin^2 t . $$
	As $p q = 0$ and $p$, $q$ are Murray-von Neumann equivalent via $w$, it follows that $t \mapsto P(t)$ is a path of projections with $P(0) = p$ and $P(\pi/2) = q$.

	Define
	$$e = 1 - \left( \sum_{i = 0}^{M - N} u^i p u^{-i} + \sum_{i = 1}^{N - 1} u^{-i} P(i \pi / 2 N) u^i \right) . $$
	According to Lemma \ref{closest technical lemma for tracial rank of A}, $u^{-i} w u^i \in A_x$ for $i = 0, \ldots, m - 1$. It is clear that $e \in A_x$. \removeme[QQQQ. Need more details.]{It follows that $e$ is a projection.}

	We first show that for $e \in A_x$ above, the following hold.

	1) $\norm{[x, e]} \leq \ep$ for all $x \in \CalF \cup \{ u \}$;   \hfill  (C1)

	2) $\text{dist}(e x e, e A_x e) \leq \ep$ for all $x \in \CalF \cup \{ u \}$. \hfill (C2)

	For the part of (C1) involving $u$, note that
	\begin{align*}
	    u e u^* - e  & = 1 - u \left( \displaystyle\sum_{i = 0}^{M - N} u^i p u^{-i} + \sum_{i = 1}^{N - 1} u^{-i} P(i \pi / 2 N) u^i \right) u^* \vspace{2mm}  \\
                    & \displaystyle \hspace{2cm} - \left( 1 -  \left(\sum_{i = 0}^{M - N} u^i p u^{-i} + \sum_{i = 1}^{N - 1} u^{-i} P(i \pi / 2 N) u^i \right) \right) \vspace{3mm} \\
                    & = \displaystyle - \sum_{i = 1}^{M - N + 1} u^i p u^{-i} + \sum_{i = 0}^{M - N} u^i p u^{-i} +  \sum_{i = 1}^{N - 1} u^{-i} P(i \pi / 2 N) u^i \vspace{2mm} \\
                     & \displaystyle \hspace{3cm} - \sum_{i = 0}^{N - 2} u^{-i} P((i + 1) \pi / 2 N) u^i \vspace{3mm} \\
                     & = p - u^{M - N + 1} p (u^*)^{M - N + 1} +  (u^*)^{N - 1} P((N - 1) \pi / 2 N) u^{N - 1} - P(\pi / 2 N) \vspace{2mm} \\
                    &  \displaystyle \hspace{3cm} + \sum_{i = 1}^{N - 2} u^{-i} ( P( i \pi / 2 N) - P((i + 1) \pi / 2 N) ) u^i \vspace{2mm} \\
                    & = \displaystyle p - P(\pi / 2 N) +  u^{-(N - 1)} P((N - 1) \pi / 2 N) u^{N - 1} - u^{M - N + 1} p u^{-(M - N + 1)} \vspace{2mm}  \\
                    & \displaystyle \hspace{3cm} + \sum_{i = 1}^{N - 2} u^{-i} ( P( i \pi / 2 N) - P((i + 1) \pi / 2 N) ) u^i  .
	\end{align*}

	\noindent \removeme[QQQQ Show all the details.]{As $2 \pi / N < \ep$, we get $\norm{u e u^* - e} < \ep$. It then follows that $\norm{u e - e u} < \ep $.} By Lemma \ref{closest technical lemma for tracial rank of A}, $\norm{f e - e f} < \ep$ for all $f \in \CalF$. So far, we have checked (C1).

	For $f \in \CalF \subset C(X \times \T \times \T)$, as $f \in A_x$, we get $e f e \in e A_x e$. As $e u \in A_x$, it is clear that $e u e = e (e u) e \in e A_x e$. Thus we have checked (C2).

\vspace{2mm}

	Let $\CalC$ be the set of all the unital separable \ca s $C$ such that there exist $N \in \N$ and one dimensional finite CW complexes $X_i$ and $d_i \in \N$ with $1 \leq i \leq N$ and
	$$ C \cong \bigoplus_{n = 1}^N M_{d_n} \left( C( X_n ) \right) . $$
	Note that $\ep$ can be chosen to be arbitrarily small, and also note that $e A_x e$ has tracial rank no more than one, which implies that $e A_x e$ is $\CalC$-Popa.

	By Lemma \ref{Popa-Popa}, $A$ is also $\CalC$-Popa. According to \cite[Lemma 3.6.6]{Lin6}, $A$ has property (SP). For the given element $c \in A_+$, there exists a non-zero projection $q \in \mathrm{Her}(c)$. Let $\delta_0 = \inf \{ \tau(q) \colon \tau \in T(A) \}$. As $A$ is simple and $q \neq 0$, we get $\tau(q) > 0$ for all $\tau \in T(A)$. As $T(A)$ is a weak* closed subset of the unit ball of $A\sp*$, noting that the unit ball of $A\sp*$ is weak* compact by Alaoglu's Theorem, it follows that $T(A)$ is also compact. Thus $\delta_0 > 0$.

	Without loss of generality, we can assume that $\ep < \min \{1, \frac{1}{8} \delta_0, \frac{1}{(40 \delta_0)^2} \}$ and $q \in \CalF$.

%\vspace{2mm}

	It remains to show that $1 - e$ is Murray-von Neumann equivalent to a projection in $\mathrm{Her}(c)$.
%	As $A_x$ has tracial rank one, we have Blackardar's comparability in $A_x$, and $\tau(1 - e) < \tau(e)$ for all tracial states implies that $1 - e \lesssim e$.

	As $q \in \CalF$, we have
	$$ \norm{[q, e]} \leq \ep \ \text{and} \ \text{dist}(e q e, e A_x e) \leq \ep . $$
	We can find $b \in (e A_x e)_{sa}$ such that $\norm{e q e - b} \leq \ep$. Note that $\norm{[q, e]} \leq \ep$ implies that $\norm{(e q e)^2 - e q e} \leq \ep$.  According to Proposition \ref{tech prop 1}, there exists a projection $q' \in A$ such that $\norm{q' - e q e} \leq \sqrt{\ep}$ and $q' \preceq e q e$ as in Definition \ref{projection M-vN equivalent to a projection in hereditary subalgebra of a positive element}.

	Note that we have
%	$$
%	\begin{array}{lll}
	\begin{align*}
	\norm{b^2 - b} & \leq \norm{b^2 - (e q e)^2} + \norm{(e q e)^2 - e q e} + \norm{e q e - b} \\
		& \leq 3 \ep + \ep + \ep \\
		& = 5 \ep.
	\end{align*}
%	\end{array}
%	$$
	By Proposition \ref{tech prop 1} again, there exists a projection $p \in e A_x e$ such that
	$$ \norm{p - b} \leq \sqrt{5 \ep} \ \text{and} \ [p] \leq [b] . $$
	As
	$$ \norm{p - q'} \leq \norm{p - b} + \norm{b - e q e} + \norm{e q e - q'} \leq \sqrt{5 \ep} + \ep + \sqrt{\ep} , $$
	it follows that $[p] = [q']$.
	As
	$$ q' \preceq e q e \ \text{and} \ e q e \preceq q , $$
	%[e q e] = [(e q) (q e)] = [(q e) (e q)] = [q e q] \leq [q] , $$
	we conclude that $p \lesssim q$ in $A$.

\vspace{4mm}

	%QQQQ [ Not needed.? ]  We will show that $p \neq 0$.  Note that $\norm{p - b} \leq \sqrt{5 \ep}$ and $\norm{e q e - b} \leq \ep$. It remains to show that $e q e \neq 0$.

	Note that
	$$ q = e q e + (1 - e) q e + e q (1 - e) + (1 - e) q (1 - e) . $$
	For every $\tau \in T(A)$, we have
	$$ \tau(q) = \tau(e q e) + \tau((1 - e) q (1 - e)) + \tau( (1 - e) q e + e q (1 - e) ) . $$
	According to (C1) and our choice of $\ep$, we have
	$$ \tau(e q e) + \tau((1 - e) q (1 - e)) > \tau(q) - \ep > \frac{1}{2} \tau(q) . $$
	As $\tau$ is a tracial state and $e$ is a projection,
	$$ \tau((1 - e) q (1 - e)) \leq \tau((1 - e) 1 (1 - e)) = \tau(1 - e) . $$
	\removeme[More details??]{Note that $\tau(1 - e) < \frac{1}{4} \tau(q)$} for all $\tau \in T(A)$ (because $\tau(1 - e) < \frac{1}{4} \delta_0$). We can conclude that
	$$ \tau(e q e) > \frac{1}{2} \tau(q) - \tau((1 - e) q (1 - e)) \geq \frac{1}{2} \tau(q) - \tau(1 - e)  > \frac{1}{4} \tau(q) \geq \frac{1}{4} \delta_0 > 0 . $$
	
%	QQQQ [ Not needed. ? ] and it follows that $e q e \neq 0$.
	
%	QQQQ: Why traces determines the [] \leq [] order? The reason shall be stated.
	
%	QQQQ  [ Not needed. ? ] So far, we can see that the projection $p \in A_x$ we get above is not zero and that $p$ satisfies $[p] \leq [c]$.

	%QQQQ Not needed?
%	Let $\delta_1 = \inf_{\tau \in T(A_x)} \{ \tau(p) \}$. As $p \neq 0$ and $A_x$ is simple, we get $\delta_1 > 0$.
	In our construction, note that
	$$ \norm{p - e q e} \leq \norm{p - b} + \norm{b - e q e} \leq \sqrt{5 \ep} + \ep . $$
%	and $\tau(e q e) > \frac{1}{4} \delta_0$ for all $\tau \in T(A)$.
	It follows that
	$$ \tau(p) \geq \frac{1}{4} \delta_0 - (\sqrt{5 \ep} + \ep) \geq \frac{1}{8} \delta_0 \ \ \text{for all} \ \tau \in T(A). $$
	According to our construction, we have
	$$ \tau(1 - e) < M \cdot \frac{\ep}{M} = \ep \leq \frac{1}{8} \delta_0 \leq \tau(p) $$
	for all $\tau \in T(A)$, \removeme[why we can apply Blackardar's comparability?]{which then implies that $1 - e \lesssim p$.} As $[p] \leq [c]$ (as in Definition \ref{projection M-vN equivalent to a projection in hereditary subalgebra of a positive element}), we get $[1 - e] \leq [c]$ (as in Definition \ref{projection M-vN equivalent to a projection in hereditary subalgebra of a positive element}), which finishes the proof.
\end{proof}

\end{thm}

\onecm

\removeme[Finish this cor later.]{}

%	\begin{cor}

%	Let X be the Cantor set and $\af \times \mathrm{R}_{\xi} \times \mathrm{R}_{\eta}$ be a minimal action on $X \times \T \times \T$, and let $A$ be the crossed product C*-algebra of minimal system $(X \times \T \times \T, \af \times \mathrm{R}_{\xi} \times \mathrm{R}_{\eta})$, then $A$ is an A$\T$ algebra.

%	\begin{proof}

%	By Theorem \ref{tracial rank of A without rigidity}, $A$ has tracial rank no more than one.

%	\end{proof}

%	\end{cor}

%	\vspace{1cm}

	The following result on the $K$-theory of the crossed product \ca \ above follows from Pimsner-Voiculescu six-term exact sequence.

\begin{prp} \label{K-data of crossed product C-algebra}

	Let $A$ be the crossed product \ca\, of the minimal dynamical system $(X \times \T \times \T, \af \times \mathrm{R}_{\xi} \times \mathrm{R}_{\eta})$. Then
	$$ K_0(A) \cong C(X, \Z^2) / \{f - f \circ \af^{-1} \colon f \in C(X, \Z^2) \} \oplus \Z^2 $$
	and
	$$ K_1(A) \cong C(X, \Z^2) / \{f - f \circ \af^{-1} \colon f \in C(X, \Z^2) \} \oplus \Z^2 . $$

\begin{proof}
	Use $j: C(X \times \T^2) \rightarrow A$ to denote the canonical embedding of $C(X \times \T^2)$ into $A$. We have the Pimsner-Voiculescu six-term exact sequence:
	$$
	\xymatrix{
		K_0(C(X \times \T^2)) \ar@{->}[rr]^{\id_{*0} - \af_{*0}} && K_0(C(X \times \T^2)) \ar@{->}[rr]^{j_{*0}} && K_0(A) \ar@{->}[d] \\
		K_1(A) \ar@{->}[u] && K_1(C(X \times \T^2)) \ar@{->}[ll]_{j_{*1}} && K_1(C(X \times \T^2)) \ar@{->}[ll]_{\id_{*1} - \af_{*1}} .
	}
	$$

	We know that
	$$ K_0(C(\T^2)) \cong \Z^2, \ K_1(C(\T^2)) \cong \Z^2 $$
	and
	$$ K_0(C(X)) \cong C(X, \Z), K_1(C(X))) = 0 . $$
	According to the K\"unneth theorem, $K_0(C(X \times \T^2)) \cong C(X, \Z^2)$ and $K_1(C(X \times \T^2)) \cong C(X, \Z^2)$.

	For $i = 0, 1$, consider the image of $\id_{*i} - \af_{*i}$. They are both isomorphic to
	$$ \{ f - f \circ \af^{-1} \colon f \in C(X, \Z^2) \} . $$

	The kernel of $\id_{*i} - \af_{*i}$ for $i = 0, 1$ is
	$$ \{ f \in C(X, \Z^2) \colon f = f \circ \af \} . $$
	Assume that $f$ is in the kernel of $\id_{*i} - \af_{*i}$ for $i = 0, 1$. Fix $x_0 \in X$. We have $f(\af^{n}(x_0)) = f(x_0)$ for all $n \in \Z$. As $\af$ is a minimal homeomorphism of the Cantor set $X$ and $f$ is continuous, $f$ must be a constant function from $X$ to $\Z^2$. Now we conclude that
	$$ \ker(\id_{*i} - \af_{*i}) \cong \Z^2 . $$

	As the six-term sequence above is exact, we have the short exact sequence:
    $$ 0 \longrightarrow \coker(\id_{*0} - \af_{*0}) \longrightarrow K_0(A) \longrightarrow \ker(\id_{*1} - \af_{*1}) \longrightarrow 0 . $$

	As $\ker(\id_{*i} - \af_{*i}) \cong \Z^2$ and $\Z^2$ is projective, it follows that
	$$ K_0(A) \cong \coker(\id_{*0} - \af_{*0}) \oplus \Z^2 . $$
	As $\coker(\id_{*0} - \af_{*0}) \cong C(X, \Z^2)/\{ f - f \circ \af \colon f \in C(X, \Z^2) \}$, we get
	$$ K_0(A) \cong C(X, \Z^2)/\{ f - f \circ \af \colon f \in C(X, \Z^2) \} \oplus \Z^2 . $$
	Similarly, we get that $K_1(A) \cong C(X, \Z^2)/\{ f - f \circ \af \colon f \in C(X, \Z^2) \} \oplus \Z^2$.
\end{proof}

\end{prp}

\vspace{1cm}

	If we require a certain ``rigidity" condition on the dynamical system $(X \times \T \times \T, \af \times \mathrm{R}_{\xi} \times \mathrm{R}_{\eta})$, then the tracial rank of the crossed product will be zero.

\vspace{5mm}

\begin{dfn} \label{definition of rigidity}

	Let $(X \times \T \times \T, \af \times \mathrm{R}_{\xi} \times \mathrm{R}_{\eta})$ be a minimal dynamical system. Let $\mu$ be an $\af \times \mathrm{R}_{\xi} \times \mathrm{R}_{\eta}$-invariant probability measure on $X \times \T \times \T$. It will induce an $\af$-invariant probability measure on $X$ defined by $\pi(u)(D) = \mu(D \times \T \times \T)$ for every Borel set $D \subset X$. We say that $(X \times \T \times \T, \af \times \mathrm{R}_{\xi} \times \mathrm{R}_{\eta})$ is rigid if $\pi$ gives a one-to-one map between the $\af \times \mathrm{R}_{\xi} \times \mathrm{R}_{\eta}$-invariant probability measures and the $\af$-invariant probability measures.

\end{dfn}

\noindent \textbf{Remark:} For minimal actions on $X \times \T \times \T$ of the type $\af \times \mathrm{R}_{\xi} \times \mathrm{R}_{\eta}$, it is easy to see that $\pi$ always maps the set of $\af \times \mathrm{R}_{\xi} \times \mathrm{R}_{\eta}$-invariant probability measures over $X \times \T \times \T$ onto the set of $\af$-invariant measures over $X$.

\vspace{1mm}

	According to Theorem 4.6 in \cite{Lin-Phillips}, the ``rigidity" condition defined above implies that the crossed product \ca \ has tracial rank zero.
	
\vspace{2mm}

\begin{prp} \label{rigidity implies tracial rank zero for rotation case}

	Let $(X \times \T \times \T, \af \times \mathrm{R}_{\xi} \times \mathrm{R}_{\eta})$ be a minimal dynamical system. If it is rigid, then the corresponding crossed  product \ca \ $C^*(\Z, X \times \T \times \T, \af \times \mathrm{R}_{\xi} \times \mathrm{R}_{\eta})$ has tracial rank zero.

\begin{proof}
	Use $A$ to denote $C^*(\Z, X \times \T \times \T, \af \times \mathrm{R}_{\xi} \times \mathrm{R}_{\eta})$. We will show that
	$$ \rho \colon K_0(A) \longrightarrow \text{Aff}(T(A)) $$
	has a dense range, which will then imply that $\mathrm{TR}(A) = 0$ according to \cite[Theorem 4.6]{Lin-Phillips}.

	For the crossed product \ca \ $B = C^*(\Z, X, \af)$, we know that $B$ has tracial rank zero and $\rho_B \colon K_0(B) \rightarrow T(B)$ has the dense range. According to \cite[Theorem 1.1]{Putnam}, $K_0(A) \cong C(X, \Z) / \{ f - f \circ \af^{-1} \}$. For every $x \in K_0(A)$, we can find $f \in C(X, \Z)$ such that $\hat{x}(\tau) := \tau(x)$ equals $\tau(f) = \int_X f \, \mathrm{d} \mu_{\tau}$.

	As $\af \times \mathrm{R}_{\xi} \times \mathrm{R}_{\eta}$ is rigid, there is a one-to-one correspondence between $(\af \times \mathrm{R}_{\xi} \times \mathrm{R}_{\eta})$-invariant measures and $\af$-invariant measures. In other words, $T(A)$ is homeomorphic to $T(B)$ (as two convex compact sets). Let $h \in C(X)$ be a projection. Then $h \otimes 1_{C(\T \times \T)}$ is a projection in $A$.

	\removeme[Is it necessary?]
%	By Proposition \ref{K-data of crossed product C-algebra}, we know that there is an order homomorphism $j \colon K_0(B) \rightarrow K_0(A)$. \red{For every $p \in M_{\infty}(B)$, we have  $\tau([p])$ ... }

	As $\rho_B$ has a dense range in $\text{Aff}{(T(B))}$, we have that $\rho$ has dense range in $\text{Aff}(T(A))$. As $X \times \T \times \T$ is an infinite finite dimensional metric space and $\af \times \mathrm{R}_{\xi} \times \mathrm{R}_{\eta}$ is minimal, according to \cite[Theorem 4.6]{Lin-Phillips}, $C^*(\Z, X \times \T \times \T, \af \times \mathrm{R}_{\xi} \times \mathrm{R}_{\eta})$ has tracial rank zero.
\end{proof}

\end{prp}

\vspace{2cm}

\section{Examples} \label{Examples with cocycles being rotations}

	This section contains examples of minimal dynamical systems of type $(X \times \T \times \T, \af \times \mbox{R}_{\xi} \times \mbox{R}_{\eta})$ that is rigid. It also contains a concrete example of a minimal dynamical system of the same type but is not rigid.
	
\vspace{2mm}
	
	We start with a criterion for determining whether a dynamical system of $(X \times \T \times \T, \af \times \mathrm{R}_{\xi} \times \mathrm{R}_{\eta})$ is minimal or not. This result is a special case of the remark of page 582 in \cite{Furstenberg}. The proof here essentially follows that of Lemma 4.2 of \cite{LM1}.

\vspace{2mm}

\begin{lem} \label{induction lemma on minimality}

	Let $Y$ be a compact metric space, and let $\bt \times \mathrm{R}_{\eta}$ be a skew product homeomorphism of $Y \times \T$ with $\bt \in \text{Homeo}(Y)$, $\eta \colon Y \rightarrow \T$ and
	$$  (\bt \times \mathrm{R}_{\eta}) (y, t) = (\bt(y), t + \eta(y)) \ \text{with} \ \T \ \text{identified with} \ \R / \Z . $$
	Then $\bt \times \mathrm{R}_{\eta}$ is minimal if and only if $(Y, \bt)$ is minimal and there exist no $f \in C(Y, \T)$ and non-zero integer $n$ such that
	$$ n \eta = f \circ \bt - f . $$

\end{lem}

\begin{proof}
	Proof of the ``if" part:

	If $(Y, \bt)$ is minimal and there exist no $f \in C(Y, \T)$ and non-zero integer $n$ such that $n \eta = f \circ \bt - f$, we will prove that $\bt \times \mathrm{R}_{\eta}$ is minimal.

	If $\bt \times \mathrm{R}_{\eta}$ is not minimal, then there exists a proper minimal subset $E$ of $Y \times \T$. Let $\pi_Y \colon Y \times \T \rightarrow Y$ be the canonical projection onto $Y$. Note that $\pi_Y \circ (\bt \times \mathrm{R}_{\eta}) = \bt \circ \pi_Y$. It follows that $\pi_Y(E)$ is an invariant subset of $Y$. As $Y$ is compact, so is $\pi_Y(E)$. Since $(Y, \bt)$ is minimal, the closed invariant set $\pi_Y(E)$ must be $Y$.

	Let's consider
	$$ D := \{ t \in \T: (\id_Y \times \mathrm{R}_t) ( E ) = E \} . $$
	As $(\id_Y \times \id_{\T}) ( E ) = E$, the set $D$ is not empty.
%	If $t_1, t_2 \in D$, then
%	$$ (\id \times \mathrm{R}_{t_1 + t_2})E = (\id \times \mathrm{R}_{t_1})((\id \times \mathrm{R}_{t_2})E) = (\id \times \mathrm{R}_{t_2})E = E . $$
%	Thus $t_1 + t_2 \in D$.
%	
%	If $t \in D$, then
%	$$ (\id \times \mathrm{R}_{- t}) E = (\id \times \mathrm{R}_{- t}) ((\id \times \mathrm{R}_{t}) E) = E , $$
%	which implies that $- t \in D$.
	Note that $D$ is a subgroup of $\T$. It follows that $D$ is a non-empty subgroup of $\T$ (with $\T$ identified with the quotient group $\R / \Z$).

	If we have $\{ t_n \}_{n \in \N} \subset D$ such that $t_n \rightarrow t$, then for any $\omega \in E$, we have $(\id \times \mathrm{R}_{t_n}) \omega \in E$. Then $t_n \rightarrow t$ implies that $(\id \times \mathrm{R}_{t_n}) w \rightarrow (\id \times \mathrm{R}_{t}) w$. As $E$ is closed, $(\id \times \mathrm{R}_{t}) w \in E$.

	So far, we have shown that if $t_n \in D$ for $n \in \N$ and $t_n \rightarrow t$, then $t \in D$. Note that ``$\{ t_n \}_{n \in \N} \subset D$ and $t_n \rightarrow t$" is equivalent to ``$\{ - t_n \}_{n \in \N} \subset D$ and $-t_n \rightarrow -t$". It follows that $- t \in D$. In other words, we have
	$$ (\id \times \mathrm{R}_{t}) ( E ) \subset E \ \text{and} \ (\id \times \mathrm{R}_{-t}) ( E ) \subset E . $$
	Then we get
	$$ E = (\id \times \mathrm{R}_{t}) ((\id \times \mathrm{R}_{-t}) ( E ) ) \subset (\id \times \mathrm{R}_{t}) ( E ) \subset E , $$
	which implies that $(\id \times \mathrm{R}_{t}) E = E$. In other words, $D$ is closed.

	As $E$ is a proper subset and $\pi_Y(E) = Y$, $D$ must be a proper subgroup of $\T$. Otherwise, for any $(y, t) \in Y \times \T$, as $\pi_Y(E) = Y$, there exists $t' \in \T$ such that $(y, t') \in E$. Since $t - t' \in D = \T$,  $(y, t) = (\id \times \mathrm{R}_{t - t'}) (y, t') \in E $, which indicates that $E = Y \times T$, contradicting the fact that $E$ is a proper subset.

	As a proper closed subgroup of $\T$, $D$ must be
	$$ \left\{ \frac{k}{n} \right\}_{0 \leq k \leq n - 1} \ \text{with} \ n = |D| . $$
	Let $\pi_{\T}$ be the canonical projection from $Y \times \T$ onto $\T$. For $y \in Y$, use $E_y$ to denote $\pi_{\T} (E \, \cap \pi_Y^{-1}(\{y\}))$.

	Using the fact that $E$ is a minimal subset of $(\bt, \mathrm{R}_{\eta})$, we will show that $E_y$ must be $n$ points distributed evenly on the circle for all $y \in Y$.

	We claim that if $t, t' \in E_y$, then for any $m \in \Z$, $t + m (t' - t)$ must be in $E_y$. To prove this claim, if $t, t' \in E_y$, then there exists $\{ k_n \}_{n \in \N}$ such that $k_n \rightarrow \infty$ and $\text{dist}((\bt \times \mathrm{R}_{\eta})^{k_n}(y, t), (y, t')) \rightarrow 0$. Note that
	$$ \text{dist}((\bt \times \mathrm{R}_{\eta})^{k_n}(y, t), (y, t')) = \text{dist}((\bt \times \mathrm{R}_{\eta})^{k_n}(y, t'), (y, t + 2 (t' - t))) . $$
	It follows that $(y, t + 2 (t' - t)) \in \overline{\text{Orbit}_{\bt \times \mathrm{R}_{\eta}}( (y, t) ) }$. By induction, we conclude that if $t, t' \in E_y$, then for any $m \in \Z$, $t + m (t' - t)$ is also in $E_y$, proving the claim.

	For any $y \in Y$, consider $E_y$, which is a non-empty closed subset of $\T$. Let
	$$ l_y = \inf_{ t_1, t_2 \in E_y } \text{dist}(t_1, t_2) . $$
	Note that if $t, t' \in E_y$, then $t + m (t' - t) \in E_y$. The fact that $E_y \subsetneq \T$ implies that $l_y > 0$. It is then clear that $E_y$ is made up of $1 / l_y$ points distributed evenly on $\T$.

\vspace{2mm}

	\textbf{Claim:} For every $y \in Y$, $1 / l_y = |D|$.
	
	For given $y \in Y$, as $(\id \times \mathrm{R}_t) ( E ) = E$ for all $t \in D$, we get that $E_y$ is invariant under $\mathrm{R}_t$ for all $t \in D$. It then follows that $1 / l_y = k n$ with $k \in \N$ and $n = |D|$.
	
	If $k > 1$, write
	$$ E_y = \{ (y, t_1), \ldots, (y, t_{k n}) \} . $$
	Use $\text{Orbit}_{\bt \times \mathrm{R}_{\eta}} ( E_y )$ to denote $\bigcup_{m = 1}^{\infty} (\bt \times \mathrm{R}_{\eta})^{m} ( E_y )$.
	
	As $\bt$ is minimal, for every $y' \in Y$, there is a sequence $(m_k)_{k \in \N}$ such that
	$$ \bt^{m_k} ( y ) \rightarrow y' . $$
	The fact that $\text{Orbit}_{\bt \times \mathrm{R}_{\eta}} ( E_y )$ is dense implies that there exists $t' \in \T$ such that $(y', t')$ is in the closure of $\text{Orbit}_{\bt \times \mathrm{R}_{\eta}} ( E_y )$. Note that for every $m \in \N$, $(\bt \times \mathrm{R}_{\eta})^m ( E_y )$ consists of $k n$ points distributed evenly on the circle. It follows that $E_{y'}$ contains at least $n k$ points distributed evenly on the circle.
	
	Now we have shown that for every $a \in Y$, $E_a$ is made up of at least $n k$ evenly distributed points on the circle, which then implies that $D$ contain at least $n k$ elements. The assumption that $k > 1$ gives a contradiction.
	
	We then conclude that $k = 1$, which proves the claim.
	
\vspace{2mm}	
	
	By the claim above, for all $y \in Y$, the set $E_y$ is made up of $n$ points distributed evenly on $\T$. If we define
	$$ n E = \{ (x, n t) \colon (x, t) \in E \} , $$
	then $n E$ is the graph of some continuous map $g \colon Y \rightarrow \T$. As $E$ is closed, so is $n E$, which implies that $g$ is continuous. As $E$ is $(\bt \times \mathrm{R}_{\eta})$-invariant, for every $(x, t) \in E$, it follows that
	$$ (\bt \times \mathrm{R}_{\eta})(x, t) = (\bt(x), t + \eta(x)) \in E . $$
	In other words, we have $n (t + \eta(x)) = g(\beta(x))$. As $n t = g(x)$, it follows that $n \eta = g \circ \bt - g$, which finishes the proof of ``if" part.

\vspace{4mm}

	Proof of the ``only if" part:

	Suppose $\bt \times \mathrm{R}_{\eta}$ is minimal. Then it is clear that $(Y, \bt)$ is a minimal system.

	Suppose that there exists nonzero $n \in \Z$ such that $n \eta = g \circ \bt - g$ for some $g \in C(X, T)$. Let
	$$ E = \{ (y, t) \in Y \times \T \colon  n t = g(y) \} . $$
	For $(y, t) \in E $, we have $(\bt \times \mathrm{R}_{\eta})(y, t) = (\bt(y), t + \eta(y))$. As
	$$ n (t + \eta(y)) = n t + n \eta(y) = g(y) + n \eta(y) = g(\bt(y)) , $$
	 it follows that $E$ is $(\bt \times \mathrm{R}_{\eta})$-invariant.

	As $g$ is continuous, $E$ is closed. And it is clear that $E$ is a proper subset of $Y \times \T$. Now we have a proper closed $(\bt \times \mathrm{R}_{\eta})$-invariant set in $Y \times \T$, contradicting the minimality of $\bt \times \mathrm{R}_{\eta}$.
\end{proof}

\vspace{2mm}

	Lemma \ref{induction lemma on minimality} provides an inductive approach to determine the minimality of some dynamical systems. Following this lemma, we get the proposition below.

\vspace{4mm}

\begin{prp} \label{criterion for minimality on X times T2 with cocycles being rotations}

	Let $\alpha \times \mathrm{R}_{\xi} \times \mathrm{R}_{\eta}$ be a homeomorphism of $X \times \T \times \T$. Then $\alpha \times \mathrm{R}_{\xi} \times \mathrm{R}_{\eta}$ is minimal if and only if

	i) $(X, \alpha)$ is minimal,

	ii) $\xi$ is not a torsion element in $C(X, \T) / \{ f \circ \af -  f \}$,

	iii) For $\widetilde{\eta} \in C(X \times \T, \T)$ defined by $\tilde{\eta}(x, t) = \eta(x)$, the map $\widetilde{\eta}$ is not a torsion element in
	$$ C(X \times \T, \T)  / \{ f \circ (\af \times \mathrm{R}_{\xi} ) - f  \colon f \in C( X \times \T, \T ) \} . $$

\end{prp}

\begin{proof}
	The ``if" part:

%	If i), ii) and iii) are true, we need to show that $\alpha \times \mathrm{R}_{\xi} \times \mathrm{R}_{\eta}$ is minimal.

	Note that $(X \times \T \times \T, \alpha \times \mathrm{R}_{\xi} \times \mathrm{R}_{\eta})$ is a skew product of $\alpha \times \mathrm{R}_{\xi}$ and $\widetilde{\mathrm{R}_{\eta}}$, where $\widetilde{\mathrm{R}_{\eta}}$ is defined by
	$$ \widetilde{\mathrm{R}_{\eta}} \colon X \times \T \rightarrow \text{Homeo}(\T) , \ \text{with} \ (\widetilde{\mathrm{R}_{\eta}}(x, t)) (t') = t' + \eta(x) . $$

 	From i) and ii), using Lemma 4.2 of \cite{LM1}, $(X \times \T, \af \times \mathrm{R}_{\xi})$ is minimal. According to Lemma \ref{induction lemma on minimality}, and by iii), we conclude that $\alpha \times \mathrm{R}_{\xi} \times \mathrm{R}_{\eta}$ is minimal.

	The ``only if" part:

	As $(X \times \T \times \T, \alpha \times \mathrm{R}_{\xi} \times \mathrm{R}_{\eta})$ is the skew product of $(X \times \T, \alpha \times \mathrm{R}_{\xi})$ and $\widetilde{\mathrm{R}_{\eta}} \colon X \times \T \rightarrow \text{Homeo}(\T)$, with $\widetilde{\mathrm{R}_{\eta}}$ defined as above, the minimality of $(X \times \T \times \T, \alpha \times \mathrm{R}_{\xi} \times \mathrm{R}_{\eta})$ implies the minimality of $(X \times \T, \alpha \times \mathrm{R}_{\xi})$. By Lemma 4.2 of \cite{LM1}, that implies (i) and (ii).

	For (iii), suppose that $\widetilde{\eta}$ is a torsion element, that is, there is non-zero $n \in \Z$ and $f \in C(X \times \T, \T)$ such that $n \tilde {\eta} = f \circ (\af \times \mathrm{R}_{\xi}) - f$. By Lemma \ref{induction lemma on minimality}, it follows that $(X \times \T \times \T, \af \times \mathrm{R}_{\xi} \times \mathrm{R}_{\eta})$ is not minimal, a contradiction.
\end{proof}

Proposition \ref{criterion for minimality on X times T2 with cocycles being rotations} enables us to construct minimal dynamical systems on $X \times \T \times \T$ inductively. In fact, we have the following lemma.

\begin{lem} \label{Existence of minimal actions by induction}
	
	Given any minimal dynamical system $(X \times \T, \af \times \mathrm{R}_{\xi})$, there exist uncountably many $\theta \in [0, 1]$ such that if we use $\theta$ to denote the constant function in $C(X, \T)$ defined by $\theta(x) = \theta$ for all $x \in X$ (identifying $\T$ with $\R / \Z$), then the dynamical system $(X \times \T \times \T, \af \times \mathrm{R}_{\xi} \times \mathrm{R}_{\theta})$ is still minimal.
	
\end{lem}

\begin{proof}	
	Note that the dynamical system $(X \times \T, \af \times \mathrm{R}_{\xi})$ is minimal. According to Lemma \ref{induction lemma on minimality}, $(X, \af)$ must be a minimal dynamical system, and $\xi$ is not a torsion element in
	$$ C(X, \T) / \{ f - f \circ \af \colon f \in C(X, \T) \} . $$
	This implies that conditions i) and ii) in Proposition \ref{criterion for minimality on X times T2 with cocycles being rotations} are already satisfied.

	According to Proposition \ref{criterion for minimality on X times T2 with cocycles being rotations}, for $(X \times \T \times \T, \af \times \mathrm{R}_{\xi} \times \mathrm{R}_{\theta})$ to be minimal, we just need to find $\theta \in \R$ such that for every $n \in \Z \setminus \{ 0 \}$ and $f \in C(X \times \T, \T)$, we have
	$$ n \theta \neq f - f \circ (\af \times \mathrm{R}_{\xi}) . $$
	
	If this is not true, then we have
	$$ n \theta = f - f \circ (\af \times \mathrm{R}_{\xi}) . $$
	
	Let $F \colon X \times \T \rightarrow \R$ be a lifting of $f$. That is, $F \in C(X \times \T, \R)$ and the following diagram commutes:
	$$
	\xymatrix{
		&&  \R \ar@{->}[d]^{\pi} \\
	X \times \T  \ar@{->}[rr]_{f} \ar@{->}[rru]^{F} && \T  ,
	}
	$$
	with $\pi(t) = t$ for all $t \in \R$ (identifying $\T$ with $\R / \Z$).
	
	Using $[F]$ to denote $\pi \circ F$, it follows that
%	$$
%	\begin{array}{lll}
	\begin{align*}
		n \theta & = [F] - [F \circ (\af \times \mathrm{R}_{\xi})] \\
			& = [F - F \circ (\af \times \mathrm{R}_{\xi})] .
	\end{align*}
%	\end{array}
%	$$
	In other words, there exists $g \in C(X \times \T, \Z)$ such that
	$$ n \theta -  ( F - F \circ (\af \times \mathrm{R}_{\xi}) ) = g . $$
	For every $(\af \times \mathrm{R}_{\xi})$-invariant probability measure $\mu$, we have $\mu ( n \theta) = \mu(g)$, with $\mu(n \theta) = \displaystyle \int_{X \times \T} n \theta \, \mathrm{d} \mu$ and $\mu(g) = \displaystyle \int_{X \times \T} g \, \mathrm{d} \mu$
	
	Since $\mu (n \theta) = n \mu(\theta)$, it follows that
	$$ \mu(\theta) = \Frac{\mu(g)}{n} = \mu \left( \Frac{g}{n} \right) . $$
	Let $A$ be the crossed product \ca \ of $(X \times \T, \af \times \mathrm{R}_{\xi})$. Define
	$$ \rho: A_{sa} \longrightarrow \mathrm{Aff}(T(A)) $$
	by $\rho(a)(\tau) = \tau(a)$ for all $a \in A_{sa}$ and $\tau \in T(A)$. Then we have $\rho( \theta ) = \rho \left( \Frac{g}{n} \right)$ in $\mathrm{Aff}(T(A))$.
	
	Now we have shown that if $\theta$ (as a constant function) is a torsion element in
	$$ C(X \times \T, \T) / \{ f - f \circ \af \colon f \in C(X \times \T, \T) \} $$
	with order $n$, then there exists $g \in C(X \times \T, \Z)$ such that $\rho ( \theta ) = \rho \left( \Frac{g}{n} \right)$.
	
	As $\T$ is connected, we have $C(X \times \T, \Z) \cong C(X, \Z)$. Note that the set
	$$ \left \{ \Frac{g}{n} \colon g \in C(X \times \T, \Z) \cong C(X, \Z), n \in \Z \setminus \{ 0 \} \right \} $$
	contains countably many elements. It follows that its image under $\rho$ contains at most countably many elements. The fact that $[0, 1]$ contains uncountably many elements and $\rho(\theta) = 0$ if and only if $\theta = 0$ implies that there exists (uncountably many, in fact) $\theta \in \R$ such that $\theta$ (as a constant function) is not a torsion element in
	$$ C(X \times \T, \T) / \{ f - f \circ \af \colon f \in C(X \times \T, \T) \} , $$	
	which then implies that $(X \times \T \times \T, \af \times \mathrm{R}_{\xi} \times \mathrm{R}_{\theta})$ is still minimal.
\end{proof}

	We now give examples of rigid and non-rigid minimal actions of on $X \times \T \times \T$.

\vspace{1mm}

%	Let $\af \times \mathrm{R}_{\xi} \times \mathrm{R}_{\eta}$ be a minimal homeomorphism of $X \times \T \times \T$, it  is clear that $(X, \af)$ must be minimal.

	Let $\varphi_0 \colon \T \rightarrow \T$ be a Denjoy homeomorphism (see \cite[Definition 3.3]{PSS} or \cite[Prop 12.2.1]{KatokHasselblatt}) with rotation number $r(\gamma) = \theta$ for some $\theta \in \R \setminus \Q$. It is known that $\varphi_0$ has a unique proper invariant closed subset of $\T$, which is a Cantor set, and that $\varphi_0$ restricted on this Cantor set is minimal. Let $X$ be the Cantor set and use $\varphi \colon X \rightarrow X$ to denote the restriction of $\varphi_0$ to $X$.

	According to the Poincare Classification Theorem (see \cite[Theorem 11.2.7]{KatokHasselblatt}), there is a non-invertible continuous monotonic map $h \colon \T \rightarrow \T$ such that the following diagram commutes:
	$$
	\xymatrix{
		\T \ar@{->}[rr]^{\varphi_0} \ar[d]_{h} && \T \ar[d]^{h} \\
		\T \ar@{->}[rr]_{\mathrm{R}_{\theta}} && \T & .
	}
	$$

	Using the restriction of $\varphi$ to the invariant subset (which is the Cantor set $X$), we  get a commutative diagram:
	$$
	\xymatrix{
		X \ar@{->}[rr]^{\varphi} \ar[d]_{h \left|_X \right.} && X \ar[d]^{h \left|_X \right.} \\
		\T \ar@{->}[rr]_{\mathrm{R}_{\theta}} && \T & .
	}
	$$
	It is known that
%	\removeme[I still need to check into details of this]{}
	for a Denjoy homeomorphism, $h \left|_X \right.$ maps $X$ onto $\T$.

	Recall that for $\xi, \eta \colon \T \rightarrow \T$, the action
	$$ \gamma \colon (s, t_1, t_2) \mapsto (s + \theta, t_1 + \xi(s), t_2 + \eta(s)) $$
	is called a Furstenberg transformation. Consider the action
	$$ \af \times \mathrm{R}_{\xi \circ h} \times \mathrm{R}_{\eta \circ h} \colon X \times \T \times \T \rightarrow X \times \T \times \T . $$
	It is clear that we have the commutative diagram below :
	\begin{equation} \label{factor thru diagram from Denjoy homeomorphism}
	\xymatrix{
		X \times \T \times \T \ar@{->}[rr]^{\af \times \mathrm{R}_{\xi \circ h} \times \mathrm{R}_{\eta \circ h}} \ar[d]_{h \left|_X \right. \times \id_{\T} \times \id_{\T}} && X \times \T \times \T \ar[d]^{h \left|_X \right. \times \id_{\T} \times \id_{\T}} \\
		\T \times \T \times \T \ar[rr]^{\gamma} && \T \times \T \times \T \ \ .
	}
	\end{equation}
	In this case, if $\gamma$ is minimal, then $\af \times \mathrm{R}_{\xi \circ h} \times \mathrm{R}_{\eta \circ h}$ is also minimal, as will be shown in the next proposition.

\begin{prp} \label{minimality thru the commutative diagram}

	For the minimal dynamical systems as in diagram (\ref{factor thru diagram from Denjoy homeomorphism}), if $(\T \times \T \times \T, \gamma)$ is a minimal dynamical system, then $(X \times \T \times \T, \af \times \mathrm{R}_{\xi \circ h} \times \mathrm{R}_{\eta \circ h})$ is also a minimal dynamical system.

\end{prp}

\begin{proof}	
	Assume that $(\T \times \T \times \T, \gamma)$ is minimal and $(X \times \T \times \T, \af \times \mathrm{R}_{\xi \circ h} \times \mathrm{R}_{\eta \circ h})$ is not minimal. It then follows that there exist $(x, t_1, t_2) \in X \times \T \times \T$, nonempty open subset $D \subset X$ and open subsets $U, V \subset \T$ such that
	\begin{equation} \label{the non-density fact}
		\{ (\af \times \mathrm{R}_{\xi \circ h} \times \mathrm{R}_{\eta \circ h})^n ( x, t_1, t_2 ) \}_{n \in \N} \cap ( D \times U \times V ) = \varnothing .
	\end{equation}
	
	Define
	$$ \pi_1, \pi_2 \colon X \times \T \times \T \longrightarrow \T \times \T $$
	by
	$$ \pi_1(x, t_1, t_2) = t_1 \ \text{and} \ \pi_2(x, t_1, t_2) = t_2 . $$
	As $\af$ is a minimal action on the Cantor set $X$, the statement \ref{the non-density fact} implies that for every $k \in \N$ such that $\af^k(x) \in D$, we have
	\begin{equation} \label{quick fact following the assumption}
	\pi_1 \left( (\af \times \mathrm{R}_{\xi \circ h} \times \mathrm{R}_{\eta \circ h})^k ( x ) \right) \notin U \ \text{and} \ \pi_2 \left( (\af \times \mathrm{R}_{\xi \circ h} \times \mathrm{R}_{\eta \circ h})^k ( x ) \right) \notin V .
	\end{equation}
	
	Note that if we regard the Cantor set $X$ as a subset of $\T$, then $h \left|_X \right. \colon X \rightarrow \T$ is a noninvertible continuous monotone function. For the open set $D \subset X$, without loss of generality, we can assume that (by identifying $X$ as a subset of $\T$ and identifying $\T$ with $\R / \Z$)
	$$ D = (a, b) \cap X \ \text{with} \ a, b \in (0, 1) \ \text{and} \ a < b . $$
	It then follows that there exists $c, d \in (0, 1)$ with $c < d$ (without loss of generality, we can assume that $0 \notin h \left|_X \right. (D)$ such that $h \left|_X \right. (D)$ is one of the following:
	$$ (c, d), \ (c, d], \ [c, d) \ \text{or} \ [c, d] . $$
	In either case, there exists $c', d' \in (0, 1)$ with $c' < d'$ such that
	$$ (c', d') \subset h \left|_X \right. (D) . $$
	
	Let $t_x = h \left|_X \right. (x)$. It is then clear that
	$$ h \left|_X \right. \left( (\af \times \mathrm{R}_{\xi \circ h} \times \mathrm{R}_{\eta \circ h})^n ( x, t_1, t_2 ) \right) = \gamma^n (t_x, t_1, t_2) $$
	for all $n \in \N$. As $h \left|_X \right. (D)$ is monotone, for every $k \in \N$, if $\mathrm{R}_{\theta}^k ( t_x ) \in (c', d')$, then we have $\af^k ( x ) \in D$, which implies (see (\ref{quick fact following the assumption})) that
	$$
	\pi_1 \left( (\af \times \mathrm{R}_{\xi \circ h} \times \mathrm{R}_{\eta \circ h})^k ( x, t_1, t_2 ) \right) \notin U \ \text{and} \ \pi_2 \left( (\af \times \mathrm{R}_{\xi \circ h} \times \mathrm{R}_{\eta \circ h})^k ( x, t_1, t_2 ) \right) \notin V .
	$$
		
	Define $$ \rho_1, \rho_2 \colon \T \times \T \times \T \longrightarrow \T \times \T $$
	by $\rho_1(t_0, t_1, t_2) = t_1$ and $\rho_2(t_0, t_1, t_2) = t_2$. It is easy to check that for all $n \in \N$, we have
	$$ \pi_i \left( (\af \times \mathrm{R}_{\xi \circ h} \times \mathrm{R}_{\eta \circ h})^k ( x, t_1, t_2 ) \right) = \rho_i \left(  \gamma^k (t_x, t_1, t_2) \right) . $$
	Then we have that for every $k \in \N$ such that $\mathrm{R}_{\theta}^k ( t_x ) \in (c', d')$,
	$$ \rho_1 \left(  \gamma^k (td_x, t_1, t_2) \right) \notin U \ \text{and} \ \rho_2 \left(  \gamma^k (t_x, t_1, t_2) \right) \notin V . $$
	According to the definition of the Furstenberg transformation $\gamma$, it follows that
	$$ \{ \gamma^n (t_x, t_1, t_2) \}_{n \in \N} \cap \left( (c', d') \times U \times V \right) = \varnothing , $$
	contradicting the minimality of $\gamma$, which finishes the proof.
\end{proof}
	
	%$h \left|_X \right. \times \id \times \id$ is surjective, minimality of $\gamma$ will imply minimality of.
	
%QQQQ. Finish it. In fact, for every $(x, t_1, t_2) \in X \times \T \times \T$, we have
%	$$ (h \left|_X \right. \times \id \times \id) \circ (\af \times  \mathrm{R}_{\xi \circ h} \times \mathrm{R}_{\eta \circ h})^n (x, t_1, t_2) = (\gamma^n \circ h) (x, t_1, t_2) . $$
%	
%	QQQQ. As $ $

\vspace{2mm}

	The proposition below shows that if the two dynamical systems in Prop \ref{minimality thru the commutative diagram} are minimal, then there is a one-to-one correspondence between the invariant measures on them.

\vspace{2mm}

\begin{prp} \label{one-to-one correspondence between invariant measures}

	If the dynamical systems $(\T \times \T \times \T, \gamma)$ and $(X \times \T \times \T, \af \times \mathrm{R}_{\xi \circ h} \times \mathrm{R}_{\eta \circ h})$ (as in diagram (\ref{factor thru diagram from Denjoy homeomorphism})) are minimal, then there is a one-to-one correspondence between the $\af \times \mathrm{R}_{\xi \circ h} \times \mathrm{R}_{\eta \circ h}$-invariant probability measures and the $\gamma$-invariant probability measures.

%% This is not good (in fact, it is wrong).
%	The map $h \left|_X \right. \times \id \times \id$ maps every Borel subset of $X \times \T \times \T$ to a Borel set in $\T \times \T \times \T$. For any Borel subset $D \subset \T \times \T \times \T$, its preimage $(h \left|_X \right. \times \id \times \id)^{-1}(D)$ is a Borel subset of $X \times \T \times \T$.

\end{prp}

\begin{proof}
%% Original proof not correct.

%	Restricted to $\T \times \T$, $h \left|_X \right. \times \id \times \id$ is exactly the identity map. It remains to show that $h \left|_X \right.$ and $(h \left|_X \right.)^{-1}$ map Borel subsets to Borel subsets.
%
%	As for $h \colon \T \rightarrow \T$, it is continuous, increasing, but not one-to-one. Note that $X = \T \setminus X^c$ with $X^c$ being the union of countable many open sets in $\T$. If we can show that $h$ and $h^{-1}$ maps Borel sets to Borel sets, then we are done.
%
%	In fact, we just need to show that $h$ and $h^{-1}$ map all open intervals $(a, b)$ to Borel sets.
%	
%	As $h$ is increasing, it is easy to see  that $h((a, b))$ must be one of the following:
%	$$ ( h(a), h(b) ) , \ \  ( h(a), h(b) ], \ \ [ h(a), h(b) ) \ \ \text{or} \ [ h(a), h(b) ] . $$
%
%	It is now clear that $h((a, b))$ is a Borel set. Thus $h$ maps Borel sets to Borel sets.
%
%	Similarly, we can show that $h^{-1}$ maps Borel sets to Borel sets, thus finishing the proof.
	
%% Begin of new proof.

	First of all, we will define the correspondence between the $\af \times \mathrm{R}_{\xi \circ h} \times \mathrm{R}_{\eta \circ h}$-invariant probability measures and the $\gamma$-invariant probability measures.

\begin{sloppypar}	
	For simplicity, we use $H$ to denote the function $h \left|_X \right.$ in diagram (\ref{factor thru diagram from Denjoy homeomorphism}). We use \\
	$M_{\af \times \mathrm{R}_{\xi \circ h} \times \mathrm{R}_{\eta \circ h}}$ to denote the set of $\af \times \mathrm{R}_{\xi \circ h} \times \mathrm{R}_{\eta \circ h}$-invariant probability measures on $X \times \T \times \T$ and $M_{\gamma}$ to denote the set of $\gamma$-invariant probability measures on $\T \times \T \times \T$.
\end{sloppypar}
	
	Define
	$$ \varphi : M_{\af \times \mathrm{R}_{\xi \circ h} \times \mathrm{R}_{\eta \circ h}} \longrightarrow M_{\gamma} \ \text{and} \ \psi: M_{\gamma} \longrightarrow M_{\af \times \mathrm{R}_{\xi \circ h} \times \mathrm{R}_{\eta \circ h}} $$
	by
	$$ \varphi(\mu)(D) = \mu \left( ( H \times \id_{\T} \times \id_{\T} )^{-1} (D) \right) \ \text{and} \ \psi(\nu) (E) = \nu \left( ( H \times \id_{\T} \times \id_{\T} ) (E) \right) $$
	for all Borel subsets $D$ of $\T \times \T \times \T$, Borel subsets $E$ of $X \times \T \times \T$, $\mu \in M_{\af \times \mathrm{R}_{\xi \circ h} \times \mathrm{R}_{\eta \circ h}}$ and $\nu \in M_{\gamma}$.
	
	We need to show that the $\varphi$ and $\psi$ above are well-defined.
	
	As every $\mu \in M_{\af \times \mathrm{R}_{\xi \circ h} \times \mathrm{R}_{\eta \circ h}}$ is a probability measure, it follows that $\varphi(\mu) (\T \times \T \times \T) = 1$.
	
	For every Borel subset $D \subset \T \times \T \times \T$, as both $\af \times \mathrm{R}_{\xi \circ h} \times \mathrm{R}_{\eta \circ h}$ and $\gamma$ are homeomorphisms, it follows that
	$$ ( H \times \id_{\T} \times \id_{\T} )^{-1} ( \gamma ( D ) ) =  ( \af \times \mathrm{R}_{\xi \circ h} \times \mathrm{R}_{\eta \circ h} ) \left( ( H \times \id_{\T} \times \id_{\T} )^{-1} ( D ) \right) , $$
	which implies that $\varphi(\mu)$ is $\gamma$-invariant.
		
	For a sequence of Borel subsets $D_1, D_2, \ldots$ of $\T \times \T \times \T$ such that $D_i \cap D_j = \varnothing$ if $i \neq j$, it is clear that $( H \times \id_{\T} \times \id_{\T} )^{-1} (D_1), ( H \times \id_{\T} \times \id_{\T} )^{-1} (D_2), \ldots$ are Borel subsets of $X \times \T \times \T$ (as $H \times \id_{\T} \times \id_{\T}$ is continuous) satisfying $( H \times \id_{\T} \times \id_{\T} )^{-1} (D_i) \cap ( H \times \id_{\T} \times \id_{\T} )^{-1} (D_j) = \varnothing$ if $i \neq j$. Then we have that
	$$ \varphi(\mu) \left( \bigsqcup_{n = 1}^{\infty} D_n \right) = \sum_{n = 1}^{\infty} \varphi(\mu) ( D_n ) . $$
	
	So far, we have shown that $\varphi$ is a well-defined map from $M_{\af \times \mathrm{R}_{\xi \circ h} \times \mathrm{R}_{\eta \circ h}}$ to $M_{\gamma}$.
	
	Now we will check the map $\psi$.
	
	As every $\nu \in M_{\gamma}$ is a probability measure, it follows that
	$$ \psi(\nu) (X \times \T \times \T) = \nu ( \T \times \T \times \T ) = 1 . $$
	
	For every Borel subset $E \subset X \times \T \times \T$, we will show that $\psi(\nu) (E)$ is well-defined. According to the definition of $\psi(\nu)$, we just need to show that $(H \times \id_{\T} \times \id_{\T}) (E) $ is $\nu$-measurable.
	
	For any two open subsets $S_1$ and $S_2$ of $X \times \T \times \T$, we have
	$$ (H \times \id_{\T} \times \id_{\T}) ( S_1 \cup S_2 ) = (H \times \id_{\T} \times \id_{\T}) (S_1) \cup (H \times \id_{\T} \times \id_{\T}) (S_2) , $$
	$$ (H \times \id_{\T} \times \id_{\T}) (S_i^{c}) = ( (H \times \id_{\T} \times \id_{\T}) (S_i) )^{c} \ \ \text{for} \ i = 1, 2. $$
	
	As $H$ is not one-to-one, we cannot get
	$$ (H \times \id_{\T} \times \id_{\T}) ( S_1 \cap S_2 ) = (H \times \id_{\T} \times \id_{\T}) (S_1) \cap (H \times \id_{\T} \times \id_{\T}) (S_2) , $$
	but we still have
	$$  (H \times \id_{\T} \times \id_{\T}) ( S_1 \cap S_2 ) \subset (H \times \id_{\T} \times \id_{\T}) (S_1) \cap (H \times \id_{\T} \times \id_{\T}) (S_2) . $$
	We will consider $ \left( (H \times \id_{\T} \times \id_{\T}) (S_1) \cap (H \times \id_{\T} \times \id_{\T}) (S_2) \right) \setminus (H \times \id_{\T} \times \id_{\T}) ( S_1 \cap S_2 ) . $
	
	Note that $H$ is just the restriction of $h$ to $X$, where $h$ is a noninvertible continuous monotone map from $\T$ to $\T$ (see \cite[Theorem 11.2.7]{KatokHasselblatt}). It follows that $H : X \rightarrow \T$ is one-to-one except at countablely many points of $X$. Use $X_0$ to denote this subset consists of countably many points. Then we have that
	$$ \left( (H \times \id_{\T} \times \id_{\T}) (S_1) \cap (H \times \id_{\T} \times \id_{\T}) (S_2) \right) \setminus (H \times \id_{\T} \times \id_{\T}) ( S_1 \cap S_2 ) \, \subset \, H(X_0) \times \T \times \T . $$
	
	As $\nu ( \T \times \T \times \T ) = 1$ and the minimal action $\gamma$ has the skew product structure, it follows that for every $t \in \T$, $\nu( \{ t \} \times \T \times \T) = 0$, which then implies that $\nu( H(X_0) \times \T \times \T ) = 0$. Then we get that
	$$ \left( (H \times \id_{\T} \times \id_{\T}) (S_1) \cap (H \times \id_{\T} \times \id_{\T}) (S_2) \right) \setminus (H \times \id_{\T} \times \id_{\T}) ( S_1 \cap S_2 ) $$
	is of measure zero for all $\gamma$-invariant measure $\nu$.
	
	For two sets $A$ and $B$, we use $A \bigtriangleup B$ to denote $( A \cap B^c ) \cup ( A^c \cap B)$.
	
%	Then we have that for every two Borel subsets $F$ and $W$ in $X \times \T \times \T$, the set
%	$$ \left( (H \times \id_{\T} \times \id_{\T}) (F) \cap (H \times \id_{\T} \times \id_{\T}) (W) \right) \bigtriangleup (H \times \id_{\T} \times \id_{\T}) ( F \cap W ) $$
%	is of measure zero for any $\gamma$-invariant measure $\nu$.
	
	For every Borel subset $F$ of $X \times \T \times \T$, as $F$ is generated by open sets via taking complements, countably many unions and intersections, it follows that there exists a Borel set $F'$, such that
	 $$ (H \times \id_{\T} \times \id_{\T}) (F) \bigtriangleup F' $$
	is of measure zero for all $\gamma$-invariant measure $\nu$. 	
%	So far, we get that
%	$$ \left( F' \cap W' \right) \bigtriangleup (H \times \id_{\T} \times \id_{\T}) ( F \cap W ) $$
%	is a set of measure zero.
	Note that $F'$ is a Borel set. For every $\gamma$-invariant measure $\nu$, $F'$ is both $\nu$-measurable. It then follows that $(H \times \id_{\T} \times_{\T}) (F)$ is measurable. Recall that
	$$ \psi(\nu) (F) = \nu \left( ( H \times \id_{\T} \times \id_{\T} ) (F) \right) . $$
	It follows that for $\psi(v)$ is well-defined on all the Borel subsets of $X \times \T \times \T$.
	
	For a sequence of Borel subsets $E_1, E_2, \ldots$ of $X \times \T \times \T$ such that $D_i \cap D_j = \varnothing$ if $i \neq j$, and for every $\gamma$-invariant probability measure $\nu$, we will show that
	$$ \psi(\nu) \left( \bigsqcup_{n = 1}^{\infty} E_n \right) = \sum_{n = 1}^{\infty} \psi(\nu) ( E_n ) . $$
	According to the definition, we have
	$$ \psi(\nu) \left( \bigsqcup_{n = 1}^{\infty} E_n \right) = \nu \left( ( H \times \id_{\T} \times \id_{\T} ) \left( \bigsqcup_{n = 1}^{\infty} E_n \right) \right) $$
	Note that
	$$ \left( ( H \times \id_{\T} \times \id_{\T} ) \left( \bigsqcup_{n = 1}^{\infty} E_n \right) \right) = \left( \bigcup_{i = 1}^{\infty} ( H \times \id_{\T} \times \id_{\T} ) ( E_n ) \right) $$
	and
	$$ ( H \times \id_{\T} \times \id_{\T} ) ( E_i ) \cap ( H \times \id_{\T} \times \id_{\T} ) ( E_j ) \subset H(X_0) \times \T \times \T \ \text{for} \ i \neq j . $$
	Recall that $H(X_0) \times \T \times \T$ is a set of measure zero for every $\gamma$-invariant probability measure. It follows that
	$$ \psi(\nu) \left( \bigsqcup_{n = 1}^{\infty} E_n \right) = \sum_{n = 1}^{\infty} \psi(\nu) ( E_n ) . $$
	
	For every Borel subset $E \subset X \times \T \times \T$, according to the commutative diagram (\ref{factor thru diagram from Denjoy homeomorphism}), we have
	$$ ( \gamma \circ ( H \times \id_{\T} \times \id_{\T} ) ) E = \left( ( H \times \id_{\T} \times \id_{\T} ) \circ (\af \times \mathrm{R}_{\xi \circ h} \times \mathrm{R}_{\eta \circ h}) \right) (E) . $$
	It then follows that
	\begin{align*}
		\psi ( \nu ) (E) & = \nu ( ( H \times \id_{\T} \times \id_{\T} ) E ) \\
			& = \nu ( \gamma \left( ( H \times \id_{\T} \times \id_{\T} ) E \right) ) \\
			& = \nu \left( ( H \times \id_{\T} \times \id_{\T} ) ( (\af \times \mathrm{R}_{\xi \circ h} \times \mathrm{R}_{\eta \circ h}) E ) \right) \\
			& = \psi ( \nu ) \left( (\af \times \mathrm{R}_{\xi \circ h} \times \mathrm{R}_{\eta \circ h}) E \right) ,
	\end{align*}
	which implies that $\psi(\nu)$ is $\af \times \mathrm{R}_{\xi \circ h} \times \mathrm{R}_{\eta \circ h}$-invariant.
	
	So far, we have shown that $\psi$ is a well-defined map from $M_{\gamma}$ to $M_{\af \times \mathrm{R}_{\xi \circ h} \times \mathrm{R}_{\eta \circ h}}$.

	Now we will show that for every $\af \times \mathrm{R}_{\xi \circ h} \times \mathrm{R}_{\eta \circ h}$-invariant measure $\mu$ and $\gamma$-invariant measure $\nu$, we have
	$$ (\varphi \circ \psi) (\nu) = \nu \ \text{and} \ (\psi \circ \varphi) (\mu) = \mu . $$
	
	In fact, we just need to show that for every Borel subset $D$ of $\T \times \T \times \T$ and every Borel subset $E$ of $X \times \T \times \T$,
	\begin{equation} \label{checkpoint one}
		\nu \left( ( H \times \id_{\T} \times \id_{\T} ) ( (H \times \id_{\T} \times \id_{\T})^{-1} (D) ) \bigtriangleup D \right) = 0
	 \end{equation}
	and
	\begin{equation} \label{checkpoint two}
		\mu \left( ( H \times \id_{\T} \times \id_{\T} )^{-1} ( (H \times \id_{\T} \times \id_{\T}) (E) ) \bigtriangleup E \right) = 0 .
	\end{equation}
	
	As
	$$ ( H \times \id_{\T} \times \id_{\T} ) ( (H \times \id_{\T} \times \id_{\T})^{-1} (D) ) = D , $$
	the equation (\ref{checkpoint one}) holds.
	
	Note that
	$$ \left( ( H \times \id_{\T} \times \id_{\T} )^{-1} ( (H \times \id_{\T} \times \id_{\T}) (E) ) \bigtriangleup E \right) \subset X_0 \times \T \times \T . $$
	The fact that $X_0$ consists of countably many points and the minimal action $\af \times \mathrm{R}_{\xi \circ h} \times \mathrm{R}_{\eta \circ h}$ has skew product structure implies that
	$$ \mu( X_0 \times \T \times \T ) = 0 . $$
	It then follows that the equation (\ref{checkpoint two}) holds, which finishes the proof.
\end{proof}

\vspace{4mm}

	By Proposition \ref{one-to-one correspondence between invariant measures} above, there is a one-to-one correspondence between the $\af \times \mathrm{R}_{\xi \circ h} \times \mathrm{R}_{\eta \circ h}$-invariant probability measures and the $\gamma$-invariant probability measures
	%QQQ: why?
	(because if two measures coincide on all the Borel sets, they must be the same measure).

	It follows that a minimal Furstenberg transformation on $\T^3$ that is uniquely ergodic
	%QQQQ: (see \cite{Furstenberg})
	will yield an example of a rigid minimal action on $X \times \T \times \T$, and a minimal transformation on $\T^3$ that is not uniquely ergodic will yield an example of a non-rigid minimal action on $X \times \T \times \T$.

\vspace{4mm}

\begin{ex} \label{rigid example}

	This is an example of rigid minimal dynamical system $(X \times \T \times \T, \af \times \mathrm{R}_{\xi} \times \mathrm{R}_{\eta})$.

\end{ex}
	
	Let $(X, \af)$ be a Denjoy homeomorphism with rotation number $\theta_1 \in \R \setminus \Q$.
	
	Choose $\theta_2, \theta_3$ such that $1, \theta_1, \theta_2, \theta_3 \in \R$ are linearly independent over $\Q$. That is, if $\lambda_0, \lambda_1, \lambda_2, \lambda_3 \in \Q$ and satisfy
	$$ \lambda_0 + \lambda_1 \theta_1 + \lambda_2 \theta_2 + \lambda_3 \theta_3 = 0, $$
	then $\lambda_i = 0$ for $i = 0, \ldots, 3$.
	
	The dynamical system $(\T \times \T \times \T, \mathrm{R}_{\theta_1} \times \mathrm{R}_{\theta_2} \times \mathrm{R}_{\theta_3})$ is minimal and uniquely ergodic.
	
	Define $\varphi \colon X \rightarrow \text{Homeo}(\T^2)$ by
	$$ \varphi(x) (z_1, z_2) = (z_1 e^{2 \pi i \theta_2}, z_2 e^{2 \pi i \theta_3}). $$
	
	As $(\T \times \T \times \T, \mathrm{R}_{\theta_1} \times \mathrm{R}_{\theta_2} \times \mathrm{R}_{\theta_3})$ is uniquely ergodic, so is $(X \times \T^2, \af \times \varphi)$. This gives an example of a rigid minimal dynamical system $(X \times \T \times \T, \af \times \mathrm{R}_{\xi} \times \mathrm{R}_{\eta})$.
	
\vspace{4mm}
	
\begin{ex} \label{nonrigid example}
	
	We will give an example of minimal dynamical system $(X \times \T \times \T, \af \times \mathrm{R}_{\xi} \times \mathrm{R}_{\eta})$ such that it is not rigid.

\end{ex}

\begin{sloppypar}	
	According to \cite{Furstenberg} (see page 585), there exists a minimal a Furstenberg transformation
	$$ \gamma_0 \colon \T^2 \longrightarrow \T^2 $$
	such that
	$$ \gamma_0 ( z_1, z_2 ) =  (z_1 e^{2 \pi i \theta}, f(z_1) z_2) \ \text{for some} \ \theta \in \R \setminus \Q \ \text{and} \ \text{contractible} \ f \in C(\T, \T) ,  $$
	and $\gamma_0$ is not uniquely ergodic.
\end{sloppypar}
	
	Let $(\T, \varphi)$ be a Denjoy homeomorphism with rotation number $\theta$. Let $(X, \af)$ be the minimal Cantor dynamical system derived from $(\T, \varphi)$ which factors through $(\T, \mathrm{R}_{\theta})$. In other words, $\af = \varphi \left|_X \right.$ and we have the commutative diagram
	\begin{equation} \label{Denjoy homeomorphism factor thru}
	\xymatrix{
	X \ar@{->}[rr]^{\af} \ar@{->}[d]_{\pi} && X \ar@{->}[d]^{\pi} \\
	\T \ar@{->}[rr]_{\mathrm{R}_{\theta}} && \T & ,
	}
	\end{equation}
	with $\pi \colon X \rightarrow \T$ being a surjective map.
	
	Define $\xi \colon X \rightarrow \text{Homeo}(\T)$ by $\xi(x) (z) =  f( \pi(x) ) z $. We can then check that the following diagram commutes:
	$$
	\xymatrix{
	X \times \T \ar@{->}[rr]^{\af \times \mathrm{R}_{\xi}} \ar@{->}[d]_{\pi \times \id_{\T}} && X \times \T \ar@{->}[d]^{\pi \times \id_{\T}} \\
	\T^2 \ar@{->}[rr]_{\gamma_0} && \T^2 & .
	}
	$$
	
	As $\pi$ is surjective, so is $\pi \times \id_{\T}$. Minimality of $\gamma_0$ then implies minimality of $\af \times \mathrm{R}_{\xi}$. As $\gamma_0$ is not uniquely ergodic, similarly to the proof of Proposition \ref{one-to-one correspondence between invariant measures}, it follows that $(X \times \T, \af \times \mathrm{R}_{\xi})$ is not uniquely ergodic.
	
	In the commutative diagram (\ref{Denjoy homeomorphism factor thru}), note that $\pi$ is onto, and $(\T, \mathrm{R}_{\theta})$ is uniquely ergodic. It follows that $(X, \af)$ is also uniquely ergodic.
	
	As $(X \times \T, \af \times \mathrm{R}_{\xi})$ is not uniquely ergodic, there exist more than one $(\af \times \mathrm{R}_{\xi})$-invariant probability measure. Let $\mu$ and $\nu$ to be two such measures on $X \times \T$ that are different from each other.
	
	According to Lemma \ref{Existence of minimal actions by induction}, there exists $\theta \in \R$ such that if we use $\mathrm{R}_{\theta}$ to denote the function in $C(X, \text{Homeo}(\T))$ defined by
	$$ \mathrm{R}_{\theta}(x) (z) = z e^{2 \pi i \theta} \ \text{for all} \ x \in X \ \text{and} \ z \in \T , $$
	then the dynamical system $(X \times \T \times \T, \af \times \mathrm{R}_{\xi} \times \mathrm{R}_{\theta})$ is still minimal.
	
	Use $m$ to denote the Lebesgue measure on $\T$. For the $(\af \times \mathrm{R}_{\xi})$-invariant probability measures $\mu$ and $\nu$, as $\mathrm{R}_{\theta}$ is a rotation of the circle, we can check that both $\mu \times m$ and $\nu \times m$ are $(\af \times \mathrm{R}_{\xi} \times \mathrm{R}_{\theta})$-invariant probability measures on $X \times \T \times \T$.
	
	As $\mu$ and $\nu$ are different measures, it is clear that $\mu \times m$ is different from $\nu \times m$.
	
	Now we have at least two $(\af \times \mathrm{R}_{\xi} \times \mathrm{R}_{\theta})$-invariant measures. Note that $(X, \af)$ is uniquely ergodic. We have that the dynamical system $(X \times \T \times \T, \af \times \mathrm{R}_{\xi} \times \mathrm{R}_{\theta})$ is not uniquely ergodic.

\vspace{2mm}

\noindent \textbf{Remark:} For this example, the corresponding crossed product \ca \ has tracial rank one and the dynamical system $(X \times \T \times \T, \af \times \mathrm{R}_{\xi} \times \mathrm{R}_{\theta})$ is not rigid. The reason is as follows.
	
	Consider the dynamical system $(X \times \T_1, \af \times \mathrm{R}_{\xi})$. It is not uniquely ergodic. As $(X, \af)$ is uniquely ergodic, it follows that $(X \times \T_1, \af \times \mathrm{R}_{\xi})$ is not rigid.
	
	Use $A$ to denote the crossed product \ca \ $C^*(\Z, X \times \T_1, \af \times \mathrm{R}_{\xi})$. According to Theorem 4.3 of \cite{LM2}, the algebra $A$ has tracial rank one. By Proposition 1.10 (1) of \cite{Ph2}, $\rho_A( K_0(A) )$ is not dense in $\mathrm{Aff}(T(A))$.
	
	Note that $A$ is an A$\T$-algebra. According to Theorem 2.1 of \cite{EGL}, $A$ is approximately divisible. By Theorem 1.4 (e) of \cite{BKR}, and noting that real rank of $A$ is not zero (as tracial rank of $A$ is one and $A$ is A$\T$-algebra), we have that the projections in $A$ does not separate traces of $A$. In other words, there exist two $(\af \times \mathrm{R}_{\xi})$-invariant measures $\mu$ and $\nu$ such that
	$$ \mu \neq \nu , \ \text{and} \ \mu(x) = \nu(x) \ \text{for all} \ x \in K_0(A) . $$
	
	Define measures $\mu_X, \nu_X$ by
	$$ \mu_X( D ) = \mu( D \times \T ) \ \ \text{and} \ \ \nu_X(D) = \nu( D \times \T ) $$
	for all Borel sets $D \subset X$. It is clear that both $\mu_X$ and $\nu_X$ are $\af$-invariant probability measures on $X$.
	
%	By Lemma 2.4 of \cite{LM1},
%	$$ K_0(A) \cong C(X, \Z) / \{ f - f \circ \af \colon f \in C(X, \Z) \} \oplus \Z . $$
	
	Note that $C(X, \Z)$ is generated by the projections in $C(X)$. Also note that the $\C$-linear span of $C(X, \Z)$ is dense in $C(X, \R)$. The fact that the projections in $A$ do not separate $\mu$ and $\nu$ implies that $C(X, \Z)$ do not separate $\mu_X$ and $\nu_X$, which then implies that $\mu_X = \nu_X$.
	
	Use $B$ to denote 	$C^*(\Z, X \times \T_1 \times \T_2, \af \times \mathrm{R}_{\xi} \times \mathrm{R}_{\theta})$. Let $m$ be the Lebesgue measure on $\T$. It is clear that $\mu \times m$ and $\nu \times m$ are two $(\af \times \mathrm{R}_{\xi} \times \mathrm{R}_{\theta})$-invariant probability measures.
	
	We will show that the projections in $B$ do not separate $\mu \times m$ and $\nu \times m$.	 
	
	From Proposition \ref{K-data of crossed product C-algebra},
	\begin{equation} \label{K0 of B}
		K_0(B) \cong C(X, \Z^2) / \{(f,g ) - (f, g) \circ \af^{-1} \colon f, g \in C(X, \Z) \} \oplus \Z \oplus \Z.
	\end{equation}

	The two copies of $\Z$ correspond to the two generalized Rieffel projections $e_1$ and $e_2$, given by $e_1 = g_1 u^* + f_1 + u g_1$, and $e_2 = g_2 u^* + f_2 + u g_2$, where $e_i, f_i, g_i$ are defined similarly to the functions defined in Section 6 of \cite{LM1}, $f_1 (x, z_1, z_2) = f_1 (x, z_1, z_2' )$ and $f_2 (x, z_1, z_2) = f_1 (x, z_1', z_2 )$ for all $z_1, z_1' \in \T_1, z_2, z_2' \in \T_2$.
	
	As the projections in $A$ do not distinguish $\mu$ and $\nu$, it follows that the elements in $K_0(B)$ that correspond to the first two summands of \ref{K0 of B} do not separate $\mu \times m$ and $\nu \times m$.
	
	For the generalized Rieffel projection $e_2$, as $f_2(x, z_1, z_2)$ is independent of $z_1$, we have
	$$ f(x, z_1, z_2) = F_2(x, z_2) \ \mbox{for some} \ F \in C(X \times \T_2, \R) . $$
	
	Recall that for a measure $\sigma$ on $X$ and $f \in C(X)$, we use $\sigma(f)$ to denote $\int_X f(x) \ \mathrm{d} \mu$ (see Section \ref{Sec Introduction and Notation}).  We check that
\vspace{-5mm}	
	\begin{align*}
		(\mu \times m) (e_2) & = (\mu \times m) (f_2) \vspace{2mm} \\
			& = \displaystyle\int_{(X \times \T_1) \times \T_2} f_2(x, z_1, z_2) \ \mathrm{d} ( \mu \times m ) \vspace{2mm} \\
			& = \displaystyle \int_{X \times \T_2} F_2(x, z_2) \ \mathrm{d} ( \mu_X \times m ) \vspace{2mm} \\
			& = \displaystyle \int_{X \times \T_2} F_2(x, z_2) \ \mathrm{d} ( \nu_X \times m ) \vspace{2mm} \\
			& = \displaystyle\int_{(X \times \T_1) \times \T_2} f_2(x, z_1, z_2) \ \mathrm{d} ( \nu \times m ) \vspace{2mm} \vspace{2mm} \\
			& = (\nu \times m) (f_2) \\
			& = (\nu \times m) (e_2) .
	\end{align*}
	
	Then we have shown that $e_2$ does not separate $\mu \times m$ and $\nu \times m$ either, which then implies that the projections in $B$ cannot separate traces of $B$.
	
	According to Theorem 1.4 of \cite{BKR}, the real rank of $B$ is not zero. Then it follows that the tracial rank of $B$ is not zero.
	
	By Theorem \ref{tracial rank of A without rigidity}, the tracial rank of $B$ must be one.

	According to Proposition \ref{rigidity implies tracial rank zero for rotation case}, the dynamical system $(X \times \T \times \T, \af \times \mathrm{R}_{\xi} \times \mathrm{R}_{\theta})$ is not rigid.

%	QQQQ:
%
%	Example (3) : We will give an example of minimal dynamical system $(X \times \T \times \T, \af \times \mathrm{R}_{\xi} \times \mathrm{R}_{\eta})$ such that it is not rigid, and the projection does not separate traces in the corresponding crossed product \ca.

%\chapter{APPROXIMATE K-CONJUGACY}

% OF TWO MINIMAL DYNAMICAL SYSTEMS ON \texorpdfstring{$X \times \T \times \T$}{X times T times T} }

\vspace{2cm}

\section{Approximate Conjugacies} \label{Approximate conjugacies}

	In this section, we start with a sufficient condition for approximate K-conjugacy between two minimal dynamical systems $(X \times \T \times \T, \af \times \mathrm{R}_{\xi_1} \times \mathrm{R}_{\eta_1})$ and $(X \times \T \times \T, \bt \times \mathrm{R}_{\xi_2} \times \mathrm{R}_{\eta_2})$. Then we give an if and only if condition for weak approximate conjugacy of these two dynamical systems, showing that weak approximate conjugacy just depends on $\af$ and $\bt$. In Theorem \ref{approximate K-conjugacy for rotation case}, an if and only if condition for approximate K-conjugacy between these two dynamical systems is given.

\vspace{2mm}

\begin{sloppypar}
	In \cite{LM3}, several notions of approximate conjugacy between dynamical systems are introduced. In \cite{LM1}, it is shown that for rigid minimal systems on $X \times \T$ (with $X$ being the Cantor set and $\T$ being the circle; see Definition 3.1 of \cite{LM1}), the corresponding crossed product \ca s are isomorphic if and only if the dynamical systems are approximately K-conjugate.
\end{sloppypar}

	For two minimal rigid dynamical systems $(X \times \T \times \T, \af \times \mathrm{R}_{\xi} \times \mathrm{R}_{\eta})$ and $(X \times \T \times \T, \bt \times \mathrm{R}_{\xi_1} \times \mathrm{R}_{\eta_1})$, we study the relationship between approximate K-conjugacy and the isomorphism of crossed product \ca s.

	We start with basic definitions and facts about conjugacy and approximate conjugacy.

\begin{dfn}

	Let $X, Y$ be two compact metric spaces, and let $\af \in \mathrm{Homeo}(X)$ and $\bt \in \mathrm{Homeo}(Y)$ be two minimal actions. We say that $(X, \af)$ and $(Y, \bt)$ are conjugate if there exists $\sigma \in \mathrm{Homeo}(X, Y)$ such that $\sigma \circ \af = \bt \circ \sigma$. We say that $(X, \af)$ and $(Y, \bt)$ are flip conjugate if $(X, \af)$ is conjugate to $(Y, \bt)$ or $(Y, \bt^{-1})$.

\end{dfn}

\begin{dfn}
 
	Let $X, Y$ be two compact metric spaces, and let $\af \in \mathrm{Homeo}(X)$ and $\bt \in \mathrm{Homeo}(Y)$ be two minimal actions. We say that $(X, \af)$ and $(Y, \bt)$ are weakly approximately conjugate if there exist $\sigma_n \in \mathrm{Homeo}(X, Y)$ and $\gamma_n \in \mathrm{Homeo}(Y, X)$ for $n \in \N$ such that
	$$ \dist( f \circ \sigma_n \circ \af, f \circ \bt \circ \sigma_n ) \rightarrow 0 \ \  \text{and} \ \ \dist( g \circ \af \circ \gamma_n, g \circ \gamma_n \circ \bt) \rightarrow 0 \ \ \text{as} \  n \rightarrow \infty$$
	for all $f \in C(X)$ and $g \in C(Y)$, where $\dist(f_1, f_2)$ is defined to be $\sup_{x \in D} \dist( f_1(x), f_2(x) )$ for all continuous functions $f_1, f_2$ on the metric space $D$.
 
\end{dfn}

\begin{sloppypar}
	It is clear that if two minimal dynamical systems are conjugate, then they are weakly approximately conjugate. Generally speaking, the inverse implication does not hold.
\end{sloppypar}

\vspace{4mm}

	Now we will recall the definition of $C^*$-strong approximate conjugacy (which is defined by Huaxin Lin in \cite{Lin6}).

	Given minimal dynamical systems $(X, \af)$ and $(Y, \bt)$, if they are flip conjugate, then it is easy to check that the corresponding crossed product \ca s $C^*(\Z, X, \af)$ and $C^*(\Z, Y, \bt)$ are isomorphic.

	According to \cite{Tomiyama} (Corollary of Theorem 2), for two minimal dynamical systems $(X, \af)$ and $(Y, \bt)$, there exists an isomorphism
	$$ \varphi \colon C^*(\Z, X, \af) \longrightarrow C^*(\Z, Y, \bt) $$
satisfying $\varphi(C(X)) = C(Y)$ if and only if these two dynamical systems are flip conjugate.

\vspace{1cm}

	In view of Tomiyama's result above, $C\sp*$-strong approximate flip conjugacy is defined as below.

\vspace{2mm}

\begin{dfn}[See \cite{Lin6}]

\begin{sloppypar}
	Let $(X, \af)$ and $(X, \bt)$ be two minimal dynamical systems such that \\ $\mathrm{TR}(C\sp*(\Z, X, \af)) = \mathrm{TR}(C\sp*(\Z, X, \bt)) = 0$, we say that $(X, \af)$ and $(X, \bt)$ are $C\sp*$-strongly approximately flip conjugate if there exists a sequence of isomorphisms
	$$ \varphi_n \colon C\sp*(\Z, X, \af) \rightarrow C\sp*(\Z, X, \bt) , \ \psi_n \colon C\sp*(\Z, X, \bt) \rightarrow C\sp*(\Z, X, \af) $$
	and a sequence of isomorphisms $\chi_n, \lambda_n \colon C(X) \rightarrow C(X)$ such that
\end{sloppypar}

	1) $[\varphi_n] = [\varphi_m] = [\psi_n^{-1}]$ in $KL( C^*(\Z, X, \af) , C^*(\Z, X, \af) )$ for all $m, n \in \N$,

	2) $\displaystyle\lim_{n \rightarrow \infty} \norm{\varphi_n \circ j_{\af}(f) - j_{\bt} \circ \chi_n(f)} = 0$ and $\displaystyle\lim_{n \rightarrow \infty} \norm{\psi_n \circ j_{\bt}(f) - j_{\af} \circ \lambda_n(f)} = 0$ for all $f \in C(X)$, with $j_{\af}, j_{\bt}$ being the injections from $C(X)$ into $C^*(\Z, X, \af)$ and $C^*(\Z, X, \bt)$.

\end{dfn}

\vspace{1cm}

	Some notation will be introduced before the next result about $C\sp*$-strong approximate conjugacy.

	\removeme[figure out all the details]{Let $A$ be a separable amenable \ca \ that satisfies Universal Coefficient Theorem. For $\theta \in KL(A, B)$, there are induced homomorphisms $\Gamma(\theta)_i \colon K_i(A) \rightarrow K_i(B)$ for $i = 0, 1$. Define $\rho_A \colon A_{sa} \longrightarrow \mathrm{Aff}(T(A))$ by $\rho_A(a)(\tau) = \tau(a)$ for all $a \in A_{sa}$ and  $\tau \in T(A)$. Suppose $A$ and $B$ are two unital simple \ca s with tracial rank zero and $\gamma \colon K_0(A) \rightarrow K_0(B)$ is an order preserving homomorphism. As $A$ has real rank zero, $\gamma$ will induce a positive homomorphism $\gamma_{\rho} \colon \mathrm{Aff}(T(A)) \rightarrow \mathrm{Aff}(T(B))$}.

	The theorem below (\cite[Theorem 2.5]{Lin6}) gives one necessary condition for $C\sp*$-strong approximate flip conjugacy between two crossed product \ca s.

\vspace{2mm}

\begin{thm}

%\begin{sloppypar}
	Let $(X, \af)$ and $(X, \bt)$ be two minimal dynamical systems such that the corresponding crossed product \ca s \ $A_{\af}$ and $A_{\bt}$ both have tracial rank zero. Then $\af$ and $\bt$ are $C\sp*$-strongly approximately flip conjugate if the following holds: There is an isomorphism $\chi \colon C(X) \rightarrow C(X)$ and there is $\theta \in KL(A_{\af}, A_{\bt})$ such that $\Gamma(\theta)$ gives an isomorphism
	$$ \Gamma(\theta) \colon (K_0(A_{\af}), K_0(A_{\af})_+, [1], K_1(A_{\af})) \rightarrow (K_0(A_{\bt}), K_0(A_{\bt})_+, [1], K_1(A_{\bt})) , $$
	and such that
	$$ [j_{\af}] \times \theta = [j_{\bt} \circ \chi] \ \text{in} \ KL(C(X), A_{\bt})  $$
and
	$$\rho_{A_{\bt}} \circ j_{\bt} \circ \chi(f) = ((\Gamma(\theta)_0)_{\rho}) \circ \rho_{A_{\af}} \circ j_{\af}(f)$$
for all $f \in C(X)_{sa}$.
%\end{sloppypar}

\end{thm}

\vspace{2mm}

	If $K_i(C(X))$ is torsion free, then a simplified version of this result holds (\cite[Corollary 2.6]{Lin6}).

\vspace{2mm}

\begin{cor} \label{C strongly flip conjugate for TR zero case}

	Let $X$ be a compact metric space with torsion free $K$-theory. Let $(X, \af)$ and $(X, \bt)$ be two minimal dynamical systems such that $\mathrm{TR}(A_{\af}) = \mathrm{TR}(A_{\bt}) = 0$. Suppose that there is an order isomorphism that maps $[1_{A_{\af}}]$ to $[1_{A_{\bt}}]$:
	$$ \gamma \colon (K_0(A_{\af}), K_0(A_{\af})_+, [1_{A_{\af}}], K_1(A_{\af})) \rightarrow (K_0(A_{\bt}), K_0(A_{\bt})_+, [1_{A_{\bt}}], K_1(A_{\bt})) , $$
	such that there exists an isomorphism $\chi \colon C(X) \rightarrow C(X)$ satisfying
	$$ \gamma \circ (j_{\af})_{*i} = (j_{\bt} \circ \chi)_{*i}  \ \text{for} \ i = 0, 1 \ \text{and} \ \gamma_{\rho} \circ j_{\af} = \rho_{A_{\bt}} \circ j_{\bt} \circ \chi \ \text{on} \ C(X)_{sa} . $$
	Then $(X, \af)$ and $(X, \bt)$ are $C\sp*$-strongly approximately flip conjugate.

\end{cor}

\vspace{5mm}

	In the rest of this section, for a minimal homeomorphism $\af$ on the Cantor set $X$, we will use $K^0(X, \af)$ to denote the ordered group
	$$ C(X, \Z^2) / \{ f - f \circ \af^{-1} \colon f \in C(X, \Z^2) \} $$
	with the positive cone being (denoted by $K^0(X, \af)_+$)
	$$ C(X, D) / \{ f - f \circ \af^{-1} \colon f \in C(X, \Z^2) \} $$
	where $D$ is as defined in Lemma \ref{K_i of A_x}. In $K^0(X, \af)$, we define the unit element to be
	$$ [ (1, 0)_{C(X, \Z^2)} ] \in C(X, \Z^2) / \{ f - f \circ \af^{-1} \colon f \in C(X, \Z^2) \} , $$
	with $(1, 0)_{C(X, \Z^2)}$ being the constant function in $C(X, \Z^2)$ that maps every $x \in X$ to $(1, 0) \in \Z^2$. We use $1_{K^0(X, \af)}$ to denote this unit element.
	
\vspace{5mm}

\begin{lem} \label{technical lifting lemma}

	Let $X$ be the Cantor set. For every minimal action $\af \in \text{Homeo}(X)$, if there is an order isomorphism %QQQQ.
	$$ \varphi \colon (K^0(X, \af), K^0(X, \af)_+, 1_{K^0(X, \af)}) \longrightarrow (K^0(X, \bt), K^0(X, \bt)_+, 1_{K^0(X, \bt)}), $$ then there is an order isomorphism
	$$ \widetilde{\varphi} \colon (C(X, \Z^2), C(X, D), (1, 0)_{C(X, \Z^2)}) \longrightarrow (C(X, \Z^2), C(X, D), (1, 0)_{C(X, \Z^2)}) $$
	such that the following diagram commutes:
	\begin{equation} \label{diag 1}
%	$$
	\xymatrix{
		(C(X, \Z^2), C(X, D)) \ar@{->}[rr]^{\widetilde{\varphi}} \ar[d]_{\pi_{\af}} && (C(X, \Z^2), C(X, D)) \ar[d]^{\pi_{\bt}} \\
		(K^0(X, \af), K^0(X, \af)_+) \ar@{->}[rr]^{\varphi} && (K^0(X, \bt), K^0(X, \bt)_+) \ \ \ ,
	}
%	$$
	\end{equation}
	where $\pi_{\af}, \pi_{\bt}$ are the canonical projections from $C(X, \Z^2)$ to $K^0(X, \af)$ and $K^0(X, \bt)$. In fact, there exists $\sigma \in \text{Homeo}(X)$ such that $\varphi(F) = F \circ \sigma^{-1}$ for all $F \in C(X, \Z^2)$.

\end{lem}

\begin{proof}
	The proof is based on \cite[Theorem 2.6]{LM3}.

	Define $\overline{K^0(X, \af)}$ to be
	$$ C(X, \Z) / \{ g - g \circ \af^{-1} \colon g \in C(X, \Z) \} $$
	and $\overline{K^0(X, \af)}_+$ to be
	$$ C(X, \Z^{+} \cup \{ 0 \} ) / \{ g - g \circ \af^{-1} \colon g \in C(X, \Z) \} . $$
	
	We can check that $(\overline{K^0(X, \af)}, \overline{K^0(X, \af)}_+)$ gives an ordered group with order unit.

	Define
	$$ h \colon K^0(X, \af) \rightarrow \overline{K^0(X, \af)} \ \text{by} \ h([f]) = [f_{1}] $$
	for every $f = (f_1, f_2) \in C(X, \Z^2)$, with $\ f_1, f_2 \in C(X, \Z)$.
	
	From the definition, we can check that $h$ is surjective and $h(K^0(X, \af)_+) = \overline{K^0(X, \af)}_+$.

	For the isomorphism
	$$ \varphi \colon (K^0(X, \af), K^0(X, \af)_+) \rightarrow (K^0(X, \bt), K^0(X, \bt)_+) , $$
	define
	$$ \varphi_0 \colon \overline{K^0(X, \af)} \rightarrow \overline{K^0(X, \bt)} \ \text{by} \   \varphi_0([f]) = h(\varphi([(f, 0)])) $$
	for all $f \in C(X, \Z)$.

	Suppose that there exist $f_1, f_2, g \in C(X, \Z)$ such that $f_1 - f_2 = g - g \circ \af^{-1}$. Then it follows that $(f_1, 0) - (f_2, 0) = (g, 0) - (g, 0) \circ \af^{-1}$, which implies that $\varphi([(f_1, 0 )]) = \varphi([(f_1, 0 )])$. It is now clear that $\varphi_0$ is well-defined.

	Note that $\varphi_0([1_{C(X, \Z)}]) = h(\varphi([ ( 1, 0 )_{C(X, \Z^2)} ]))$. As $\varphi$ is unital, $\varphi( 1_{K^0(X, \af)} ) = 1_{K^0(X, \bt)}$, which then implies that $\varphi_0([1_{C(X, \Z)}]) = h( [ ( 1, 0 )_{C(X, \Z^2)} ] ) = [1_{C(X, \Z)}]$. We can now claim that $\varphi_0$ is unital.

	For any $f \in C(X, \Z^+ \cup \{ 0 \})$, $\varphi_0([f]) = h(\varphi([(f, 0)]))$. As both $\varphi$ and $h$ are order preserving, $\varphi_0$ is also order preserving.

	So far, we have that $\varphi_0 \colon \overline{K^0(X, \af)} \rightarrow \overline{K^0(X, \bt)}$ is untial and order preserving. According to \cite[Theorem 2.6]{LM3}, there exists a continuous order preserving map
	$$ \widetilde{\varphi_0} \colon (C(X, \Z), C(X, \Z)_+, 1_{C(X, \Z)}) \rightarrow (C(X, \Z), , C(X, \Z)_+, 1_{C(X, \Z)}) , $$
such that the following diagram commutes:
	\begin{equation} \label{diag 2}
	\xymatrix{
		(C(X, \Z), C(X, Z)_+) \ar@{->}[rr]^{\widetilde{\varphi_0}} \ar[d]_{\pi'_{\af}} && (C(X, \Z), C(X, Z)_+) \ar[d]^{\pi'_{\bt}} \\
		(\overline{K^0(X, \af)}, \overline{K^0(X, \af)}_+) \ar@{->}[rr]^{\varphi_0} && (\overline{K^0(X, \bt)}, \overline{K^0(X, \bt)}_+) \ \ \ .
	}
	\end{equation}

	Now we need to construct the unital positive linear map
	$$ \widetilde{\varphi} \colon (C(X, \Z^2), C(X, D)) \rightarrow (C(X, \Z^2), C(X, D)) , $$
	such that diagram (\ref{diag 1}) commutes.

%	$$
%	\xymatrix{
%		(C(X, \Z^2), C(X, D)) \ar@{->}[rr]^{\widetilde{\varphi}} \ar[d]_{\pi_{\af}} && (C(X, \Z^2), C(X, D)) \ar[d]^{\pi_{\bt}} \\
%		(K^0(X, \af), K^0(X, \af)_+) \ar@{->}[rr]^{\varphi} && (K^0(X, \bt), K^0(X, \bt)_+) \ \ \ .
%	}  \hfill
%	$$

	For the $\widetilde{\varphi_0}$ we get, note that $\widetilde{\varphi_0}$ is a unital positive isomorphism from $K_0(C(X))$ to $K_0(C(X))$. As $C(X)$ is a unital AF-algebra, by the existence theorem of classification of unital AF-algebras, there exists an isomorphism $\psi \colon C(X) \rightarrow C(X)$ such that (identifying $K_0( C(X) )$ with $C(X, \Z)$ and $K_0( C(X) )_+$ with $C(X, \Z)_+$)
	$$ \psi_{*0} \colon (C(X, \Z), C(X, \Z)_+, [1]) \rightarrow (C(X, \Z), C(X, \Z)_+, [1]) $$
	coincides with $\widetilde{\varphi_0}$.

	As $\psi$ is an isomorphism, there exists $\sigma \colon X \rightarrow X$ such that $\psi(f) = f \circ \sigma^{-1}$ for all $f \in C(X)$.

	Define $\widetilde{\varphi} \colon C(X, \Z^2) \rightarrow C(X, \Z^2)$ by $\widetilde{\varphi}((f, g)) =  (\psi(f), \psi(g)) $ for all $f, g \in C(X, \Z)$. In other words, $\widetilde{\varphi}((f, g)) = (f, g) \circ \sigma^{-1}$ for all $(f, g) \in C(X, \Z^2)$.

	For the $\widetilde{\varphi}$ above-defined, it is easy to check that it is unital and linear. It remains to show  that $\widetilde{\varphi}$ maps positive cone to positive cone, and makes the diagram commute.

	For every $(f, g) \in C(X, D)$, we get $\widetilde{\varphi} ( (f, g) ) = (f, g) \circ \sigma^{-1}$. As $(f, g) \in C(X, D)$, it is clear that $(f, g) \circ \sigma^{-1} \in C(X, D)$. So far, we proved that $\widetilde{\varphi}$ is a positive map.

	We can check that
	\begin{equation*}
	\begin{split}
	\pi_{\beta} \circ \widetilde{\varphi} ( (f, g) ) & = \pi_{\beta} ( h(f), h(g) )\\
		& = \pi_{\beta} ( \widetilde{\varphi_0}(f), \widetilde{\varphi_0}(g) ) \\
		& = \pi_{\beta}(\widetilde{\varphi_0}(f), 0) + \pi_{\beta}(0, \widetilde{\varphi_0}(g)) \\
		& = (\pi'_{\beta} \circ \widetilde{\varphi_0}(f), 0) + (0, \pi'_{\beta} \circ \widetilde{\varphi_0}(g)) \\
		& = (\varphi_0 \circ \pi'_{\af} (f), 0) + (0, \varphi_0 \circ \pi'_{\af}(g)) \\
		& = \varphi \circ \pi_{\af} ( (f, 0) ) + \varphi \circ \pi_{\af} ( (0, g) ) \\
		& = \varphi \circ \pi_{\af} ( (f, g) ) ,
	\end{split}
	\end{equation*}
	which implies the commutativity of diagram (\ref{diag 1}).

	As $\widetilde{\varphi} ( (f, g) ) = (f, g) \circ \sigma^{-1}$ for all $f, g \in C(X, \Z)$,  we get that $\widetilde{\varphi}$ is an isomorphism, which finishes the proof.
\end{proof}

\begin{thm} \label{C*-strong approximate flip conjugacy}

	Let  $(X \times \T \times \T, \af \times \mathrm{R}_{\xi_1} \times \mathrm{R}_{\eta_1})$ and $(X \times \T \times \T, \bt \times \mathrm{R}_{\xi_2} \times \mathrm{R}_{\eta_2})$ be two minimal rigid Cantor dynamical systems. Use $A$, $B$ to denote the two corresponding crossed product \ca s. According to Proposition \ref{K-data of crossed product C-algebra}, $K^0(X, \af)$ is a direct summand of $K_0(A)$ and $K^0(X, \bt)$ is a direct summand of $K_0(B)$. Let
	$$ j_A : K^0(X, \af) \rightarrow K_0(A) \cong K^0(X, \af) \oplus \Z^2 \ \ \text{and} \ \ j_B: K^0(X, \bt) \rightarrow K_0(B) \cong K^0(X, \af) \oplus \Z^2  $$
	be defined by
	$$ j_A(x) = (x, 0) \ \text{and} \ j_B(x) = (x, 0) . $$
	If there is an order preserving isomorphism $\rho$ from $K_0(A)$ to $K_0(B)$ that maps $K^0(X, \af)$ onto $K^0(X, \bt)$, then these two dynamical systems are $C\sp*$-strongly approximately conjugate.

\end{thm}

\begin{proof}
	We have the following commutative diagram:
	$$
	\xymatrix{
		K_0(A) \ar@{->}[rr]^{\rho} && K_0(B) \\
		K^0(X, \af) \ar@{->}[u]^{j_{A}} \ar@{->}[rr]^{\rho \left|_{K^0(X, \af)} \right.} && K^0(X, \bt) \ar@{->}[u]_{j_B} \ .
	}
	$$

	According to Lemma \ref{technical lifting lemma}, we can lift
	$$ \rho \left|_{K^0(X, \af)} \right. \colon K^0(X, \af) \longrightarrow K^0(X, \bt) $$
	to
	$$ \widetilde{\rho} \colon C(X, \Z^2) \longrightarrow C(X, \Z^2) , $$	
	which will yield the commutative diagram
	$$
	\xymatrix{
		K_0(A) \ar@{->}[rr]^{\rho} && K_0(B) \\
		K^0(X, \af) \ar@{->}[u]^{j_{A}} \ar@{->}[rr]^{\rho \left|_{K^0(X, \af)} \right.} && K^0(X, \bt) \ar@{->}[u]_{j_B} \\
		C(X, \Z^2) \ar@{->}[u]^{\pi_{\af}} \ar@{->}[rr]^{\widetilde{\rho}} && C(X, \Z^2) \ar@{->}[u]_{\pi_{\bt}} \ .
	}
	$$
	In fact, according to Lemma \ref{technical lifting lemma}, there exists $\sigma \in \text{Homeo}(X)$ such that $\widetilde{\rho} (F) = F \circ \sigma^{-1}$. Define
	$$ \chi \colon C(X \times \T^2) \rightarrow C(X \times \T^2) $$
	by $\chi(f) = f \circ (\sigma \times \id_{\T^2})$ for all $f \in C(X \times \T^2)$.

	According to the K\"unneth Theorem, we get that $K_0(C(X \times \T^2)) \cong C(X, \Z^2)$. By Lemma \ref{K_0 of C(T^2)}, if we identify $K_0(C(X \times \T^2))$ with $C(X, \Z^2)$, the positive cone will be identified with $C(X, D)$, with $D$ as defined in Lemma \ref{K_0 of C(T^2)}. Choose $x \in X$. According to Lemma \ref{K_i of A_x}, we know that $K_0(A_x) \cong K^0(X, \af)$ and $K_0(B_x) \cong K^0(X, \bt)$, with $A_x$, $B_x$ being the subalgebras of $A$ and $B$, as in Definition \ref{dfn of Ax}.

	Now we have the commutative diagram
	$$
	\xymatrix{
		K_0(A) \ar@{->}[rr]^{\rho} && K_0(B) \\
		K_0(C(X \times \T^2)) \ar[u]^{(j_{\af})_{*0}} \ar@{->}[rr]^{\widetilde{\rho}} && K_0(C(X \times \T^2)) \ar[u]_{(j_{\bt})_{*0}} \ .
	}
	$$

	Note that $\widetilde{\rho}$ is induced by the $\chi \colon C(X \times \T^2) \rightarrow C(X \times \T^2)$ defined above. We have shown that $\rho \circ (j_{\af})_{*i} = (j_{\bt} \circ \chi)_{*i} , \ \text{i = 0, 1}$.

	We will show that $\gamma_{\rho} \circ j_{\af} = \rho_{A_{\bt}} \circ j_{\bt} \circ \chi \ \text{on} \ C(X)_{sa}$.

	For every tracial state $\tau \in T( C^*(\Z, X, \bt) )$, we know that it corresponds to a $\bt$-invariant probability meausure $\mu_B$ (in such sense that $\tau(a) = \mu(E(a))$, with $E$ being the conditional expectation from $C^*(\Z, X, \bt)$ to $C(X)$).

	For every $\bt$-invariant probability measure $\mu_B$ on $X$, if we use $v$ to denote standard Lebesgue measure on $\T$, it is then clear that $\mu_B \times v \times v$ is $\bt \times \mathrm{R}_{\xi_2} \times \mathrm{R}_{\eta_2}$-invariant. As the dynamical system $(X \times \T \times \T, \bt \times \mathrm{R}_{\xi_2} \times \mathrm{R}_{\eta_2})$ is rigid, for every $\bt \times \mathrm{R}_{\xi_2} \times \mathrm{R}_{\eta_2}$-invariant probability measure, it must be $\mu \times v \times v$, with $\mu$ being an $\bt$-invariant probability measure and $v$ being the Lebesgue probability measure.
	
	Note that $A$ denotes $C^*(\Z, X \times \T \times \T, \af \times \mathrm{R}_{\xi_1} \times \mathrm{R}_{\eta_1})$ and $B$ denotes $C^*(\Z, X \times \T \times \T, \bt \times \mathrm{R}_{\xi_1} \times \mathrm{R}_{\eta_1})$.
	According to Proposition \ref{K-data of crossed product C-algebra}, the fact that $K_0( A )$ is isomorphic to $K_0( B )$ implies that $K_1( A )$ is also isomorphic to $K_1( B )$. According to Proposition \ref{rigidity implies tracial rank zero for rotation case}, the tracial rank of $A$ and $B$ are both zero, thus classifiable via the K-data.
	
	Let $\varphi \colon A \rightarrow B$ be the \ca \ isomorphism such that
	$$ \varphi_{*0} \colon K_0( A ) \longrightarrow K_0( B ) $$
	coincides with the $\rho$ in the statement. Define
	$$ \varphi^* \colon T( B ) \longrightarrow T( A ) $$
	as $\varphi^*(\tau_B) (a) = \tau_B (\varphi(a) )$ for all $a \in A$ and $\tau_B \in T( B )$.

	Note that a \ca \ with tracial rank zero must have real rank zero. We can now claim that for every $a \in C^*(\Z, X, \af)_{sa}$ and $\tau_B \in T(B)$ given by $\mu_B \times v \times v$,
	$$ ( \gamma_{\rho} \circ j_{\af} (a) ) (\tau_B) = \varphi^*(\tau_B) (a) . $$

	Consider
	$$ a = f \otimes g \otimes h \in C(X \times \T \times \T)_{sa} \subset A_{sa} $$
	with $f \in C(X)_{sa}, g \in C(\T)_{sa}$ and $h \in C(\T)_{sa}$, and use $\tau_A$ to denote $\varphi^*(\tau_B)$. As $\af \times \mathrm{R}_{\xi_1} \times \mathrm{R}_{\eta_1}$ is rigid, there exists an $\af$-invariant measure $\mu_A$ such that $\tau_A (a) = (\mu_A \times v \times v) ( E( a ) )$, with $E$ being the conditional expectation from $A$ to $C(X \times \T \times \T)$ and $v$ being the Lebesgue measure on the circle.
%	For every $\bt \times \mathrm{R}_{\xi_2} \times \mathrm{R}_{\eta_2}$-invariant measure $\mu \times v \times v$ (with $\mu$ and $v$ as defined above), use $\tau_B$ to denote the corresponding tracial state on $B$, then it will determine a tracial $\tau$ state on $A$, with $\tau$ defined by
%	$$ \tau_A(a) =  \tau_B( \varphi(a) ) . $$
	It follows that $( \gamma_{\rho} \circ j_{\af} (a) ) (\tau_B) = \tau_A (a) = \mu_A (f) \cdot v(g) \cdot v(h)$.

	As for $( (\rho_{A_{\bt}} \circ j_{\bt} \circ \chi) (a) ) (\tau_B)$, we know from the definition that
	$$ ( (\rho_{A_{\bt}} \circ j_{\bt} \circ \chi) (a) ) (\tau_B) = \tau_B ( \chi (f \otimes g \otimes h) ) =  (\mu_B \times v \times v) (\chi(f \otimes g \otimes h )). $$
	Recall the definition of $\chi$. We have
	$$ (\mu_B \times v \times v) (\chi(f \otimes g \otimes h )) = \mu_B(f \circ \sigma^{-1}) \cdot v(g) \cdot v(h). $$
	If we can show that $\mu_B(f \circ \sigma^{-1}) = \mu_A(f)$, then it follows that
	$$ (\mu_B \times v \times v) (\chi(f \otimes g \otimes h )) = \mu_A(f) \cdot v(g) \cdot v(h) = (\mu_A \times v \times v) ( f \otimes g \otimes h ) , $$
	and we can then get
	$$ \gamma_{\rho} \circ j_{\af} = \rho_{A_{\bt}} \circ j_{\bt} \circ \chi \ \text{on} \ C(X \times \T^2)_{sa} . $$

	We will show that for all $f \in C(X, \Z)$ and $\mu_A, \mu_B$ as given above, we have $\mu_B(f \circ \sigma^{-1}) = \mu_A(f)$. If that is done, noting that the $\C$-linear span of $C(X, \Z)$ is dense in $C(X)_{sa}$, we get $\mu_B(f \circ \sigma^{-1}) = \mu_A(f)$ for all $f \in C(X)$.

	According to our notation, for $g \in C(X)$, we have
%	$$
%	\begin{array}{lll}
	\begin{align*}
	\mu_A( g ) & = (\mu_A \times v \times v) ( g \otimes \id_{\T} \otimes \id_{\T} ) \\
		& = \tau_A( g \otimes \id_{\T} \otimes \id_{\T} )  \\
		& = \ph^*(\tau_B) ( g \otimes \id_{\T} \otimes \id_{\T} )  \\
		& = \tau_B ( \ph( g \otimes \id_{\T} \otimes \id_{\T} ) )  .
% wrong		& = & \mu_B ( \varphi(g) \otimes \id_{\T} \otimes \id_{\T} ) .
% it does not make sense		& = & \varphi^*(\mu_B) (g) .
	\end{align*}
%	\end{array}
%	$$
	
%	Note that we can identify $f$ as an element of $K_0(C(X))$.

	According to digram (\ref{diag 2}) in the proof of Lemma \ref{technical lifting lemma}, we have the commutative diagram
	\begin{equation} \label{diag 3}
		\xymatrix{
			K_0( C(X) ) \ar@{->}[rr]^{\widetilde{\varphi_0}} \ar[d]_{\pi'_{\af}} && K_0( C(X) ) \ar[d]^{\pi'_{\bt}} \\
			K_0( C^*(\Z, X, \af) ) \ar@{->}[rr]^{\varphi_0} && K_0( C^*(\Z, X, \bt) ) \ \ \ ,
		}
	\end{equation}
	where $C^*(\Z, X, \af)$ and $C^*(\Z, X, \bt)$ are the crossed product \ca s of dynamical systems $(X, \af)$ and $(X, \bt)$ respectively, $\widetilde{\varphi_0}$, $\varphi_0$ are order preserving isomorphisms, and $\widetilde{\varphi_0}$ agrees with $\chi$ as a map from $C(X, \Z)$ to $C(X, \Z)$.
	
	By the proof of Lemma \ref{technical lifting lemma}, for all $f \in C(X, \Z)$, if we identify $C(X, \Z)$ with $K_0( C(X) )$, we get
	$$ \widetilde{\varphi_0}(f) = f \circ \sigma^{-1} . $$
	
	From the commutative diagram (\ref{diag 3}), we can conclude that (although we cannot claim that $\varphi(f \otimes \id_{\T} \otimes \id_{\T}) = \chi(f) \otimes \id_{\T} \otimes \id_{\T}$)
	$$ \tau_B ( \varphi(f \otimes \id_{\T} \otimes \id_{\T} ) ) = \tau_B ( \chi(f) \otimes \id_{\T} \otimes \id_{\T} ) . $$
	As $\chi(f) = f \circ \sigma^{-1}$, it follows that
%	$$
%	\begin{array}{lll}
	\begin{align*}
	\mu_A( f ) & = (\mu_A \times v \times v) ( f \otimes \id_{\T} \otimes \id_{\T} ) \\
		& = \tau_A( f \otimes \id_{\T} \otimes \id_{\T} )  \\
		& = \ph^*(\tau_B) ( f \otimes \id_{\T} \otimes \id_{\T} )  \\
		& = \tau_B ( \ph( f \otimes \id_{\T} \otimes \id_{\T} ) )  \\
		& = \tau_B ( \chi(f) \otimes \id_{\T} \otimes \id_{\T} ) \\
		& = \mu_B( \chi(f) ) \\
		& = \mu_B( f \circ \sigma^{-1} ) .
	\end{align*}
%	\end{array}
%	$$
	
	Now we have that $\mu_A( f ) = \mu_B( f \circ \sigma^{-1})$ for all $f \in C(X, \Z)$. Note that the $\C$-linear span of $C(X, \Z)$ is dense in $C(X)$, we get
	$$ \mu_A( f ) = \mu_B( f \circ \sigma^{-1}) \ \text{for all f} \ \in C(X)_{sa} . $$
%	$$ \mu_B( f \circ \sigma^{-1} ) = \mu_B ( \chi(f) ) = \mu_B (\varphi(f)) = \varphi^*(\mu_B) (f) = \mu_A(f) . $$

%	The commutativity of the diagram (\ref{diag 3}) above implies that $\pi'_{\bt} ( (f \circ \sigma^{-1}) ) = \varphi_0 ( \pi'_{\af} (f) )$.
	As both dynamical systems $\af \times \mathrm{R}_{\xi_1} \times \mathrm{R}_{\eta_1}$ and $\bt \times \mathrm{R}_{\xi_2} \times \mathrm{R}_{\eta_2}$ are rigid, by Proposition \ref{rigidity implies tracial rank zero for rotation case}, we have $\mathrm{TR}(A) = \mathrm{TR}(B) = 0$. According to Corollary \ref{C strongly flip conjugate for TR zero case}, these two dynamical systems $(X \times \T \times \T, \af \times \mathrm{R}_{\xi_1} \times \mathrm{R}_{\eta_1})$ and $(X \times \T \times \T, \bt \times \mathrm{R}_{\xi_2} \times \mathrm{R}_{\eta_2})$ are $C\sp*$-strongly approximately conjugate.
\end{proof}

\vspace{1cm}

	We studied the weakly approximate conjugacy between to dynamical systems $\af \times \mathrm{R}_{\xi_1} \times \mathrm{R}_{\eta_1}$ and $\bt \times \mathrm{R}_{\xi_2} \times \mathrm{R}_{\eta_2}$ and give an if and only if condition for the weakly approximate conjugacy.

	For minimal homeomorphisms $\af \times \mathrm{R}_{\xi_1} \times \mathrm{R}_{\eta_1}$ and $\bt \times \mathrm{R}_{\xi_2} \times \mathrm{R}_{\eta_2}$, the following lemma shows that whether they are weakly approximately conjugate or not is determined by $\af$ and $\bt$ only, and has nothing to do with $\mathrm{R}_{\xi_i}$ and $\mathrm{R}_{\eta_i}$ for $i = 1, 2$.

\begin{lem} \label{iff condition for weak approximate conjugacy of rotation case}

	Let $(X, \af)$ and $(X, \bt)$ be two minimal Cantor dynamical systems. For continuous maps $\xi_1, \xi_2, \eta_1, \eta_2 \colon X \rightarrow \T$, $(X \times \T \times \T, \af \times \mathrm{R}_{\xi_1} \times \mathrm{R}_{\eta_1})$ and $(X \times \T \times \T, \bt \times \mathrm{R}_{\xi_2} \times \mathrm{R}_{\eta_2})$ are weakly approximately conjugate if and only if $(X, \af)$ and $(X, \bt)$ are weakly approximately conjugate.

\end{lem}

\begin{proof}
	The ``if" part:

	For every $\ep > 0$, we will show that there exists $\sigma_n \in \text{Homeo}(X \times \T \times \T)$ such that
	$$ \dist(\sigma_n \circ \af \circ \sigma_n^{-1} , \bt) < \ep . $$

	As $(X, \bt)$ is a minimal Cantor dynamical system, there exists a Kakutani-Rokhlin partition
	$$ \{ X_{s, k} \colon 1 \leq s \leq n, 0 \leq k < h(s) \} $$
	such that $h(s) > 5 / \ep$, and $\diam(X_{s, j}) < \ep / 5$, where $\diam(X_{s, j})$ is defined to be $\sup_{x, y \in X_{s, j}} \dist(x, y)$.

	For any two clopen sets $X_{s_1, j_1}$ and $X_{s_2, j_2}$ in the Kakutani-Rokhlin partition, there exists \\ $\delta_{s_1, j_1; s_2, j_2} > 0$ such that if $x, y \in X_{s_1, j_1} \bigsqcup X_{s_2, j_2}$ and $\dist(x, y) < \delta_{s_1, j_1; s_2, j_2}$, then either $x, y \in X_{s_1, j_1}$ or $x, y \in X_{s_2, j_2}$.

	Let $\delta = \min \delta_{s, j; s', j'}$, where $X_{s, j}$ and $X_{s', j'}$ traverse through all pairs of  distinct clopen sets in the Kakutani-Rokhlin partition above.

	As $(X, \af)$ and $(X, \bt)$ are weakly approximately conjugate, there exists $\gamma_n \in \text{Homeo}(X)$
 such that
 	$$\dist(\gamma \circ \af \circ \gamma^{-1} (x), \beta(x)) < \delta . $$
	According to the definition of $\delta$, it follows that for every $X_{s, j}$ in the Kakutani-Rokhlin partition above, we have
	$$ \gamma \circ \af \circ \gamma^{-1} (X_{s, j}) = \bt(X_{s, j}) . $$

	Without loss of generality (replacing $\af$ with $\gamma \circ \af \circ \gamma^{-1}$), we can assume that $\af$ and $\bt$ satisfies
	$$ \af (X_{s, j}) = \bt(X_{s, j}) . $$

	Identify $\T$ with $\R / \Z$, and define $\pi$ by $\pi \colon \R \rightarrow \R / \Z, t \mapsto t + \Z$. For all $x \in X_{s, 0}$, define $h(x) = 0$. For $x \in X_{s, k}$ with $0 < k < h(s)$, define
	$$ f_1(x) = \displaystyle \sum_{j = 1}^k (\xi_2 - \xi_1) (\af^{-j} (x)) . $$
	As $\xi_1$ and $\xi_2$ are both in $C(X, \T)$, it follows that the above defined $f_1$ is a continuous function from $X$ to $\T$.

	For $x \in X_{s, k}$, define
	$$ g_1(x) = \displaystyle\sum_{j = 1}^{h(s)} (\xi_2 - \xi_1) (\af^{-j} ( \af^{h(s) - k}(x) )). $$
	It is also clear that $g_1 \in C(X, \T)$.

	As $X$ is totally disconnected, we can divide $X$ into $\bigsqcup_{k = 1}^N X_k$, with every $X_k$ being a clopen subset of $X$ satisfying $\dist(h(x), h(y)) < \frac{1}{4}$ for $x, y$ in the same $X_k$. For $g_1 \left|_{X_k} \right.$, we can lift it to continuous function $G_{1, k} \colon X_k \rightarrow [0 - \frac{1}{4}, 1 + \frac{1}{4}]$ satisfying $g_1 \left|_{X_k} \right. = \pi \circ G_{1, k}$.

	Define $G_1 \colon X \rightarrow \R$ by setting $G_1(x)$ to be $G_{1, k}(x)$ if $x \in X_k$. It is then easy to check that $G_1$ is a lifting of $g_1$ satisfying
	$$ g_1 = \pi \circ G_1 \ \text{and} \ G_1(x) \in [0 - \frac{1}{4}, 1 + \frac{1}{4}] \ \text{for all x} \ \in X . $$
	
	For $x \in X_{s, k}$, define
	$$ s_1(x) = f_1(x) - \Frac{G_1(x) \cdot k}{h(s)} + \Z . $$
	Similarly, define $f_2(x) = 0$ if $x \in X_{s, 0}$ and
	$$ f_2(x) = \displaystyle \sum_{j = 1}^k (\eta_2 - \eta_1) (\af^{-j} (x)) $$
	for $x \in X_{s, k}$ with $0 < k < h(s)$. Define
	$$ g_2(x) = \displaystyle\sum_{j = 1}^{h(s)} (\eta_2 - \eta_1) \left( \af^{-j} \left( \af^{h(s) - k}(x) \right) \right) . $$
	
	As $X$ is totally disconnected, we can find a lifting $G_2 \in C(X, \R)$ such that
	$$ g_2 = \pi \circ G_2 \ \text{and} \ G_2(x) \in \left[0 - \frac{1}{4}, 1 + \frac{1}{4} \right] $$
	for all $x \in X$.

	For $x \in X_{s, k}$, define
	$$ s_2(x) = f_2(x) - \Frac{G_2(x) \cdot k}{h(s)} + \Z . $$
	For the $s_1$ and $s_2$ we have defined, it is easy to check that they are continuous function from $X$ to $\R / \Z$. According to our identification, we can regard $s_1$ and $s_2$ as functions in $C(X, \T)$.

	We will show that $(\id_X \times \mathrm{R}_{s_1} \times \mathrm{R}_{s_2})$ will approximately conjugate $\af \times \mathrm{R}_{\xi_1} \times \mathrm{R}_{\eta_1}$ and $\bt \times \mathrm{R}_{\xi_2} \times \mathrm{R}_{\eta_2}$.

	For every $(x, t_1, t_2) \in X \times \T \times \T$, we have
	$$
	\begin{array}{l}
		(\id_x \times \mathrm{R}_{s_1} \times \mathrm{R}_{s_2}) \circ (\af \times \mathrm{R}_{\xi_1} \times \mathrm{R}_{\eta_1}) \circ (\id_x \times \mathrm{R}_{s_1} \times \mathrm{R}_{s_2})^{-1} (x, t_1, t_2)  \\
		=  (\id_x \times \mathrm{R}_{s_1} \times \mathrm{R}_{s_2}) \circ (\af \times \mathrm{R}_{\xi_1} \times \mathrm{R}_{\eta_1}) (x, t_1 - s_1(x), t_2 - s_2(x) )  \\
		= (\id_x \times \mathrm{R}_{s_1} \times \mathrm{R}_{s_2}) (\af(x), t_1 - s_1(x) + \xi_1(x), t_2 - s_2(x) + \eta_1(x)) \\
		= ( \af(x), t_1 + \xi_1(x) - s_1(x) + s_1(\af(x)), t_2 + \eta_1(x) - s_2(x) + s_2(\af(x)) ) ,
	\end{array}
	$$
	and it is clear that
	$$ (\bt \times \xi_2 \times \eta_2) (x, t_1, t_2) = (\bt(x), t_1 + \xi_2(x), t_2 + \eta_2(x)) . $$

	As $ \af (X_{s, j}) = \bt(X_{s, j})$ and $\diam(X_{s, j}) < \ep / 5$, we have $\dist(\af(x), \bt(x)) < \ep / 5$ for all $x \in X$.
	Consider the distance between $t_1 + \xi_1(x) - s_1(x) + s_1(\af(x))$ and $t_1 + \xi_2(x)$. We get
	$$ | t_1 + \xi_1(x) - s_1(x) + s_1(\af(x)) - ( t_1 +  \xi_2(x)) | = | s_1(\af(x)) -s_1(x) + \xi_1(x) - \xi_2(x) | . $$
	According to the definition of $s_1$, if $x \in X_{s, h(s)}$ (that is, $x$ is on the roof), then
%	$$
%	\begin{array}{lll}
	\begin{align*}
		s_1(x) & = \displaystyle\sum_{j = 1}^{h(s)} (\xi_2 - \xi_1) \left( \af^{-j} (x) \right) - G_1(x) \vspace{1mm} \\
			& = \displaystyle\sum_{j = 1}^{h(s)} (\xi_2 - \xi_1) \left( \af^{-j} (x) \right) - \displaystyle\sum_{j = 0}^{h(s)} (\xi_2 - \xi_1) ( \af^{-j} (x) ) \\
			& = - (\xi_2 - \xi_1) (x) \\
			& = 0 .
	\end{align*}
%	\end{array}
%	$$
	We know that $s_1 \left( \af(x) \right) = 0$ as $(\af^{- h(s)}) (x) \in X_{s, 0}$. It is then clear that
	$$ | s_1(\af(x)) - s_1(x) + \xi_1(x) - \xi_2(x) | = 0 $$
	if $x$ is in the roof set.

	If $x$ is not in the roof, in other words, for $x \in X_{s, k}$ with $0 \leq k < h(s) - 1$, we have
%	$$
%	\begin{array}{lll}
	\begin{align*}
		s_1(\af(x)) - s_1(x) & = (\xi_2 - \xi_1) (x) -  \Frac{G_1(x)}{h(s)} .
	\end{align*}
%	\end{array}
%	$$
	As $G_1(x) \in [0 - \frac{1}{4}, 1 + \frac{1}{4}]$ for all $x$, and we have $h(s) > 5 / \ep$ for all $s$, it then follows that
	$$ | s_1(\af(x)) - s_1(x) + \xi_1(x) - \xi_2(x) | < 2 \ep / 5 \ \text{for all} \ x \in X . $$
	Similarly, we have
	$$ | t_2 + \eta_1(x) - s_2(x) + s_2(\af(x)) - ( t_2 +  \eta_2(x)) | = | s_2(\af(x)) - s_2(x) + \eta_1(x) - \eta_2(x) | $$
	and
	$$ | s_2(\af(x)) - s_2(x) + \eta_1(x) - \eta_2(x) | < 2 \ep / 5 \ \text{for all} \ x \in X . $$

	So far, we have proved that
	$$
	\begin{array}{l}
		\dist \left( (\id_x \times \mathrm{R}_{s_1} \times \mathrm{R}_{s_2}) \circ (\af \times \mathrm{R}_{\xi_1} \times \mathrm{R}_{\eta_1}) \circ (\id_x \times \mathrm{R}_{s_1} \times \mathrm{R}_{s_2})^{-1}, \bt \times \mathrm{R}_{\xi_2} \times \mathrm{R}_{\eta_2} \right)  \\
		 \hspace{3cm} < \ep / 5 + 2 \ep / 5 + 2 \ep / 5 \\
		 \hspace{3cm} = \ep .
	\end{array}
	$$
	As we can construct such conjugacy maps for all $\ep > 0$, it follows that $\af \times \mathrm{R}_{\xi_1} \times \mathrm{R}_{\eta_1}$ is weakly approximately conjugate to $\bt \times \mathrm{R}_{\xi_2} \times \mathrm{R}_{\eta_2}$ if $\af$ is weakly approximately conjugate to $\bt$.

\vspace{2mm}

	The ``only if" part.

	If  a sequence of $\sigma_n$ in $\text{Homeo}(X \times \T^2)$ approximately conjugates $\af \times \mathrm{R}_{\xi_1} \times \mathrm{R}_{\eta_1}$
	%QQQQ shall the next "to" be "and"?
	to $\bt \times \mathrm{R}_{\xi_2} \times \mathrm{R}_{\eta_2}$, as $X$ is totally disconnected, we can write $\sigma_n$ as $\gamma_n \times \varphi$, with $\gamma_n \in \text{Homeo}(X)$ and $\varphi \colon X \rightarrow \text{Homeo} (\T^2)$ being a continuous map.

	Let $P \colon X \times \T^2 \rightarrow X$ be defined by $P(x, (t_1, t_2)) = x$ (the canonical projection onto $X$).  We can easily check that
	$$ P ( (\sigma_n \circ (\af \times \mathrm{R}_{\xi_1} \times \mathrm{R}_{\eta_1}) \circ \sigma_n^{-1}) ( x, (t_1, t_2) ) ) = (\gamma_n \circ \af \circ \gamma_n^{-1}) (x) . $$
	As $(\sigma_n \circ (\af \times \mathrm{R}_{\xi_1} \times \mathrm{R}_{\eta_1}) \circ \sigma_n^{-1}) \longrightarrow \bt \times \mathrm{R}_{\xi_2} \times \mathrm{R}_{\eta_2}$, we have
	$$ P ( (\sigma_n \circ (\af \times \mathrm{R}_{\xi_1} \times \mathrm{R}_{\eta_1}) \circ \sigma_n^{-1}) ( x, (t_1, t_2) ) ) \longrightarrow P( (\bt \times \mathrm{R}_{\xi_2} \times \mathrm{R}_{\eta_2}) ( x, (t_1, t_2) )  ) , $$
	which then implies that
	$$ (\gamma_n \circ \af \circ \gamma_n^{-1}) (x) \longrightarrow \bt(x) \ \text{for all} \ x \in X . $$
\end{proof}

%	The idea of approximate $K$-conjugacy is propsed by Huaxin Lin in a series of papers, see \cite{Lin6}, \cite{LM1}, \cite{LM2}, \cite{LM3}.

\begin{sloppypar}
	From Lemma \ref{iff condition for weak approximate conjugacy of rotation case}, we know that the if and only if condition for $\af \times \mathrm{R}_{\xi_1} \times \mathrm{R}_{\eta_1}$ and $\bt \times \mathrm{R}_{\xi_2} \times \mathrm{R}_{\eta_2}$ to be weakly approximately conjugate is that $\af$ and $\bt$ are weakly approximately conjugate.

	One might be wondering whether we have weak approximate conjugacy between $\af \times \mathrm{R}_{\xi_1} \times \mathrm{R}_{\eta_1}$ and $\bt \times \mathrm{R}_{\xi_2} \times \mathrm{R}_{\eta_2}$, can we expect to have the isomorphism between \ca s \ $C^*(\Z, X \times \T \times \T, \af \times \mathrm{R}_{\xi_1} \times \mathrm{R}_{\eta_1})$ and $C^*(\Z, X \times \T \times \T, \bt \times \mathrm{R}_{\xi_2} \times \mathrm{R}_{\eta_2})$?

	Generally speaking, weak approximate conjugacy is not enough to imply that the corresponding crossed product \ca s are isomorphic. Examples can be found in \cite{M1}, \cite{LM1} and \cite{LM3}. As guessed by Lin in \cite{LM1}, if we strengthen the definition of weak approximate conjugacy (in the sense that those conjugacies will induce an isomorphism of K-data of these two crossed product \ca s), this might be equivalent to the isomorphism of two crossed product \ca s.

	That ``strengthened" version of weak approximate conjugacy is called approximate K-conjugacy (see \cite{Lin6}, \cite{LM1}, \cite{LM2}, \cite{LM3}). Before the definition of approximate K-conjugacy is given, the definition of asymptotic morphism will be given and a technical result needs to be mentioned.

\end{sloppypar}

\vspace{2mm}

\begin{dfn}

	A sequence of contractive completely positive linear maps $\{ \varphi_n \}$ from \ca \ $A$ to \ca \ $B$ is said to be an asymptotic morphism, if
	$$ \lim_{n \rightarrow \infty} \norm{\varphi_n(a b) - \varphi_n(a) \varphi_n(b)} = 0 \ \text{for all} \  a, b \in A . $$

\end{dfn}

\begin{prp} \label{weak approximate conjugacy implies asymptotic morphisms} \cite{Lin6}

	Let $(X, \af)$ and $(X, \bt)$ be two dynamical systems. If there exists a sequence of homeomorphisms $\sigma_n \colon X \rightarrow X$ such that $\lim_{n \rightarrow \infty} \dist(\sigma_n \circ \af \circ \sigma_n^{-1}, \bt) = 0$, then for a sequence of unitaries $\{ z_n \}$ in $A_{\af}$ with
	$$ \lim_{n \rightarrow \infty} \norm{ z_n j_{\af}(f) - j_{\af} (f) z_n } = 0 \ \text{for all} \ f \in C(X) , $$
	there exists a unital asymptotic morphism $\{ \varphi_n^{\sigma} \}$ from $A_{\bt}$ to $A_{\af}$ such that
	$$ \lim_{n \rightarrow \infty} \norm{\psi_n^{\sigma}(u_{\bt}) - u_{\af} z_n} = 0 \ \text{and} $$
	$$ \lim_{n \rightarrow \infty} \norm{\psi_n^{\sigma}(j_{\bt}(f)) - j_{\af} (f \circ \sigma_n)} = 0 $$
	for all $f \in C(X)$.

\end{prp}

\begin{proof}
	This is Proposition 3.1 in \cite{Lin6}. \qedhere

\end{proof}

	Now we can give the definition of approximate K-conjugacy between two dynamical systems $(X, \af)$ and $(X, \bt)$.

\begin{dfn} \label{definition of approximate K-conjugacy}

	For two minimal dynamical systems $(X, \af)$ and $(Y, \bt)$, with $X$ and $Y$ being compact metrizable spaces, we say that $(X, \af)$ and $(Y, \bt)$ are approximately K-conjugate if there exist homeomorphisms $\sigma_n \colon X \rightarrow Y$, $\tau_n \colon Y \rightarrow X$, and an isomorphism
	$$ \rho \colon K_*( C^*(\Z, Y, \bt) ) \rightarrow K_*( C^*(\Z, X, \af) ) $$ between $K$-groups such that
	$$ \sigma_n \circ \af \circ \sigma_n^{-1} \rightarrow \bt, \ \ \tau_n \circ \bt \circ \tau_n^{-1} \rightarrow \af , $$
	and the associated discrete asymptotic morphisms $\psi_n \colon B \rightarrow A$ and $\varphi_n \colon A \rightarrow B$ induce the isomorphisms $\rho$ and $\rho^{-1}$ respectively.

\end{dfn}

\vspace{2mm}

	\textbf{Remark:} According to Proposition \ref{weak approximate conjugacy implies asymptotic morphisms}, the weak approximate conjugacy maps will induce asymptotic morphisms. But it is not generally true that the asymptotic morphisms will induce a homomorphism of $K_0$ and $K_1$ data. In Definition \ref{definition of approximate K-conjugacy}, those approximate conjugacies must not only induce a pair of homomorphisms between $K_i(A)$ and $K_i(B)$, in addition, these homomorphisms must be a pair of isomorphisms that are inverses of each other.

	For the classical case of minimal Cantor dynamical systems, it is shown in \cite{LM3} that two Cantor minimal dynamical systems are approximately K-conjugate if and only if the corresponding crossed product \ca s are isomorphic. For the case of $(X \times \T, \af \times \mathrm{R}_{\xi})$, with $\af \in \text{Homeo}(X)$ being minimal homeomorphism and $\xi \colon X \rightarrow \T$ being a continuous map, similar results are obtained in Theorem 7.8 of \cite{LM1}.

	Based on Theorem \ref{C*-strong approximate flip conjugacy} and Lemma \ref{iff condition for weak approximate conjugacy of rotation case}, we will give an if and only if condition for approximate K-conjugacy between $\af \times \mathrm{R}_{\xi_1} \times \mathrm{R}_{\eta_1}$ and $\bt \times \mathrm{R}_{\xi_2} \times \mathrm{R}_{\eta_2}$.

\begin{thm} \label{approximate K-conjugacy for rotation case}

	Let $X$ be the Cantor set. Let $\af, \bt \in \text{Homeo}(X)$ be minimal homeomorphisms, and let $\xi_1, \xi_2, \eta_1, \eta_2 \colon X \rightarrow \T$ be continuous map such that both $\af \times \mathrm{R}_{\xi_1} \times \mathrm{R}_{\eta_1}$ and $\bt \times \mathrm{R}_{\xi_2} \times \mathrm{R}_{\eta_2}$ are minimal rigid homeomorphism of $X \times \T \times \T$ (as in Definition \ref{definition of rigidity}). Use $A$ to denote the crossed product \ca \ corresponding  to the minimal system $(X \times \T \times \T, \af \times \mathrm{R}_{\xi_1} \times \mathrm{R}_{\eta_1})$, and $B$ to denote the one corresponding to $(X \times \T \times \T, \bt \times \mathrm{R}_{\xi_2} \times \mathrm{R}_{\eta_2})$. Use $K^0(X, \af)$ to denote $C(X, \Z) / \{ f - f \circ \af^{-1} \colon f \in C(X, \Z^2) \}$ and $K^0(X, \bt)$ to denote $C(X, \Z) / \{ f - f \circ \bt^{-1} \colon f \in C(X, \Z^2) \}$.

	The following are equivalent:

	1) $(X \times \T \times \T, \af \times \mathrm{R}_{\xi_1} \times \mathrm{R}_{\eta_1})$ and $(X \times \T \times \T, \bt \times \mathrm{R}_{\xi_2} \times \mathrm{R}_{\eta_2})$ are approximately K-conjugate,

	2) There is an order isomorphism $\rho \colon K_0(B) \rightarrow K_0(A)$ that maps $K^0(X, \bt)$ to $K^0(X, \af)$.

\end{thm}

\begin{proof}
	1) $\Rightarrow$ 2) :

	If $(X \times \T \times \T, \af \times \mathrm{R}_{\xi_1} \times \mathrm{R}_{\eta_1})$ and $(X \times \T \times \T, \bt \times \mathrm{R}_{\xi_2} \times \mathrm{R}_{\eta_2})$ are approximately K-conjugate, according to the definition of approximate K-conjugacy (Definition \ref{definition of approximate K-conjugacy}), there exists $\sigma_n \in \text{Homeo}(X \times \T \times \T)$ such that
	$$ \dist(\sigma_n \circ (\af \times \mathrm{R}_{\xi_1} \times \mathrm{R}_{\eta_1}) \circ \sigma_n^{-1}, \bt \times \mathrm{R}_{\xi_2} \times \mathrm{R}_{\eta_2}) \longrightarrow 0 , $$
	and the discrete asymptotic morphism induced by $\{ \sigma_n \colon n \in \N \}$ will yield an isomorphism from $K_*(B)$ to $K_*(A)$.

	That is, there exists an isomorphism
	$$ \phi_0 \colon (K_0(B), K_0(B)_+, [1_B], K_1(B)) \rightarrow (K_0(A), K_0(A)_+, [1_A], K_1(A)) . $$
	Define $\phi$ to be the restriction of $\phi_0$ on $K_0(A)$. We just need to show that $\phi$ maps $K^0(X, \bt)$ to $K^0(X, \af)$.

%	According to Proposition \ref{rigidity implies tracial rank zero for rotation case}, we know that $\mathrm{TR}(A) = \mathrm{TR}(B) = 0$. It then follows that $A$ and $B$ are classifiable via the $K$-data.

%	Note that $\rho$ gives an order preserving isomorphism between $K_0(B)$ and $K_0(A)$, and an isomorphism between $K_1(B)$ and

	\removeme[How to get this? It is not trivial] According to the Pimsner-Voiculescu six-term exact sequence (as in the proof of Proposition \ref{K-data of crossed product C-algebra}), we have
	$$ (j_{\bt})_0 ( C(X \times \T \times \T) ) \cong K^0(X, \bt) = C(X, \Z^2) / \{ f - f \circ \af^{-1} \colon f \in C(X, \Z^2) \} . $$

	As $\af \times \mathrm{R}_{\xi_1} \times \mathrm{R}_{\eta_1}$ and $\bt \times \mathrm{R}_{\xi_2} \times \mathrm{R}_{\eta_2}$ are approximately K-conjugate, for given projection $p \in M_{\infty} (B)$, there exists $N \in \N$ such that for all $m, n > N$, we have $[p \circ \sigma_n] = [p \circ \sigma_m]$ in $K_0(A)$.

	It is obvious that $[p \circ \sigma_n] \in (j_{\af})_* ( C(X \times \T \times \T) )$. Then we can conclude that the isomorphism $\rho$ induced by the conjugacy maps will map $K^0(X, \bt)$ to $K^0(X, \af)$.

\vspace{4mm}

	2) $\Rightarrow$ 1) :

	It is easy to check that 2) implies the following commutative diagram:

	$$
	\xymatrix{
		K_0(B) \ar@{->}[rr]^{\displaystyle\rho} && K_0(A) \\
		K^0(X, \bt) \ar@{->}[rr]_{\displaystyle \rho \left|_{K^0(X, \bt)} \right.} \ar@{->}[u]^{(\displaystyle j_{\bt})_{*0}} &&  K^0(X, \af) \ar@{->}[u]_{\displaystyle (j_{\af})_{*0}} \ .
	}
	$$
	According to Theorem \ref{C*-strong approximate flip conjugacy}, the two minimal homeomophisms $\af \times \mathrm{R}_{\xi_1} \times \mathrm{R}_{\eta_1}$ and $\bt \times \mathrm{R}_{\xi_2} \times \mathrm{R}_{\eta_2}$ are $C^*$-strongly flip conjugate.

	The map $\rho$ above induces an order preserving isomorphism between $K^0(X, \bt)$ (which is isomorphic to $C(X, \Z^2) / \{ f - f \circ \bt^{-1} \}$, with order described as in Lemma \ref{K_i of A_x}) and $K^0(X, \af)$ (which is isomorphic to $C(X, \Z^2) / \{ f - f \circ \af^{-1} \}$, with order described as in Lemma \ref{K_i of A_x}). Note that
	$$ K_0( C^*(\Z, X, \af) ) \cong C(X, \Z) / \{ g - g \circ \af^{-1} \colon g \in C(X, \Z) \} , $$
	with
	$$ K_0( C^*(\Z, X, \af) )_+ \cong C(X, \Z) / \{ g - g \circ \af^{-1} \colon g \in C(X, \Z), g \geq 0 \} . $$
	
	It follows that there is an order isomorphism
	$$
	\begin{array}{l}
	\widetilde{\rho} \colon ( K_0( C^*(\Z, X, \bt) ) , K_0( C^*(\Z, X, \bt) )_+, [1_{C^*(\Z, X, \bt)}] ) \\
	\hspace{1cm} \longrightarrow ( K_0( C^*(\Z, X, \af) ) , K_0( C^*(\Z, X, \af) )_+, [1_{C^*(\Z, X, \af)}] ) .
	\end{array}
	$$
	According to Theorem 5.4 of \cite{LM3}, $(X, \af)$ and $(X, \bt)$ are approximately K-conjugate. Thus they are weakly approximately conjugate.

	For any $\ep > 0$ and any finite subset $\CalF \subset C(X \times \T \times \T)$, as $\bt$ is minimal, we can find Kakutani-Rokhlin partition
	$$ \CalP = \{ X(s, k) \colon s \in S, 1 \leq k \leq H(s) \} $$
	such that $H(s) > \Frac{32 \pi}{\ep}$ for all $s \in S$ and
	%QQQQ \removeme[check the ep /16]
	$\diam(X(s, k)) < \Frac{\ep}{16}$.

	As $C(X \times \T_1 \times \T_2)$ is generated by
	$$ \{ 1_{D}, z_1, z_2 \colon D \ \text{is a clopen subset of } \ X, z_i \ \text{is the identity function on} \  T_i \} , $$
	without loss of generality, we can assume that
	$$ \CalF = \{ 1_{X(s, k)}, z_1 1_{X(s, k)}, z_2 1_{X(s, k)} \colon s \in S, 1 \leq k \leq H(s) \}. $$

	The fact that $(X, \af)$ and $(X, \bt)$ are approximately K-conjugate implies that there exist $\{ \sigma_n \in \text{Homeo}(X) : n \in \N \}$ such that
	$$ \sigma_n \circ \af \circ \sigma_n^{-1} \longrightarrow \bt . $$
	By choosing $n$ large enough, just as in the proof of the ``if" part of Theorem \ref{iff condition for weak approximate conjugacy of rotation case},  we get
	$$ (\sigma_n \circ \af \circ \sigma_n^{-1}) ( X(s, k) ) = \bt( X(s, k) ) \ \text{for} \ s \in S, 1 \leq k \leq H(s) . $$
	Without loss of generality, we can assume that
	$$ \af ( X(s, k) ) = \bt( X(s, k) ) \ \text{for} \ s \in S, 1 \leq k \leq H(s) . $$

	As in the proof of ``if" part of Theorem \ref{iff condition for weak approximate conjugacy of rotation case}, there exist maps $\{ \id_X \times \mathrm{R}_{g_n} \times \mathrm{R}_{h_n} \}_{n \in \N}$ such that
	$$ (\id_X \times \mathrm{R}_{g_n} \times \mathrm{R}_{h_n}) \circ (\af \times \mathrm{R}_{\xi_1} \times \mathrm{R}_{\eta_1}) \circ (\id_X \times \mathrm{R}_{g_n} \times \mathrm{R}_{h_n})^{-1} \longrightarrow (\bt \times \mathrm{R}_{\xi_2} \times \mathrm{R}_{\eta_2}) , $$
	with all the $g_n, h_n \colon X \rightarrow \T$ being continuous functions as defined in the proof of Theorem \ref{iff condition for weak approximate conjugacy of rotation case}.

	We will show that the conjugacy maps $\{ \id_X \times \mathrm{R}_{g_n} \times \mathrm{R}_{h_n} \colon n \in \N \}$ will induce an isomorphism between $K_*(B)$ and $K_*(A)$.

	The idea is like this: We know that these two dynamical systems $\af \times \mathrm{R}_{\xi_1} \times \mathrm{R}_{\eta_1}$ and $\bt \times \mathrm{R}_{\xi_2} \times \mathrm{R}_{\eta_2}$ are $C^*$-strongly flip conjugate. Thus there exists $\psi_n \colon B \rightarrow A$ such that the following diagram approximately commutes:
	$$
	\xymatrix{
	B  \ar@{->}[rr]^{\displaystyle\psi_n} && A \\ \\
	C(X \times \T \times \T) \ar@{->}[uu]^{j_B} \ar@{->}[rr]^{\displaystyle\chi_n} && C( X \times \T \times \T ) \ar@{->}[uu]_{j_A}  .
	}
	$$

	As we had assumed that (without loss of generality) $\af ( X(s, k) ) = \bt( X(s, k) ) \ \text{for} \ s \in S, k = 1, \ldots, H(s)$, the $\chi_n$ in the diagram above satisfies
	$$ \dist(\chi_n(x), x) < \diam(X(s, k)) < \ep / M $$
	for $x \in X(s, k)$. In other words, restricted on $C(X \times \T \times \T)$, $\chi_n$ is close to the identity map.

	Note that $\{ \psi_n \}$ are isomorphisms and $[\psi_n] = [\psi_m]$ in $KL(B, A)$ for $m, n$ large enough. If we can find $W_n \in U(A)$ such that $f \circ \sigma_n$ is close to $W_n^* \psi_n(f) W_n$ in $A$, and $W_n^* \psi_n (u_B) W_n$ is close to $u_A z_n$ in $A$, where $z_n$ is a unitary element that ``almost" commutes with $C(X \times \T \times \T)$, then it follows that the conjugacy maps $\{ \id_X \times \mathrm{R}_{g_n} \times \mathrm{R}_{h_n} \colon n \in \N \}$ will induce an isomorphism between $K_*(B)$ and $K_*(A)$.

	The complete proof is as below:

%	Define $q_s = 1_{X(s, 1)} \circ \af^{-1}$, and define
%	$$ g(x, z_1, z_2) p(s, H(s)) = (\af \times \mathrm{R}_{\xi_1} \times \mathrm{R}_{\eta_1})^{H(s)} \circ (\bt \times \mathrm{R}_{\xi_2} \times \mathrm{R}_{\eta_2})^{- H(s)} ( z \cdot 1_{X(s, H(s))}) . $$

%	As $X$ is totally disconnected, there exists $\omega \colon $

%	Choose $\id_X \times \mathrm{R}_{g_n} \times \mathrm{R}_{h_n}$ such that
%	$$ \dist( (\id_X \times \mathrm{R}_{g_n} \times \mathrm{R}_{h_n}) \circ (\af \times \mathrm{R}_{\xi_1} \times \mathrm{R}_{\eta_1}) \circ (\id_X \times \mathrm{R}_{g_n} \times \mathrm{R}_{h_n})^{-1}, \bt \times \mathrm{R}_{\xi_2} \times \mathrm{R}_{\eta_2}  )  < \ep / M . $$

	Let $g_1, g_2, f_1, f_2$ be as defined in the proof of Lemma \ref{iff condition for weak approximate conjugacy of rotation case}, and let
$$ \CalF_1 = \{ g_i \cdot 1_{X(s, k)}, f_i \cdot 1_{X(s, k)} \colon s \in S, 1 \leq k \leq H(s) \} . $$

	We can further divide $\af^{-1} ( X(s, 1) )$ into the disjoint union of clopen sets $Y(s, 1)$, $Y(s,2)$, $\ldots$, $Y(s, N(s))$, and choose $x_{s, j} \in Y(s, j)$ such that
	$$ | f(x) - f(x_{s, j}) | < \ep / 16 \ \text{for all} \ f \in \CalF_1, 1 \leq j \leq N(s), s \in S . $$
	Let $G_1, G_2$ be the same as the one defined in the proof of Theorem \ref{iff condition for weak approximate conjugacy of rotation case}.  That is, $G_1$ is the lifting of $g_1(x) = \displaystyle\sum_{j = 1}^{h(s)} (\xi_2 - \xi_1) (\af^{-j} ( \af^{h(s) - k}(x) ))$, $G_2$ is the lifting of $g_2(x) = \displaystyle\sum_{j = 1}^{h(s)} (\eta_2 - \eta_1) (\af^{-j} ( \af^{h(s) - k}(x) ))$, and $G_i(x) \in [0 - \frac{1}{4}, 1 + \frac{1}{4}]$. As both $G_1, G_2$ are path connected to the zero function, it is clear that
	$$ [z_i \cdot 1_{Y(s, j)}] = [z_i \cdot e^{- i 2 \pi G_k / H(s)} \cdot 1_{Y(s, j)}] $$
	in $K_1(A)$ for $i = 1, 2$ and $k = 1, 2$.

	Let
	$$ \iota_{s, j} \colon C(1_{Y_{s, j}} \times \T \times \T) \longrightarrow 1_{Y_{s, j}} \cdot A \cdot 1_{Y_{s, j}} $$
	be the inclusion map. Let two homomorphisms
	$$ \Delta_{s, j}, \ \delta_{s, j} \ \colon C(\T^2) \longrightarrow C(1_{Y_{s, j}} \times \T \times \T) $$
	be defined by
	$$ \Delta_{s, j} (f) = \id_{Y(s, j)} \otimes f $$
	and
	$$ \delta_{s, j} (f) ( x, z_1, z_2 ) = \id_{Y_{s, j}}(x) \cdot f(z_1 \cdot e^{i 2 \pi G_1(x_{s, j}) / H(s)}, z_2 \cdot e^{i 2 \pi G_2(x_{s, j}) / H(s)}) . $$
	Consider the maps
	$$ \iota_{s, j} \circ \Delta_{s, j}, \ \iota_{s, j} \circ \delta_{s, j} \ \colon C(\T^2) \longrightarrow 1_{Y_{s, j}} \cdot A \cdot 1_{Y_{s, j}} . $$
	It is clear that these two maps are monomorphisms.

	By Proposition \ref{rigidity implies tracial rank zero for rotation case}, $\mathrm{TR}(A) = 0$, and it follows that $\mathrm{TR}(1_{Y_{s, j}} \cdot A \cdot 1_{Y_{s, j}}) = 0$.

	As $G_1, G_2$ are contractible, we can claim that
	$$ \left[ \iota_{s, j} \circ \Delta_{s, j} \right] = \left[ \iota_{s, j} \circ \delta_{s, j} \right] \ \text{in} \ KL(C(\T^2), 1_{Y_{s, j}} \cdot A \cdot 1_{Y_{s, j}}) . $$

	For every $f \in 1_{Y_{s, j}} \cdot A \cdot 1_{Y_{s, j}}$, and for every tracial state $\tau$ on $1_{Y_{s, j}} \cdot A \cdot 1_{Y_{s, j}}$, consider $\tau( (\iota_{s, j} \circ \Delta_{s, j}) (f) )$ and $\tau( (\iota_{s, j} \circ \delta_{s, j}) (f) )$. By Lemma \ref{identification of cut-down}, we can regard $1_{Y_{s, j}} \cdot A \cdot 1_{Y_{s, j}}$ as the crossed product \ca \ of the induced minimal homeomorphism of $Y_{s, j} \times \T \times \T$. \removeme[I believe it is true, prove it!] As $\af \times \mathrm{R}_{\xi} \times \mathrm{R}_{\eta}$ is rigid, it follows that the traces on $1_{Y_{s, j}} \cdot A \cdot 1_{Y_{s, j}}$ also corresponds to such measures like $\mu \times v$, with $v$ being the Lebesgue measure on the torus.

	Now we have
 	\begin{align*}
 	\tau \left( (\iota_{s, j} \circ \Delta_{s, j}) (f) \right) & = \tau \left( \id_{Y(s, j)} \otimes f \right) \\
		& = \mu(Y(s, j)) \cdot \displaystyle \int_{\T^2} f \left( (z_1, z_2) \right) \, \mathrm{d} v \vspace{2mm} \\
		& = \mu(Y(s, j)) \cdot \displaystyle \int_{\T^2} f \left (z_1 \cdot e^{i 2 \pi G_1(x_{s, j}) / H(s)}, z_2 \cdot e^{i 2 \pi G_2(x_{s, j}) / H(s)} \right) \ \mathrm{d} v \\
		& = \tau \left( (\iota_{s, j} \circ \delta_{s, j}) (f) \right) .
	\end{align*}

 	As $\mathrm{TR}(1_{Y_{s, j}} \cdot A \cdot 1_{Y_{s, j}}) = 0$, $[\iota_{s, j} \circ \Delta_{s, j}] = [\iota_{s, j} \circ \delta_{s, j}]$ and
	$$ \tau( (\iota_{s, j} \circ \Delta_{s, j}) (f) ) = \tau( (\iota_{s, j} \circ \delta_{s, j}) (f) ) $$
	for all $\tau \in T(1_{Y_{s, j}} \cdot A \cdot 1_{Y_{s, j}})$. According to Theorem 3.4 of \cite{Lin5}, the two monomorphisms $\iota_{s, j} \circ \Delta_{s, j}$ and $\iota_{s, j} \circ \delta_{s, j}$ are approximately unitarily equivalent. Thus there exists a unitary element $v_{s, j} \in 1_{Y_{s, j}} \cdot A \cdot 1_{Y_{s, j}}$ such that
$$ \norm{v_{s, j}^* z_i q_{s, j} v_{s, j} - z_i e^{- i 2 \pi G_i(x_{s, j}) / H(s) \cdot 1_{Y_{s, j}}}} < \ep / (16 K) \ \ \text{for all} \ s \in S, 1 \leq k \leq H(s), 1 \leq j \leq N(s) .$$

	Let $v_s = \displaystyle \sum_{j = 1}^{N(s)} v_{s, j}$. As $Y_{s, 1}, Y_{s, 2}, \ldots, Y_{s, N(s)}$ are mutually disjoint, we have
%	$$
%	\begin{array}{lll}
	\begin{align*}
		\norm{(v_s^k)^* z_i f(x) 1_{\af^{-1}( X(s, 1) )} v_s^k - z e^{- 2 \pi k G_i(x) / H(s)} f(x) 1_{\af^{-1}( X(s, 1) )} } & < \ep / 16 + K \ep / (16 K) + \ep / 16 \\
			& <\ep / 4 .
	\end{align*}
%	\end{array}
%	$$
	for all $f \in \CalF_1, s \in S$.

	Let
	$$ \CalF_2 = \CalF \cup \{ 1_{Y_{s, j}} ,  z_i 1_{Y_{s, j}}, z f 1_{\af^{-1}( X(s, 1) )} \colon f \in \CalF_1, s \in S, 1 \leq k \leq H(s)  \} . $$
	
	As $\af \times \mathrm{R}_{\xi} \times \mathrm{R}_{\eta}$ is $C^*$-strongly flip conjugate to $\af \times \mathrm{R}_{\xi} \times \mathrm{R}_{\eta}$, for any $\delta > 0$, and for the $\CalF_2 \subset C(X \times \T \times \T)$, there exists a \ca \ isomorphism $\psi \colon B \rightarrow A$ such that
	$$ \norm{ \psi( j_{\bt}(f) ) - j_{\af}(f) } < \delta \ \text{and} \ \norm{\psi(u_B)^* j_{\af}(f) \psi(u_B) - j_{\af} (f \circ \bt)} < \delta \ \text{for all} \ f \in \CalF_2 . $$

	Note that $1_{X(s, k)}$, for $s \in S$ and $1 \leq k \leq H(s)$, are mutually orthogonal projections and add up to $1_B$, and $\{ 1_{X(s, k)} \colon s \in S, 1 \leq k \leq H(s) \} \subset \CalF_2$. According to the perturbation lemma \cite[Lemma 2.5.7]{Lin2}, by taking $\delta$ to be small enough, the fact that $\norm{ \psi( j_{\bt}(f) ) - j_{\af}(f) } < \delta$ will imply that there exists $v \in U(A)$ such that
	$$ v \approx_{\ep / (16 K^2)} \psi(u_B) $$
	and
	$$ v^* 1_{X(s, k)} v = 1_{X(s, k)} \circ \bt \ \text{and} \ \norm{v^* f v - f \circ \bt} < \ep / (4 K) \ \text{for all} f \in \CalF_2. $$

	Define $W = \displaystyle \sum_{s \in S} \sum_{k = 1}^{H(s)} 1_{X(s, k)} v^{-k} v_s^k u^k$. Then we can check that
%	$$
%	\begin{array}{lll}
	\begin{align*}
	W^* W & = \displaystyle \left( \sum_{s \in S} \sum_{k = 1}^{H(s)} 1_{X(s, k)} v^{-k} v_s^k u^k \right)^* \cdot \sum_{s' \in S} \sum_{k' = 1}^{H(s)} 1_{X(s', k')} v^{-k'} v_{s'}^{k'} u^{k'} \vspace{2mm} \\
		& = \displaystyle \sum_{s \in S} \sum_{k = 1}^{H(s)} \left( u^{-k} v_s^{-k} v^{k} 1_{X(s, k)}  1_{X(s, k)} v^{-k} v_s^k u^k \right) \vspace{2mm} \\
		& = \displaystyle \sum_{s \in S} \sum_{k = 1}^{H(s)}  u^{-k} v_s^{-k} 1_{\af^{-1}( X(s, 1) )} v_s^k u^k  \vspace{2mm} \\
		& = \displaystyle \sum_{s \in S} \sum_{k = 1}^{H(s)}  u^{-k} 1_{\af^{-1}( X(s, 1) )} u^k  \vspace{2mm} \\
		& = \displaystyle \sum_{s \in S} \sum_{k = 1}^{H(s)}  1_{\af^k ( \af^{-1}( X(s, 1) ) )} \vspace{2mm} \\
		& = \displaystyle \sum_{s \in S} \sum_{k = 1}^{H(s)}  1_{X(s, k)} \vspace{2mm} \\
		& = 1_A .
	\end{align*}
%	\end{array}
%	$$

	As $\mathrm{TR}(A) = 0$, we have $\mathrm{tsr}(A) = 1$. Thus $W^* W = 1_A$ implies that $W W^* = 1_A$. So far, it is checked that $W$ is a unitary element in $A$.

	As
	$$ \norm{(v_s^k)^* z_i f(x) 1_{\af^{-1}( X(s, 1) )} v_s^k - z e^{- 2 \pi k G_i(x) / H(s)} f(x) 1_{\af^{-1}( X(s, 1) )} } < \ep / 4 $$
	and
	$$\norm{v^* f v - f \circ \bt} < \ep / (4 K) \ \text{for all} \ f \in \CalF_2 \ \text{and} \ \text{for all} \  f \in \CalF_2 , $$
	 we have
%	$$
 %	\begin{array}{lll}
 	\begin{align*}
 	W^* z_i 1_{X(s, k)} W & = \displaystyle \left( \sum_{s_1 \in S} \sum_{k_1 = 1}^{H(s_1)} 1_{X(s_1, k_1)} v^{-k_1} v_{s_1}^{k_1} u^{k_1} \right)^* z_i 1_{X(s, k)} \left( \displaystyle \sum_{s_2 \in S} \sum_{k = 1}^{H(s_2)} 1_{X(s_2, k_2)} v^{-k_2} v_{s_2}^{k_2} u^{k_2} \right) \vspace{2mm} \\
	& = \displaystyle \left( \sum_{s_1 \in S} \sum_{k_1 = 1}^{H(s_1)}  u^{-k_1} v_{s_1}^{-k_1} v^{k_1} 1_{X(s_1, k_1)} \right) z_i 1_{X(s, k)} \left( \displaystyle \sum_{s_2 \in S} \sum_{k_2 = 1}^{H(s_2)} 1_{X(s_2, k_2)} v^{-k_2} v_{s_2}^{k_2} u^{k_2} \right) \vspace{2mm} \\
	& = \displaystyle u^{-k} v_s^{-k} v^k 1_{X(s, k)} z_i 1_{X(s, k)} 1_{X_{s, k}} v^{-k} v_s^k u^k \\
	& = \displaystyle u^{-k} v_s^{-k} v^k (z_i 1_{X(s, k)}) v^{-k} v_s^k u^k \\
	& \approx_{\ep / (4 K)} \displaystyle u^{-k} v_s^{-k} \left( (z_i 1_{X(s, k)}) \circ \bt^{k} \right) v_s^k u^k \\
	& \approx_{\ep / (4 K) + \ep / 4} \displaystyle (z 1_{X(s, k)}) \circ \sigma ,
	\end{align*}
%	\end{array}
%	$$
	where
	$$
	\begin{array}{lll}
	\sigma(x, t_1, t_2) & = & \\
 	& \left( x, t_1 + \left( \displaystyle \sum_{j = 1}^k \xi_2 \left( \af^{j - 1} ( \bt^{-k}(x) ) \right) - \xi_1 \left( \bt^{-j}(x) \right) \right) - k G_1(x) / H(s), \right. & \\
	& \left. t_2 + \left( \displaystyle \sum_{j = 1}^k \eta_2 \left( \af^{j - 1} ( \bt^{-k}(x) ) \right) - \eta_1 \left( \bt^{-j}(x) \right) \right) - k G_1(x) / H(s) \right) , &
 	\end{array}
	$$
	for $x \in X(s, k)$ with $s \in S$ and $1 \leq k \leq H(s)$.

	Then it follows that
	$$ \norm{W^* z_i 1_{X(s, k)} W - (z_i 1_{X(s, k)}) \circ \sigma} < K (\ep / 4 K) + \ep / 4 < \ep. $$

	Similar to the proof of Theorem \ref{iff condition for weak approximate conjugacy of rotation case}, we have
 	$$ \dist(\sigma \circ (\af \times \mathrm{R}_{\xi_1} \times \mathrm{R}_{\eta_1}) \sigma^{-1}, \bt \times \mathrm{R}_{\xi_2} \times \mathrm{R}_{\eta_2}) < \ep . $$

	Consider the map $\text{ad} W \circ \psi$, we have that
	$$ \norm{(\text{ad} W \circ \psi) (j_{\bt} (f) ) - j_{\af} (f \circ \sigma) } < \ep + \delta . $$

	If $(\text{ad} W \circ \psi)$ maps $u_B$ to $u_A$ or $u_A \cdot y$ such that $\norm{y f - f y} < \ep$ for all $f \in \CalF$, then it follows that the K-map induced by approximate conjugacy map $\sigma$ (restricted to $\CalF$) will coincide with $[\text{ad} W \circ \psi] \in KL(B, A)$.
 	
	In fact, we can check that
 	$$ W^* v^* W z_i 1_{X(s, k)} W^* v W \approx_d{\ep} u_A^* z_i 1_{X(s, k)} u_A , $$
	which then implies that $\norm{y f - f y} < \ep$ if we define $y = u_A^* (W^* v W) \in U(A)$.

 	As
 	$$(\text{ad} W \circ \psi) (u_B) = W \psi(u_B) W \approx_{\ep / (16 K^2)} W^* v W= u_A y , $$
we may claim that the K-map induced by approximate conjugacy map $\sigma$ (restricted to $\CalF$) will coincide with $[\text{ad} W \circ \psi] \in KL(B, A)$.

	As $C(X \times \T \times \T)$ is separable, by taking $\CalF$ to be large enough and $\ep \rightarrow 0$, it follows that the weak approximate conjugacy map $\sigma$ will induce an isomorphism from $K_i(B)$ to $K_i(A)$, which finishes the proof.
\end{proof}

%\newpage

\addcontentsline{toc}{chapter}{BIBLIOGRAPHY}


\begin{thebibliography}{99}

%\singlespacing

\setlength{\itemsep}{+1mm}

% The trick to prevent the bibliography from being double spaced.
%\setlength{\itemsep}{-1mm}

%\linespread{1}

\bibitem[BKR]{BKR}

Blackadar, B; Kumjian, A.; Rordam, M.  ``{\it Approximately central matrix units and the structure of Noncommutative Tori}", K-theory, 6 (1992), pp. 267-284.

\bibitem[CE]{CE}
Choi, M-D; Effros, E. ``{\it The completely positive lifting problem for \ca s}", Ann. of Math. 104 (1976), pp. 585-609.

\bibitem[EGL]{EGL}
George, A. Elliott; Gong, Guihua; Li, Liangqing. ``{\it Approximate divisibility of simple inductive limit \ca s}", Contemporary Mathematics, 228, pp. 87-97.


\bibitem[Furstenberg]{Furstenberg}
Furstenberg, H. ``{\it Strict ergodicity and transformation of the torus}", American Journal of Mathematics, Vol. 83, No. 4 (1961), pp. 573-601.


\bibitem[GPS]{GPS}
Giordano, Thierry; Putnam, Ian F.; Skau, Christian F. ``{\it Topological orbit equivalence and $C^*$-crossed products.}"  J. Reine Angew. Math. 469 (1995), pp. 51-111.


\bibitem[HLX]{HLX}
Hu, Shanwen; Lin, Huaxin; Xue, Yifeng. ``{\it The tracial topological rank of extensions of C*-algebras}", Math. Scand. 94 (2004), pp. 125-147.


\bibitem[HPS]{HPS}
Herman, R. H.; Putnam, I. F.; Skau, C. F. ``{\it Ordered Bratteli diagrams, dimension groups and topological dynamics}", Internat. J. Math. 3 (1992), pp. 827-864.

\bibitem[KatokHasselblatt]{KatokHasselblatt}
Katok, Anatole; Hasselblatt, Boris. ``{\it Introduction to modern theory of dynamical systems (Encyclopedia of Mathematics and its Applications)}". Cambridge University Press, 1996. 822 pp. ISBN: 0521575575

\bibitem[Lin-Phillips]{Lin-Phillips}
Lin, Huaxin; Phillips, N. Christopher. ``{\it Crossed products by minimal homeomorphism}", J. reine angew. Math. 641 (2010), pp. 95Ñ122

\bibitem[Lin1]{Lin1}
Lin, Huaxin. ``{\it Approximate unitary equivalence in simple \ca s of tracial rank one}", http://arxiv.org/abs/0801.2929v3


\bibitem[Lin2]{Lin2}
Lin, Huaxin. ``{\it An introduction to the classification of amenable $C\sp *$-algebras}", World Scientific Publishing Co., Inc., River Edge, NJ, 2001. xii+320 pp. ISBN: 981-02-4680-3


%\bibitem[Lin3]{Lin3}
%Lin, Huaxin. ``{\it Approximate Unitary Equivalence in Simple $C\sp*$-algebras of Tracial Rank One}", http://arxiv.org/abs/0801.2929


%\bibitem[Lin3]{Lin4}
%Lin, Huaxin. ``{\it An introduction to the classification of amenable $C^*$-algebras.}" World Scientific Publishing Co., Inc., River Edge, NJ, 2001. ISBN: 981-02-4680-3


\bibitem[Lin3]{Lin5}
Lin, Huaxin. ``{\it Classification of homomorphisms and dynamical systems.}"  Trans. Amer. Math. Soc.  359  (2007) no. 2, pp. 859-895.


\bibitem[Lin4]{Lin6}
Lin, Huaxin. ``{\it C*-algebras and K-thoery in dynamical systems}", preprint


\bibitem[LM1]{LM1}
Lin, Huaxin. ; Matui, Hiroki. ``{\it Minimal dynamical systems on the product of the Cantor set and the circle}'', Commun. Math. Phys. 257 (2005), pp. 425-471.


\bibitem[LM2]{LM2}
Lin, Huaxin; Matui, Hiroki. ``{\it Minimal dynamical systems on the product of the Cantor set and the circle II}'', Sel. math.,  New ser. 12 (2006), pp. 199-239 .


\bibitem[LM3]{LM3}
Lin, Huaxin; Matui, Hiroki. ``{\it Minimal dynamical systems and approximate conjugacy}"  Math. Ann.  332  (2005),  no. 4, pp. 795-822.


\bibitem[M1]{M1}
Matui, Hiroki. ``{\it Approximate conjugacy and full groups of Cantor minimal systems}",  Publ. Res. Inst. Math. Sci.  41  (2005),  no. 3, pp. 695-722.

\bibitem[Ph1]{Ph1}
%Phillips, Christopher. ``{\it Cancellation and stable rank for direct limits of recursive subhomogeneous algebras}", http://arxiv.org/abs/math/0101157v1
 Phillips, N. Christopher. ``{\it Cancellation and stable rank for direct limits of recursive subhomogeneous algebras}", Trans. Amer. Math. Soc.  359  (2007),  no. 10, pp. 4625-4652 (electronic).

\bibitem[Ph2]{Ph2}
Phillips, N. Christopher.  ``{\it Real rank and property (SP) for direct limits of recursive subhomogeneous algebras}", http://arxiv.org/abs/math/0405265v1

\bibitem[Putnam]{Putnam}
Putnam, Ian F. ``{\it The $C^*$-algebras associated with minimal homeomorphisms of the Cantor set.}"  Pacific J. Math.  136  (1989),  no. 2, pp. 329-353.

\bibitem[PSS]{PSS}
Putnam, Ian F.; Schmidt, Klaus; Skau, Christian. ``{\it $C^*$-algebras associated with Denjoy Homeomorphisms and the Circle.}" J. Operator Theory. 16 (1986), 99-126

\bibitem[Renault]{Renault}
Renault, J. ``{\it A groupoid approach to C*-algebras}", Lecture Notes in Mathematics \textbf{793}, Berline: Springer 1980.

\bibitem[Thomsen]{Thomsen}
Thomsen, Klaus. ``{\it Traces, unitary characters and crossed products by Z}", Publ. RIMS, Kyoto Univ. 31 (1995), pp. 1011-1029.

\bibitem[Tomiyama]{Tomiyama}
Tomiyama, Jun. ``{\it Topological full groups and structure of normalizers in transformation group C*-algebras}", Pacific Journal of Mathematcis, 173 (1996), no. 2, pp. 571-583.


\end{thebibliography}
\end{document}